\documentclass[11pt,reqno]{amsart}
\usepackage[breaklinks]{hyperref}
\usepackage{comment}
\usepackage{amsmath}
\usepackage{amssymb}
\usepackage{amsfonts}
\usepackage{palatino, mathrsfs}
\usepackage{url}
\usepackage{xcolor}
\usepackage{mathtools}

\usepackage{url}
\usepackage{mathrsfs}
\usepackage{amssymb}
\allowdisplaybreaks
\newcommand{\df}{\dfrac}
\newcommand{\tf}{\tfrac}

\renewcommand{\Re}{\operatorname{Re}}
\renewcommand{\Im}{\operatorname{Im}}

 \renewcommand{\a}{\alpha}
\renewcommand{\b}{\beta}
\newcommand{\e}{\epsilon}

\newcommand{\G}{\Gamma}
\renewcommand{\l}{\lambda}

\newcommand{\x}{\xi}

\renewcommand{\(}{\left\(}
\renewcommand{\)}{\right\)}
\renewcommand{\[}{\left\[}
\renewcommand{\]}{\right\]}
\let\dotlessi=\i
\renewcommand{\i}{\infty}

\numberwithin{equation}{section}
 \theoremstyle{plain}
\newtheorem{theorem}{Theorem}[section]
\newtheorem{lemma}[theorem]{Lemma}
\newtheorem{corollary}[theorem]{Corollary}
\newtheorem{proposition}[theorem]{Proposition}

\newtheorem{remark}[]{Remark}

   \makeatletter
\def\proof{\@ifnextchar[{\@oproof}{\@nproof}}
\def\@oproof[#1][#2]{\trivlist\item[\hskip\labelsep\textit{#2 Proof of\
#1.}~]\ignorespaces}
\def\@nproof{\trivlist\item[\hskip\labelsep\textit{Proof.}~]\ignorespaces}

\makeatother

\begin{document}
\title{Vorono\"{I} summation formula for the generalized divisor function $\sigma_{z}^{(k)}(n)$}
\author{Atul Dixit, Bibekananda Maji and Akshaa Vatwani}
\address{Department of Mathematics, Indian Institute of Technology Gandhinagar, Palaj, Gandhinagar 382355, Gujarat, India}\email{adixit@iitgn.ac.in}
\address{Department of Mathematics, Indian Institute of Technology Indore, Simrol, Indore 453552, Madhya Pradesh,  India}\email{bmaji@iiti.ac.in}
\address{Department of Mathematics, Indian Institute of Technology Gandhinagar, Palaj, Gandhinagar 382355, Gujarat, India}\email{akshaa.vatwani@iitgn.ac.in}
\address{Department of Mathematics, Indian Institute of Technology Gandhinagar, Palaj, Gandhinagar 382355, Gujarat, India}\email{shashankashok.c@iitgn.ac.in}
\address{Department of Mathematics, Indian Institute of Technology Gandhinagar, Palaj, Gandhinagar 382355, Gujarat, India}\email{aviral.srivastava@iitgn.ac.in}
\thanks{2020 \textit{Mathematics Subject Classification.} Primary 11M06; Secondary 33E20, 33C10.\\
	\textit{Keywords and phrases.}  Vorono\"{\dotlessi} summation formula, differential equations, generalized divisor function, elementary symmetric polynomials, Lambert series.}

\begin{abstract}
	For a fixed $z\in\mathbb{C}$ and a fixed $k\in\mathbb{N}$, let $\sigma_{z}^{(k)}(n)$ denote the sum of $z$-th powers of those divisors $d$ of $n$ whose $k$-th powers also divide $n$. This arithmetic function is a simultaneous generalization of the well-known divisor function $\sigma_z(n)$ as well as the divisor function $d^{(k)}(n)$ first studied by Wigert. The Dirichlet series of $\sigma_{z}^{(k)}(n)$ does not fall under the purview of Chandrasekharan and Narasimhan's fundamental work on Hecke's functional equation with multiple gamma factors. Nevertheless, as we show here, an explicit and elegant Vorono\"{\dotlessi} summation formula exists for this function. As its corollaries, some transformations of Wigert are generalized. The kernel $H_{z}^{(k)}(x)$ of the associated integral transform is a new generalization of the Bessel kernel. Several properties of this kernel such as its differential equation, asymptotic behavior and its special values are derived. A crucial relation between $H_{z}^{(k)}(x)$ and an associated integral $K_{z}^{(k)}(x)$ is obtained, the proof of which  is deep, and employs the theory of linear differential equations and the properties of Stirling numbers of the second kind and elementary symmetric polynomials. 
	\end{abstract}
\maketitle
\tableofcontents

\section{Introduction}\label{intro}
The summation formulas of Poisson, Vorono\"\dotlessi, Lipschitz as well as the Euler-Maclaurin and Abel-Plana summation formulas have been studied for a long time in view of their enormous applications in the mathematical sciences. Out of these, the Vorono\"{\dotlessi} summation formula is the \emph{pierre angulaire} of number theory because of its use in estimating the summatory functions of certain arithmetic functions and, in particular, due to the instrumental role that it plays in the Gauss circle problem and the Dirichlet divisor problem.

The celebrated result of Vorono\"{\dotlessi} \cite{voronoi} associated with $d(n)$, the number of divisors of $n$, is given by
\begin{align}\label{vsf}
\sideset{}{'}\sum_{n\leq x}d(n)&=x(\log x+(2\gamma-1))+\frac{1}{4}+\sqrt{x}\sum_{n=1}^{\infty}\frac{d(n)}{\sqrt{n}}\left(-Y_{1}(4\pi\sqrt{nx})-\frac{2}{\pi}K_{1}(4\pi\sqrt{nx})\right).
\end{align}
Here, $Y_{\nu}(\xi)$ and $K_{\nu}(\xi)$ denote the Bessel and modified Bessel functions of the second kind of order $\nu\notin\mathbb{Z}$ respectively defined by \cite[p.~64, 78, eq.~(6)]{watson}
\begin{align}
	Y_{\nu}(\xi)&:=\frac{J_{\nu}(\xi)\cos(\pi \nu)-J_{-\nu}(\xi)}{\sin{\pi \nu}},\label{yj}\\
	K_{\nu}(\xi)&:=\frac{\pi}{2}\frac{I_{-\nu}(\xi)-I_{\nu}(\xi)}{\sin{\pi \nu}}\label{kbes},
\end{align}
where
\begin{align}\label{sumbesselj}
	J_{\nu}(\xi)&:=\sum_{m=0}^{\infty}\frac{(-1)^m(\xi/2)^{2m+\nu}}{m!\Gamma(m+1+\nu)}, \quad |\xi|<\infty,\\
		I_{\nu}(\xi)&:=
		\begin{cases}
			e^{-\frac{1}{2}\pi \nu i}J_{\nu}(e^{\frac{1}{2}\pi i}\xi), & \text{if $-\pi<$ arg $\xi\leq\frac{\pi}{2}$,}\nonumber\\
			e^{\frac{3}{2}\pi \nu i}J_{\nu}(e^{-\frac{3}{2}\pi i}\xi), & \text{if $\frac{\pi}{2}<$ arg $\xi\leq \pi$}, 
		\end{cases}
\end{align}
are the Bessel and modified Bessel functions of the first kind respectively \cite[p.~40, 77]{watson}. When the order of the Bessel functions in \eqref{yj} and \eqref{kbes} is an integer, say, $n$, then we define $Y_n(\xi)=\lim_{\nu\to n}Y_{\nu}(\xi)$ and $K_n(\xi)=\lim_{\nu\to n}K_{\nu}(\xi)$.

 The infinite series in \eqref{vsf} is the error term $\Delta(x)$ in the Dirichlet divisor problem, that is,
\begin{equation} \label{error as inf sum}
\Delta(x)=\sqrt{x}\sum_{n=1}^{\infty}\frac{d(n)}{\sqrt{n}}\left(-Y_{1}(4\pi\sqrt{nx})-\frac{2}{\pi}K_{1}(4\pi\sqrt{nx})\right).
\end{equation}
The general form of the Vorono\"{\dotlessi} summation formula involving a test function $f$,  given by Vorono\"{\dotlessi} \cite{voronoi}, is
\begin{align}\label{kvsf}
\sum_{\alpha<n<\beta}\!\!\! d(n)f(n)\!=\! \int_{\alpha}^{\beta}\!\!(2\gamma+\log t)f(t) dt +2\pi\sum_{n=1}^{\infty}d(n)\int_{\alpha}^{\beta}f(t)\left(\dfrac{2}{\pi}K_{0}(4\pi\sqrt{nt})-Y_{0}(4\pi\sqrt{nt})\right)dt,
\end{align}
where $f(t)$ is a function of bounded variation in $(\alpha,\beta)$ and $0<\alpha<\beta$. Mathematicians have worked with several different versions of the Vorono\"{\dotlessi} summation formula differing in the conditions put forth on $f$ and catered to a particular problem they have been interested in. See, for example, Dixon and Ferrar \cite{dixfer1}, Koshliakov \cite{koshliakov}, Nasim \cite{nasim}, Wilton \cite{wilton}, Yakubovich \cite{yakubovich1} etc. Vorono\"{\dotlessi} also claimed a formula corresponding to \eqref{vsf} for $r_2(n)$, the number of representations of $n$ as sum of two squares. This was subsequently proved by Hardy \cite{hardyqjpam1915} and Sierpi\'{n}ski \cite{sierpinski}.

Today Vorono\"{\dotlessi} summation formulas are known to exist for coefficients of various $L$-functions, for example, the $L$-functions associated with modular forms, Maass forms, and more recently, with automorphic forms as well. See, for example, \cite{corbett}, \cite{guinand}, \cite{millerschmid0.5},  \cite{millerschmid1}, \cite{yakubovich0}. The reader is encouraged to read the excellent survey \cite{millerschmid} on the Vorono\"{\dotlessi} (also, Poisson) summation formulas. As mentioned in this survey, summation formulas can be used to obtain functional equations for various $L$-functions, and, likewise, the properties of the $L$-functions, in turn, can be used to derive the summation formulas. See also \cite[Section 5]{bdgz} for a recent survey.
Let us consider the following setup due to Chandrasekharan and Narasimhan \cite{cn1}. 

Let $a(n)$ and $b(n)$, $1\leq n<\infty$, be two sequences of complex numbers which are not identically 0.  Let
\begin{equation}\label{18}
	\varphi(s):=\sum_{n=1}^{\infty}\df{a(n)}{\lambda_n^s}, \quad \textup{Re}(s)>\sigma_a; \qquad
	\psi(s):=\sum_{n=1}^{\infty}\df{b(n)}{\mu_n^s}, \quad \textup{Re}(s)>\sigma_a^*,
\end{equation}
where  $\{\l_n\}$ and $\{\mu_n\}$ are two sequences of positive numbers, each tending to $\infty$, and $\sigma_a$ and $\sigma_a^*$ are, respectively,  the abscissae of absolute convergence for $\varphi(s)$ and $\psi(s)$. Suppose that $\varphi(s)$ and $\psi(s)$ have analytic continuations into the entire complex plane $\mathbb{C}$ and are analytic on $\mathbb{C}$ except for a finite set of poles, which we denote by $\bf{S}$.
For some $\delta>0$, suppose that $\varphi(s)$ and $\psi(s)$  satisfy a functional equation of the form
\begin{equation}\label{19}
	\chi(s):=(2\pi)^{-s}\Gamma(s)\varphi(s)=(2\pi)^{s-\delta}\Gamma(\delta-s)\psi(\delta-s).
\end{equation}
Chandrasekharan and Narasimhan \cite[p.~6, Lemmas 4, 5]{cn1} showed that the functional equation \eqref{19} is not only equivalent to the `modular' relation
\begin{equation}\label{modular}
	\sum_{n=1}^{\infty}a(n)e^{-\lambda_n x}=\left(\df{2\pi}{x}\right)^{\delta}\sum_{n=1}^{\infty}b(n)e^{-4\pi^2\mu_n/x}+P(x), \qquad \Re(x)>0,
\end{equation}
where
\begin{equation*}
	P(x):=\frac{1}{2\pi i}\int_{\mathcal{C}}(2\pi)^z\chi(z)x^{-z}dz,
\end{equation*}
where $\mathcal{C}$ is a curve, or a set of curves, encircling all of $\bf{S}$, but also to the Riesz-sum identity
\begin{gather}
	\df{1}{\Gamma(\rho+1)}{\sum_{\lambda_n\leq x}}^{\prime}a(n)(x-\lambda_n)^{\rho}=
	\left(\df{1}{2\pi}\right)^{\rho}
	\sum_{n=1}^{\infty}b(n)\left(\df{x}{\mu_n}\right)^{(\delta+\rho)/2}J_{\delta+\rho}(4\pi\sqrt{\mu_n x})+Q_{\rho}(x),\label{21}
\end{gather}
where $x>0$, $\rho>2\sigma_a^*-\delta-\frac12$, $J_{\nu}(z)$  is defined in \eqref{sumbesselj}, and
\begin{equation*}
	Q_{\rho}(x):=\df{1}{2\pi i}\int_{\mathcal{C}}\df{\chi(z)(2\pi)^zx^{z+\rho}}{\Gamma(\rho+1+z)}dz,
\end{equation*}
with $\mathcal{C}$ defined as before. The prime $\prime$ on the sum in the left-hand side indicates that if $\rho=0$ and $x\in\{\lambda_n\}$, then only $\tf12 a(n)$ is to be taken into account. The restriction $\rho>2\sigma_a^*-\delta-\frac12$ can be replaced by $\rho>2\sigma_a^*-\delta-\frac32$, subject to certain conditions, as enunciated in \cite[p.~14, Theorem III]{cn1}. The `modular' relation in \eqref{modular} is due to Bochner \cite{bochner}.
Later, Chandrasekharan and Narasimhan \cite{cn2} considered the  more general functional equation 
\begin{equation}\label{genDelta}
	\Delta(s)\varphi(s)=\Delta(\delta-s)\psi(\delta-s),
\end{equation}
where $\delta>0$,   $\Delta(s):=\prod_{j=1}^{N}\G\left(\alpha_js+\beta_j\right)$, with $N\geq1$, $\beta_j\in\mathbb{C}$, and $\a_j>0$ with $\sum_{j=1}^{N}\alpha_j\geq1$.


As described in \cite[p.~3800]{popov}, most versions of the Vorono\"{\dotlessi} summation formula for an arithmetic function $a(n)$, and associated with a test function $f$ require \eqref{21} to first hold for $\rho=0$. This has to be established separately as one cannot put $\rho=0$ in \eqref{21}. This is inherent in the nature of the proofs of these versions. The only exception to this that we know of is the method of Koshliakov (see \cite[p.~3800]{popov} for more details) although it requires $f$ to be analytic. 

But this suggests an important thing. If an arithmetic function $a(n)$ falls into the purview of the aforementioned setting of Chandrasekharan and Narasimhan, then the Vorono\"{\dotlessi} summation formula for it, and involving a test function $f$ (not necessarily analytic), would automatically hold, provided \eqref{21} holds for $\rho=0$.

In this paper, among other things, we obtain the Vorono\"{\dotlessi} summation formula associated with a generalized divisor function that \emph{does not} fall into the purview of the setting of Chandrasekharan and Narasimhan from \cite{cn2} except in two special cases. This arithmetic function is defined for $k \in \mathbb{N}, z \in \mathbb{C}$ by
\begin{equation}\label{sigmakz}
	\sigma^{(k)}_{z}(n):= \sum_{d^k |n} d^z. 
\end{equation}
It is easily seen that the Dirichlet series associated to $\sigma^{(k)}_{z}(n)$ is $\zeta(s) \zeta(k s-z)$, that is,
\begin{equation}\label{gen_sigma_z}
\sum_{n=1}^\infty \frac{\sigma^{(k)}_{z}(n)}{n^s}=\zeta(s) \zeta(k s-z) \qquad\left(\Re(s) > \max\left\{ 1, \frac{1+\Re(z)}{k}\right\}\right).
\end{equation}
The form of the Dirichlet series implies that the setting of Chandrasekharan and Narasimhan will not be applicable here unless $k\in\mathbb{N}$ and $z=\frac{k-1}{2}$ or unless $k=1$ and $z>-1$. This is explained in detail at the beginning of the next section. In our Vorono\"{\dotlessi} summation formula for $\sigma^{(k)}_{z}(n)$, we will also encounter another divisor function,  defined by 
\begin{equation}\label{skzn}
S_{z}^{(k)}(n):= \sum_{d_{1}^k d_2=n} d_{2}^{\frac{1+z}{k}-1}.
\end{equation}
 One can easily show that 
\begin{equation}\label{divisor_function_S_z}
	\sum_{n=1}^\infty \frac{S_{z}^{(k)}(n)}{n^s}=\zeta(ks)\zeta\left(s+1-\frac{1+z}{k}  \right) \qquad\left(\Re(s) > \max\left\{ \frac{1}{k},  \frac{1+\Re(z)}{k} \right\}\right).
\end{equation}
Observe that $\sigma_{z}^{(1)}(n)=S_{z}^{(1)}(n)=\sigma_z(n)$.

The earliest mention of the function $\sigma^{(k)}_{z}(n)$, defined slightly differently, occurs in a paper of Crum \cite{crum} although he obtains just the Dirichlet series representation \eqref{gen_sigma_z} in his work. Later, Berndt, Roy, Zaharescu and the first author \cite[Section 10.2]{bdrz1} briefly studied this function. Robles and Roy \cite{roblesroy} obtained asymptotic estimates for the summatory function of $\sigma_z^{(k)}(n)$. The special case $z=0$ of $\sigma^{(k)}_{z}(n)$ was studied in detail by Wigert \cite{wig} as early as in 1925.

Note that
	\begin{equation}\label{seriesram}
	\sum_{n=1}^{\infty}\sigma_z^{(k)}(n)e^{-ny}=\sum_{n=1}^{\infty}\frac{n^z}{e^{n^{k}y}-1}.
\end{equation}
On page $332$ of the Lost Notebook, Ramanujan considered the above series for $k\in\mathbb{N}$ and any even integer $z-k$. Although he did not give any transformation for this series for general values of $z$ and $k$ mentioned above, he did give it for $z=0$ and $k=2$ \cite[p.~332]{lnb}, which certainly shows that he considered studying these series an important task. Recently, the first and the second authors \cite[Theorem 1.2]{dixitmaji1} generalized Ramanujan's famous formula for $\zeta(2m+1)$ by obtaining a transformation for the series in \eqref{seriesram} for $z=-2m-1, m\in\mathbb{Z}$.

Various number-theoretic constructs are also intimately connected with the function $\sigma_z^{(k)}(n)$. 
An explicit appearance of the function $\sigma_{z}^{(k)}(n)$ occurs  in a result of Cohen \cite{cohen}. To state it,  consider the generalization of the Ramanujan sum defined by him \cite{cohen} to be
\begin{equation*}
	c_{\ell, k}(n):=\sum_{b=1\atop{(b,\ell^k)_k=1}}^{\ell^k} \exp{\left(\frac{2\pi ibn}{\ell^{k}}\right)},
\end{equation*} 
where the condition $(b,\ell^k)_k=1$ means there is no prime $p$ such that $p|\ell$ and $p^{k}|b$. Then the Dirichlet series of $c_{\ell, k}(n)$ satisfies \cite[Theorem 4]{cohen} (see also \cite[p.~163]{mccarthy})
\begin{equation*}
	\zeta(s)\sum_{\ell=1}^{\infty}\frac{c_{\ell, k}(n)}{\ell^s}=\sigma_{1-\frac{s}{k}}^{(k)}(n),
\end{equation*}
for $s>k$. The case $k=1$ of this identity was given by Ramanujan \cite{ram1918}.

Moreover, let $p_{k}(n)$ denote the number of power partitions of an integer $n$, that is, the number of partitions  of $n$ into parts which are $k$-th powers.  These partitions were studied by Hardy and Ramanujan in their famous work \cite{hardyramanujan}. In the new proof of the asymptotic expansion of $p_k(n)$ as $n\to\infty$ using the saddle-point method given by Tenenbaum, Wu and Li \cite[Equation (2.4)]{tenenbaum}, the series $\sum\limits_{n=1}^{\infty}\sigma_{k}^{(k)}(n)e^{-ny}$ makes its appearance. 

For $\textup{Im}(z)>0$, the Dedekind eta function is defined by $\eta(z):=e^{\frac{i\pi z}{12}}\prod\limits_{n=1}^{\infty}(1-e^{2\pi inz})$. It satisfies the modular transformation 
\begin{equation}\label{dedekind}
\eta(-1/z)=\sqrt{-iz}\eta(z).
\end{equation} 
In his recent study in the context of power partitions, Zagier \cite{zagierhrj} generalized this property. Consider the generalized eta-function $\eta_s(z)$ defined by
\begin{equation*}
	\eta_s(z):=\exp{\left(-\pi i\zeta(-s)z\right)}\prod_{n=1}^{\infty}(1-\exp{(2\pi in^{s}z)})\hspace{4mm}(z\in\mathbb{H}, s\in\mathbb{R}^{+}).
\end{equation*}
Then for $k\in\mathbb{N}$, he proved \cite[Equation (6)]{zagierhrj}
\begin{equation}\label{Zagier}
	\eta_k(-1/z)=(2\pi)^{(k-1)/2}\sqrt{-iz}\prod_{w\in\mathbb{H}\atop{w^k=\pm z}}\eta_{1/k}(w).
\end{equation}
Clearly, \eqref{Zagier} reduces to \eqref{dedekind} for $k=1$. In \cite{bdg_N}, the authors show the equivalence of \eqref{Zagier} with one of the corollaries of their general result by starting with
\begin{equation*}
	\log \eta_k(iy):=\pi \zeta(-k)y-\sum_{n=1}^{\infty}\sigma_k^{(k)}(n){e^{-2\pi n y}\over n}
\end{equation*}
for Re$(y)>0$. Thus, we see that the function $\sigma_k^{(k)}(n)$ is intimately connected with power partitions. We note here that the transformation for $\eta_s(z)$ was first obtained in an equivalent form by Ramanujan \cite[p.~330]{lnb}, and was then rediscovered by Wright \cite{wright}. Kr\"{a}tzel \cite{kratzel1} further generalized Ramanujan's result.

The extended higher Herglotz function recently studied in \cite{hhf1} has an integral representation with the integrand consisting of the sum $\sum_{n=1}^{\infty}\sigma_{-k}^{(N)}(n)e^{-2\pi nt}$; see \cite[Equation (2.8)]{hhf1}. Cohen and Ramanujan-type identities associated to $\sigma_z^{k}(n)$ and $K_{\nu}(\xi)$ were recently obtained in \cite{banerjeemaji}. This shows frequent appearance of $\sigma_z^{(k)}(n)$ in various topics in number theory.

As mentioned before, Wigert worked with the special case $z=0$ of $\sigma_{z}^{(k)}(n)$, which he denoted by $d^{(k)}(n)$. For the infinite series
$
	L_k(w):=\sum_{n=1}^{\infty}d^{(k)}(n)e^{-nw}, \quad(\mathrm{Re}(w)>0),
$
he obtained the following result \cite[p.~8-10]{wig}.
\begin{theorem}\label{Wigert_Lkw}
	For an even $k>1$, we have\footnote{Wigert simplifies this result in the special case $k=2$ in the footnote on p.~9 of \cite{wig}. However, this result was already known to Ramanujan. See \cite[p.~332]{lnb}. }
	\begin{align}\label{wigertkeven}
		L_k(w)&=\frac{\zeta(k)}{w}+w^{-\frac{1}{k}}\G\left(1+\frac{1}{k}\right)\zeta\left(\frac{1}{k}\right)+\frac{1}{4}+\frac{(-1)^{\frac{k}{2}-1}}{k}\left(\frac{2\pi}{w}\right)^{\frac{1}{k}} \times \qquad \quad  \nonumber
  		\\
		\sum_{j=0}^{\frac{k}{2}-1} \bigg\{ &e^{\frac{i\pi(2j+1)(k-1)}{2k}}\overline{L}_{k}\bigg(\!2\pi\left(\frac{2\pi}{w}\right)^{\frac{1}{k}} \!\!\! e^{-\frac{(2j+1)\pi i}{2k}}\bigg)
 	\!+\!e^{-\frac{i\pi(2j+1)(k-1)}{2k}}\overline{L}_{k}\bigg(\!2\pi\left(\frac{2\pi}{w}\right)^{\frac{1}{k}}\!\!\!
 	e^{\frac{(2j+1)\pi i}{2k}}\bigg) \! \!  \bigg\},
	\end{align}
	where $\overline{L}_{k}(w):=\displaystyle\sum_{n=1}^{\infty}\frac{n^{\frac{1}{k}-1}}{{\exp}(n^{\frac{1}{k}}w)-1}$. For an odd $k>1$, $L_k(w)$ admits the  asymptotic formula:
	\begin{align}\label{wigertkodd}
		L_k(w)&=\frac{\zeta(k)}{w}+w^{-\frac{1}{k}}\G\left(1+\frac{1}{k}\right)\zeta\left(\frac{1}{k}\right)+\frac{1}{4}\pm\frac{i}{k}(-1)^{\frac{k-1}{2}}\left(\frac{2\pi}{w}\right)^{\frac{1}{k}}\overline{L}_{k}\bigg(2\pi\left(\frac{2\pi}{w}\right)^{\frac{1}{k}}\bigg)\nonumber\\
		&\quad+4(-1)^{\frac{k+1}{2}}\sum_{j=1}^{N}\frac{\zeta(2j)\zeta(2kj-k+1)(k(2j-1))!}{(2\pi)^{2(k+1)j-k+1}}w^{2j-1}+\frac{i}{k}(-1)^{\frac{k-1}{2}}\left(\frac{2\pi}{w}\right)^{\frac{1}{k}}\nonumber\\
		&\quad\times\sum_{j=0}^{\frac{k-1}{2}}\bigg\{e^{\frac{i\pi j(k-1)}{k}}\overline{L}_{k}\bigg(2\pi\left(\frac{2\pi}{w}\right)^{\frac{1}{k}}e^{-\frac{i\pi j}{k}}\bigg)-e^{\frac{-i\pi j(k-1)}{k}}\overline{L}_{k}\bigg(2\pi\left(\frac{2\pi}{w}\right)^{\frac{1}{k}}e^{\frac{i\pi j}{k}}\bigg)\bigg\}+\Theta(N),
	\end{align}
	where
	\begin{align*}
		\Theta(N)& :=\frac{(2\pi)^k}{w}\frac{(-1)^{\frac{k-1}{2}}}{2\pi i}\int_{1+\rho-i\infty}^{1+\rho+i\infty}\left(\frac{w}{(2\pi)^{k+1}}\right)^{s}\left(\cot\frac{\pi s}{2}\mp i\right)\G(ks-k+1)\zeta(ks-k+1)\zeta(s)\, ds \nonumber\\
		& =O_{k, N}(|w|^{2N}),
	\end{align*}
	as $w\to 0$ in the region $|\arg w|\leq \l<\pi/2$. 
	
\end{theorem}
In the  case $k=1$,  the asymptotic expansion of the Lambert series $\sum_{n=1}^\infty d^{(k)}(n) \exp(-n w) $ was previously obtained by Wigert himself in \cite{wig0} (see also \cite[p.~163, Theorem 7.15]{titch}). In a follow-up paper \cite{wigannalen}, Wigert also obtained a Riesz-type identity for $d^{(k)}(n)$, of the type in \eqref{21}, for any $\rho>1$.

Koshliakov \cite[Equation (4)]{koshwigleningrad} obtained the Vorono\"{\dotlessi} summation formula for $d^{(k)}(n)$ given below. He took $k$ to be even in this result since he later wanted to give its special case for $f(w)=e^{-nw}$, which gives an exact formula \emph{only} for  even $k$ (see \eqref{wigertkeven} above), however, the result itself is true for any $k\in\mathbb{N}$.

\begin{theorem}\label{Koshliakov}
	Let $0 < \alpha < \beta$ and $ \alpha,  \beta \not\in \mathbb{Z}$.  Let $k>1$ be a natural number. Let $f(x)$ be an analytic function defined inside a closed contour containing $[\alpha,  \beta]$. Then
	\begin{align}\label{Koshliakov_eqn}
		\sum_{\a<n<\b} d^{(k)}(n)f(n) & = \int_{\a}^{\b} \left( \zeta(k) +\frac{1}{k}\zeta\left(\frac{1}{k}\right) y^{\frac{1}{k}-1}\right) f(y) dy \nonumber\\
		& + 4 (2\pi)^{1/k-1} \sum_{n=1}^\infty S^{(k)}(n) \int_{\a}^\b H_{0}^{(k)}\left( (2\pi)^{1+ 1/k} (ny)^{1/k} \right)y^{\frac{1}{k}-1} f(y) dy,
	\end{align}
	where $S^{(k)}(n):=S_{0}^{(k)}(n)$ and  $H_{0}^{(k)}(x):= \int_{0}^\infty \cos(1/t^k) \cos(x t) t^{-k} dt.$
\end{theorem}
\begin{remark}
Although the results in Koshliakov's paper \cite{koshwigleningrad} are correct, we warn the readers of many typographical errors. For example, in the argument of the function $L^{(k)}$ in his version of the above result, the expression $(2\pi)^{\frac{1}{k}-1}$ should be replaced by  $(2\pi)^{\frac{1}{k}+1}$.
\end{remark}

\begin{remark}
	 Using the fact that
	 \begin{equation*}
	 	\lim_{s\to1}\zeta(s)+\frac{1}{s}\zeta\left(\frac{1}{s}\right)y^{\frac{1}{s}-1}=2\gamma+\log(y),
	 \end{equation*}
	 Theorem \ref{Koshliakov} can be easily modified to accommodate the case $k=1$. Indeed, this gives \eqref{kvsf} upon using \eqref{Hardy} below.
\end{remark}

 One can extend Theorem \ref{Koshliakov} by letting $\a\to 0$ and $\b\to\infty$ but with the obvious need of putting further restrictions on $f$. This is, of course, permitted when $f(x)=e^{-xw}, x>0$, Re$(w)>0$, (because of the exponential decay), and results in Wigert's \eqref{wigertkeven} as a corollary. \\

In this paper, we prove Vorono\"{\dotlessi} summation formulas for the generalized divisor function $\sigma_z^{(k)}(n)$ defined in \eqref{sigmakz}. We give two such formulas, one of which applies with a test function $f$ analytic in an interval $[\a, \b]$ (see Theorem \ref{Voronoi_Sigma_z} below), while the other is not truncated to any interval and can be applied with a test function belonging to the Schwartz class (see Theorem \ref{Gen_Voronoi_Sigmakz}). Thus, our first version of the Vorono\"{\dotlessi} summation formula is a generalization of Theorem \ref{Koshliakov} of Koshliakov. There are instances in Koshliakov's paper \cite{koshwigleningrad} where the results are but merely stated but not proved at all, for example, \cite[Equation (6)]{koshwigleningrad}. Our generalization of his Equation (6), which is given in Theorem \ref{H_z_sum of two_K_z}, rigorously proves his Theorem \ref{Koshliakov} given above as a special case of our Theorem \ref{Voronoi_Sigma_z}. As can be seen, the proof of Theorem \ref{H_z_sum of two_K_z} is quite non-trivial and requires the theory of  linear differential equations \cite[p.~21, Section 6]{coddington-levinson} along with properties of combinatorial objects such as the Stirling numbers of the second kind and the elementary symmetric polynomials. Also, we later derive \eqref{wigertkeven} as a special case of a more general result, namely, Corollary \ref{Wigert_gen_ for z=2m}.

In addition to obtaining the Vorono\"{\dotlessi} summation formulas for  $\sigma_z^{(k)}(n)$, this paper is equally devoted to developing the theories of the new special functions $H_z^{(k)}(x)$ and $K_z^{(k)}(x)$ which arise in this context and which are defined in \eqref{defn_H_z^k} and \eqref{line integral_K_z} respectively. While the functions $H_z^{(1)}(x)$ and $K_z^{(1)}(x)$ reduce to certain combinations of Bessel functions, for $k>1$, their theory transcends the well-known theory of Bessel functions. This certainly warrants their study.

\section{Main results}\label{mr}
We first show that the generalized divisor function $\sigma_{z}^{(k)}(n)$ defined in \eqref{sigmakz} and its Dirichlet series in \eqref{gen_sigma_z} are not covered by the setting of Chandrasekharan and Narasimhan in \cite{cn2} unless $k\in\mathbb{N}$ and $z=\frac{k-1}{2}$ or unless $k=1$ and $z\in\mathbb{R}$. To that end, first note that the symmetric form of the functional equation of $\zeta(s)$ reads \cite[p.~22, Equation (2.6.4)]{titch}
\begin{equation}\label{zetafe}
	\pi^{-\frac{s}{2}}\Gamma\left(\frac{s}{2}\right)\zeta(s)=\pi^{-\frac{(1-s)}{2}}\Gamma\left(\frac{1-s}{2}\right)\zeta(1-s).
	\end{equation}
Along with \eqref{genDelta} and \eqref{gen_sigma_z} and the Dirichlet series defined in \eqref{18}, this implies that if
\begin{equation}\label{anbn}
a(n)=\pi^{z/2}\sigma_{z}^{(k)}(n)=b(n),\hspace{6mm} \l_n=\pi^{\frac{1}{2}(k+1)}n=\mu_n
\end{equation}
so that
\begin{align*}
\varphi(s)=\psi(s)=\pi^{\frac{1}{2}(z-(k+1)s)}\zeta(s)\zeta(ks-z),\hspace{6mm} \Delta(s)=\Gamma\left(\frac{s}{2}\right)\Gamma\left(\frac{ks-z}{2}\right),
\end{align*}
then we must have, for some $\delta>0$, 
\begin{equation*}
	\Delta(\delta-s)=\Gamma\left(\frac{1-s}{2}\right)\Gamma\left(\frac{1-ks+z}{2}\right),
\end{equation*}
 It is easy to see that this will be true only if $z=\frac{k-1}{2}$ for $k\in\mathbb{N}$ and $\delta=1$ or if $k=1$ and $\delta=z+1>0$. (We get these same conditions if we work with $S_{z}^{(k)}(n)$ rather than $\sigma_{z}^{(k)}(n)$ in \eqref{anbn}.) Thus, our Theorems \ref{Voronoi_Sigma_z} and \ref{Gen_Voronoi_Sigmakz} are covered by the setting of Chandrasekharan and Narasimhan \emph{only} in the aforementioned two special cases which force either $z$ to be rational or $k$ to be $1$. On the other hand, our Theorems \ref{Voronoi_Sigma_z} and \ref{Gen_Voronoi_Sigmakz} hold for any $k\in\mathbb{N}$ and any complex $z$ such that $-1<\textup{Re}(z)<k$.

Before stating Theorem \ref{Voronoi_Sigma_z}, we define the function\footnote{The notation here \emph{does not} mean $k$-th derivative of some function $H_z(x)$. This notation is used to comply with that used by Wigert \cite{wig} and Koshliakov \cite{koshwigleningrad} for the associated arithmetic as well as special functions and is retained throughout the paper for other functions as well. For the $j^{\textup{th}}$ derivative of, say,  $H_z^{(k)}(x)$ with respect to $x$, we use the notation $\frac{d^j}{dx^j}H_z^{(k)}(x)$.} $H_z^{(k)}(x)$ for $k\in\mathbb{N}$ and $x\geq0$ by
\begin{align}\label{defn_H_z^k}
	H_z^{(k)}(x):= \int_{0}^\infty  t^{z-k}\cos(x t)\cos\left(\frac{1}{t^k}\right)\, dt.
\end{align}
In Theorem \ref{gen_Hardy_conv}, it is shown that $H_z^{(k)}(x)$ converges for $-1<\textup{Re}(z)<k$.

For $k=1$, this integral was evaluated by Hardy \cite[p.~184, Equation (4)]{watson} who showed, in particular, that
\begin{align}\label{Hardy}
	H_0^{(1)}(x) =K_{0}(2 \sqrt{x}) -\frac{\pi}{2} Y_{0}(2 \sqrt{x}). 
\end{align}
 Also, the integral $H_{0}^{(k)}(x)$ appeared in Koshliakov's result \eqref{Koshliakov_eqn}. Hence we call the integral $H_{0}^{(k)}(x)$ as the \emph{Hardy-Koshliakov integral}. Theorem \ref{k=1} below generalizes \eqref{Hardy} for any $z$ satisfying $-1<\textup{Re}(z)<1$.
 
As is shown in \cite[Equations (1.14), (4.1)]{dkk1}, the function $H_z^{(1)}(x)$ is a special case of a kernel of Watson \cite{watsonself} given by
 	\begin{equation*}
 		\varpi_{\mu, \nu}(xy):=x^{1/2}\int_{0}^{\infty}J_{\nu}(xt)J_{\mu}\left(\frac{1}{t}\right)\frac{dt}{t},
 	\end{equation*}
namely, $H_z^{(1)}(x)=\frac{\pi}{2}x^{-z/2}\varpi_{\frac{-z-1}{2}, \frac{z-1}{2}}(x)$.
This follows from the fact that $J_{-1/2}(x)=\sqrt{\frac{2}{\pi x}}\cos(x)$. However, for $k>1$, our kernel $H_z^{(k)}(x)$ is new. In Theorem \ref{H_z_sum of two_K_z}, we show that $H_z^{(k)}(x)$ essentially equals the Meijer $G$- function $G_{0,2k+2}^{\,k+1,0} \!\left(  \,\begin{matrix} \{  \} \\ b_1, \cdots , b_{2k+2} \end{matrix} \; \Big| \frac{1}{4} \left(\frac{x}{2k} \right)^{2k}   \right)$ with the parameters $b_1, \cdots, b_{2k+2}$ defined in \eqref{coefficients_bj} below. A similar Meijer $G$-function acting as a symmetric Fourier kernel was studied by Narain \cite[Equation (1.5)]{narain}, however, it is different from ours.

We will also need an auxiliary integral in the proof of the Vorono\"{\dotlessi} summation formula for  $\sigma_{z}^{(k)}(n)$. This integral is defined for $x\neq 0$ by
\begin{align}\label{line integral_K_z}
	K_z^{(k)}(x) := \frac{1}{2 \pi i } \int_{(c)} \Gamma(s) \cos\left( \frac{\pi s}{2}  \right)  \Gamma\left( \frac{s-1-z}{k} +1 \right) \frac{\, ds}{k x^s},
\end{align}
where $k\in\mathbb{N}$ and $\max\{ 0,  1 -k +  \Re(z) \} < \Re(s)=c$  when $|\arg(x)|<\pi/(2k)$, and $\max\{ 0,  1 -k +  \Re(z) \} < \Re(s)=c\leq\frac{1+\textup{Re}(z)}{k+1}$ when $|\arg(x)|=\pi/(2k)$. (Note that the latter strip is non-empty when we additionally assume $-1<\textup{Re}(z)<k$, which will be the case in most of our results.) The integral is absolutely convergent in $|\arg(x)|<\pi/(2k)$ but only conditionally convergent for $|\arg(x)|=\pi/(2k)$. While the former is easily established using Stirling's formula (see \eqref{strivert} below), the latter requires some work, and is hence proved in detail at the end of the proof of Theorem \ref{meijergk}. As a function of $x$, $K_z^{(k)}(x)$ is analytic in $|\arg(x)|<\pi/(2k)$. Also, as a function of $z$, it is analytic in $\mathbb{C}$, provided $|\arg(x)|<\pi/(2k)$; when $|\arg(x)|=\pi/(2k)$, it is analytic in $-1<\textup{Re}(z)<k$.
 
 The functions $H_z^{(k)}(x)$ and $K_z^{(k)}(x)$ are related to each other by means of the following important identity.
 \begin{theorem}\label{H_z_sum of two_K_z}
 	Let $x\geq0$ and $k\in\mathbb{N}$. For any $z$ such that $-1 < \Re(z) < k$,
 	\begin{align} \label{H_z in terms K_z}
 		H_z^{(k)}(x) &= \frac{1}{2} \left\{ \exp\left(\frac{i \pi(k-1-z)}{2 k}  \right) K_z^{(k)}\left( e^{-\frac{i \pi}{2k}} x  \right) + \exp\left(\frac{- i \pi(k-1-z)}{2 k}  \right) K_z^{(k)}\left( e^{\frac{i \pi}{2k}} x  \right)     \right\}\nonumber\\
 		&=\frac{\pi}{\sqrt{k}2^{\frac{1+z}{k}}}G_{0,2k+2}^{\,k+1,0} \!\left(  \,\begin{matrix} \{  \} \\ b_1, \cdots , b_{2k+2} \end{matrix} \; \Big| \frac{1}{4} \left(\frac{x}{2k} \right)^{2k}   \right),
 	\end{align}
 	where the parameters $b_j, 1\leq j\leq 2k+2,$ are given by
 	\begin{align}\label{coefficients_bj}
 		b_j = \begin{cases}
 			\frac{j-1}{k},  & {\rm if} \,\,    1 \leq j \leq k,   \\
 			\frac{1}{2} -\frac{1+z}{2k},  & {\rm if} \,\,  j=k+1,  \\
 			2 + \frac{3-2j}{2k},  & {\rm if} \,\,  k+2 \leq j \leq 2k+1,  \\
 			1-  \frac{1+z}{2k},  & {\rm if} \,\,  j=2k+2,
 		\end{cases}
 	\end{align} 
 and $G_{0,2k+2}^{\,k+1,0} \!\left(  \,\begin{matrix} \{  \} \\ b_1, \cdots , b_{2k+2} \end{matrix} \; \Big| \frac{1}{4} \left(\frac{x}{2k} \right)^{2k}   \right)$ is the Meijer $G$-function defined in \eqref{MeijerG}.
 \end{theorem}
 This result will be crucially used in the proofs of Theorems \ref{Voronoi_Sigma_z} and \ref{Gen_Voronoi_Sigmakz}. The Meijer $G$-function occurring in \eqref{H_z in terms K_z} is conditionally convergent for $x\geq0$. A proof of this result is offered in the Appendix by employing the theory of linear differential equations whereas the differential equations satisfied by $H_z^{(k)}(x)$ and the Meijer $G$-function themselves are derived in Section \ref{hk}. Compare \eqref{hade_eqn} with \eqref{geqn} by means of Lemma \ref{coeff_de}.
 
 We are now ready to give our first version of the Vorono\"{\dotlessi} summation formula for  $\sigma_{z}^{(k)}(n)$ which involves $H_z^{(k)}(x)$, the aforementioned generalization of the Hardy-Koshliakov integral.
\begin{theorem}\label{Voronoi_Sigma_z}
	Let $0 < \alpha < \beta$ and $ \alpha,  \beta \not\in \mathbb{Z}$.  Let $k\in\mathbb{N}$ and $z\in\mathbb{C}$ with $-1 < \Re(z) <k$ and $z \neq k-1$.  Let $S_{z}^{(k)}(n)$ be defined in \eqref{skzn} and let $f(x)$ be analytic inside a closed contour containing $[\alpha,  \beta]$.  Then
	\begin{align}\label{Voronoi_Sigma_z_eqnn}
		\sum_{ \alpha < n < \beta }  \sigma_z^{(k)}(n)f(n) & =  \int_{\alpha}^{\beta}  f(t) \left( \zeta(k-z) + \frac{1}{k} t^{\frac{1+z}{k}-1} \zeta\left( \frac{1+z}{k} \right)\right) \, dt \nonumber  \\
		&\quad+2 (2\pi)^{\frac{1+z}{k}-1}  \sum_{n=1}^{\infty}  S_{z}^{(k)}(n)  \int_{\alpha}^{\beta }  f(t)  t^{\frac{1+z}{k}-1} H_z^{(k)}\left( ( 2\pi)^{\frac{1}{k}+1} (n t)^{\frac{1}{k}} \right)\, dt. 
	\end{align} 
If $z=k-1$, then
\begin{align}\label{Voronoi_Sigma_z_eqnnz=k-1}
	\sum_{ \alpha < n < \beta }  \sigma_{k-1}^{(k)}(n)f(n) & =  \int_{\alpha}^{\beta}  f(t) \left(\frac{(k+1)\gamma+\log(t)}{k}\right)\, dt \nonumber  \\
	&\quad+2 \sum_{n=1}^{\infty}  S_{k-1}^{(k)}(n)  \int_{\alpha}^{\beta }  f(t) H_{k-1}^{(k)}\left( ( 2\pi)^{\frac{1}{k}+1} (n t)^{\frac{1}{k}} \right)\, dt.
\end{align} 
\end{theorem}
As a corollary of the above theorem, we obtain the well-known Vorono\"{\dotlessi} summation formula for $\sigma_z(n)$ \cite[Theorem 6.1]{bdrz1}.
\begin{corollary}\label{varlauform}
	Let $0<\alpha<\beta$ and $\alpha, \beta\notin\mathbb{Z}$. Let $f$ denote a
	function analytic inside a closed contour strictly containing  $[\alpha,
	\beta]$. Assume that $-1< \textup{Re}(z) <1, z\neq0$. Then\footnote{In \cite[Theorem 6.1]{bdrz1}, it was assumed that $-1/2<\textup{Re}(z)<1/2$. However, we see here that the result actually holds for $-1<\textup{Re}(z)<1$.},
	\begin{align}\label{varlauform1}
		\sum_{\alpha<j<\beta}\sigma_{-z}(j)f(j)&=\int_{\alpha}^{\beta}(\zeta(1+z)+t^{-z}\zeta(1-z))f(t)\, dt\nonumber\\
		&\quad+2\pi\sum_{n=1}^{\infty}\sigma_{-z}(n)n^{\frac{1}{2}z}\int_{\alpha}^{\beta}t^{-\frac{1}{2}z}f(t)
		\bigg\{\left(\frac{2}{\pi}K_{z}(4\pi\sqrt{nt})-Y_{z}(4\pi\sqrt{nt})\right)\nonumber\\
		&\quad\quad\quad\quad\quad\quad\times\cos\left(\frac{\pi z}{2}\right)-J_{z}(4\pi\sqrt{n t})\sin\left(\frac{\pi z}{2}\right)\bigg\}\, dt.
	\end{align}
Moreover, if we let $k=1$ in \eqref{Voronoi_Sigma_z_eqnnz=k-1}, then using \eqref{Hardy}, we get \eqref{kvsf}.
\end{corollary}
Also, letting $z=0$ in Theorem \ref{Voronoi_Sigma_z} readily gives Theorem \ref{Koshliakov}. Thus, Theorem \ref{Voronoi_Sigma_z} is a \emph{simultaneous} generalization of Koshliakov's result in Theorem \ref{Koshliakov} and Corollary \ref{varlauform}.

Throughout the sequel, $F(s)$, or $\mathcal{M}(f)(s)$, will always denote the Mellin transform of a function $f$.  Let $\mathscr{S}(\mathbb R)$ denote the space of Schwartz functions on $\mathbb R$, that is, those functions $f$ which satisfy $f\in C^{\infty}(\mathbb{R})$ and all of whose derivatives (including $f$ itself) tend to $0$ faster than any power of $|x|$ as $|x|\to\infty$. See \cite[p.~177]{cohenbook}. Our next result is the ``infinite'' version of Theorem \ref{Voronoi_Sigma_z} for Schwartz functions $f$. 
\begin{theorem}\label{Gen_Voronoi_Sigmakz} Let $k \in \mathbb{N}, z \in \mathbb{C}$ be such that $ -1 < \Re(z) < k, z\neq k-1$.
	Let $f \in \mathscr S (\mathbb R)$.
	Then we have
	\begin{align}\label{extra}
		\sum_{n=1}^\infty \sigma_{z}^{(k)}(n) f(n) & = -\frac{1}{2}\zeta(-z)f(0^{+}) + \zeta(k-z) \int_{0}^\infty f(y) dy + \frac{1}{k} F\left( \frac{1+z}{k} \right)\zeta\left( \frac{1+z}{k} \right) \nonumber\\ 
		& + \frac{(2 \pi)^{(k+1)\left(\frac{1+z}{k}\right)-z}}{\pi^2} \sum_{n=1}^\infty S_{z}^{(k)}(n) \int_{0}^\infty H_z^{(k)}\left( (2\pi)^{1+ 1/k} (ny)^{1/k} \right)y^{\frac{1+z}{k}-1} f(y) dy,
	\end{align}
	where $S_{z}^{(k)}(n)$ and $ H_z^{(k)}(x)$ are defined in \eqref{skzn} and \eqref{defn_H_z^k} respectively. Moreover, when $z=k-1$, 
	\begin{align*}
			\sum_{n=1}^\infty \sigma_{k-1}^{(k)}(n) f(n) & = -\frac{1}{2}\zeta(1-k)f(0^{+}) + 	\int_{0}^{\infty}  f(t) \left(\frac{(k+1)\gamma+\log(t)}{k}\right) dt \nonumber\\ 
		& + 4 \sum_{n=1}^\infty S_{k-1}^{(k)}(n) \int_{0}^\infty H_{k-1}^{(k)}\left( (2\pi)^{1+ 1/k} (ny)^{1/k} \right)f(y) dy.
		\end{align*}
\end{theorem}

\begin{remark}
	If we let $z=0$ in \eqref{extra}, then one can obtain the aforementioned extension of Koshliakov's Theorem \ref{Koshliakov}, that is, the one obtained by letting $\a \rightarrow 0$ and $\b \rightarrow \infty$ in it.
\end{remark}
Define $B(z,b)$ by
\begin{align}\label{integral defn B(z,b)}
	B(z,b):=\int_0^\infty \frac{t^z \cos t}{t^2+b^2}\, dt.
\end{align}
This integral converges only in the region $ -1<\Re(z)<2$, where it is also given by \cite[p.~43, Equation (5.8)]{Ober}\label{Ober}\footnote{There is a typo in the stated formula in \cite{Ober} in that $b^{-z}$ should be replaced by $b^{2-z}$.}
\begin{align}\label{bzbanother}
	B(z,b) &= \frac{\pi b^{z-1}} {2}   \frac{\cosh b}{\cos (\pi z/2)}  +\Gamma(z-1) \sin(\pi z/2) {}_1F_2\left( 1; 1-\frac{z}{2},  \frac{3-z}{2} \bigg\vert \frac{b^2}{4}\right)\\
	&=\frac{\pi b^{z-1}} {2}   \frac{\cosh b}{\cos (\pi z/2)}  - \frac{\pi}{2 \cos(\pi z/2)} \sum_{n=0}^\infty \frac{b^{2n}}{\Gamma(2n-z+2)}\label{simplified form B(z, b)},
\end{align}
where ${}_1F_2 \left( a; b, c \big\vert z \right) := \sum_{n=0}^\infty \frac{
	(a)_n}{(b)_n (c)_n }\frac{z^n}{n!}$
is the ${}_1F_{2}$-hypergeometric function with $(a)_n:=\Gamma(a+n)/\Gamma(a)$ being the shifted factorial. 
Here, $b\in\mathbb{C}$ with $b\neq \pm i y$ for any real number $y$. 

\begin{remark}\label{analytic continuation of B(z,b)}
	At first glance, it appears that the right-hand side of \eqref{bzbanother} has singularities at every odd integer. However, in Section \ref{wigertsection} we show that the odd positive integers are removable singularities whereas  it has poles at all odd negative integers. Thus the right-hand side of \eqref{bzbanother} provides meromorphic continuation of $B(z, b)$ to the whole complex plane with simple poles at $ z=-1,-3,-5, \cdots $.  
\end{remark}

As a special case of Theorem \ref{Gen_Voronoi_Sigmakz},  we obtain the following result.
\begin{theorem} \label{thm: gen of wigert}Let $k \in \mathbb{N}, z \in \mathbb{C}$ be such that $ -1< \Re(z) < k$ and $z\neq k-1$. Let $B(z, b)$ be defined in \eqref{integral defn B(z,b)} and \eqref{bzbanother}. 
	Let $ a = 2\pi \left( \frac{ 2\pi n}{w} \right)^{1/k}$, where $\Re(w)>0$ and  $A_j = \zeta_{4k}^{(2-k)(2j-1)}$, $B_j=\zeta_{4k}^{(1-k)(2j)}$, where $\zeta_{4k}$ is the primitive $4k$-th root of unity.  
	For $k \geq 2$ even,
	\begin{align}\label{eeven}
	\sum_{n=1}^\infty \sigma_{z}^{(k)}(n) e^{-nw} &= - \frac{\zeta(-z)}{2} + \frac{\zeta(k-z)}{w}  + \frac{1}{k}\frac{\Gamma\left( \frac{1+z}{k} \right) \zeta\left( \frac{1+z}{k} \right)}{w^{(1+z)/k}} + \frac{ (-1)^{\frac{k}{2}-1} (2 \pi)^{ 2+\frac{2}{k}-z  }    }{\pi^2 k w^{2/k} }  \nonumber\\ 
	& \times  \sum_{n=1}^\infty S_{z}^{(k)}(n)  n^{ \frac{1-z}{k} }  \sum_{j=1}^{\frac{k}{2}} \Bigg[ A_j  B(z,a \zeta_{4k}^{2j-1}) + \overline{A}_j B(z,a \zeta_{4k}^{-(2j-1)}) \Bigg],
	\end{align}
and	for $k\geq 1$ odd,
	\begin{align}\label{oodd}
	\sum_{n=1}^\infty  \sigma_{z}^{(k)}(n) e^{-nw} &= - \frac{\zeta(-z)}{2} + \frac{\zeta(k-z)}{w}  + \frac{1}{k}\frac{\Gamma\left( \frac{1+z}{k} \right) \zeta\left( \frac{1+z}{k} \right)}{w^{(1+z)/k}} + \frac{ (-1)^{\frac{k-1}{2}} (2 \pi)^{1+\frac{1}{k}-z}   }{\pi^2 k w^{1/k}} \nonumber \\ 
	& \times \sum_{n=1}^\infty S_{z}^{(k)}(n)  n^{-z/k} \Bigg[ B(z+1, a)
 + \sum_{j=1}^{\frac{k-1}{2}}  \Big[ B_j~ B\left(z+1,a \zeta_{4k}^{2j}\right)
  + \overline{B}_j~ B\left(z+1,a \zeta_{4k}^{-(2j)}\right) \Big] \Bigg]. 
	\end{align}
\end{theorem}

Recently, Gupta and Maji \cite{guptamaji} have obtained a transformation for the series $\sum_{n=1}^{\infty}D_{k, r}(n)e^{-ny}$, where $D_{k, r}(n):=\sum_{d^k|n}\left(\frac{n}{d^k}\right)^r$, $y>0, k\in\mathbb{Z}$ and $r\neq-1$ is an integer. It is clear that $D_{k. r}(n)=n^r\sigma_{-kr}^{(k)}(n)$.

Letting $z=2m$ in \eqref{eeven}, we obtain a generalization of Wigert's identity \eqref{wigertkeven}.
\begin{corollary}\label{Wigert_gen_ for z=2m}
Let $k \geq 2$ be an even integer and $m$ be a non-negative integer with $0 \leq m < k/2$. For $\Re(w) >0$, we have
\begin{align*}
\sum_{n=1}^\infty \sigma_{2m}^{(k)}(n) e^{-nw} &= - \frac{\zeta(-2m)}{2} + \frac{\zeta(k-2m)}{w}  + \frac{1}{k}\frac{\Gamma\left( \frac{1+2m}{k} \right) \zeta\left( \frac{1+2m}{k} \right)}{w^{(1+2m)/k}} + \frac{(-1)^{\frac{k}{2}+ m -1}}{k} \left(\frac{2\pi}{w} \right)^{\frac{1+2m}{k}}\\ & \times \sum_{j=1}^{k/2} \Bigg[\exp\left( \frac{i\pi}{2k}(1-k+2m)(2j-1) \right) \overline{L}_{k,2m}\left(2\pi \left(\frac{2\pi}{w} \right)^{\frac{1}{k}}  e^{ \frac{i\pi}{2k}(2j-1)} \right) \\
& + \exp\left(- \frac{i\pi}{2k}(1-k+2m)(2j-1) \right) \overline{L}_{k,2m}\left(2\pi \left(\frac{2\pi}{w} \right)^{\frac{1}{k}}  e^{- \frac{i\pi}{2k}(2j-1)} \right) \Bigg],
\end{align*}
where
\begin{equation}\label{Lkzw} 
\overline{L}_{k,z}(w):= \sum_{n=1}^{\infty} S_{z}^{(k)}(n) \exp(- n^{1/k} w ).   
\end{equation}
\end{corollary}
When $m=0$,  the above corollary reduces to Wigert's identity \eqref{wigertkeven}. Letting $z=2m-1$ in \eqref{oodd} leads to the odd counterpart of Wigert's identity. We note here that our result below is an \emph{exact} formula as compared to the asymptotic formula \eqref{wigertkodd} of Wigert.
\begin{corollary}\label{Wigert_gen_ for z=2m-1}
Let $k > 1$ odd and $m$ be an integer with $1 \leq m < \frac{k+1}{2}$.  For $\Re(w) >0$,
\begin{align*}
\sum_{n=1}^\infty & \sigma_{2m-1}^{(k)}(n)  e^{-nw} = - \frac{\zeta(1-2m)}{2} + \frac{\zeta(k-2m+1)}{w}  + \frac{1}{k}\frac{\Gamma\left( \frac{2m}{k} \right) \zeta\left( \frac{2m}{k} \right)}{w^{(2m)/k}} + \frac{(-1)^{\frac{k-1}{2}+ m}}{k} \left(\frac{2\pi}{w} \right)^{\frac{2m}{k}}\\ & \hspace{1cm}\times  \Bigg[ \overline{L}_{k,2m-1}\left(2\pi \left(\frac{2\pi}{w} \right)^{\frac{1}{k}} \right) + \sum_{j=1}^{\frac{k-1}{2}} \Big[\exp\left( \frac{i\pi j}{k}(-k+2m) \right) \overline{L}_{k,2m-1}\left(2\pi \left(\frac{2\pi}{w} \right)^{\frac{1}{k}}  e^{ \frac{i\pi j}{k}} \right) \\
& \hspace{5.5cm}+ \exp\left(- \frac{i\pi j}{k}(-k+2m) \right) \overline{L}_{k,2m-1}\left(2\pi \left(\frac{2\pi}{w} \right)^{\frac{1}{k}}  e^{- \frac{i\pi j}{k}} \right) \Big] \Bigg]. 
\end{align*}
where $\overline{L}_{k, z}(w)$ is defined in \eqref{Lkzw}.
\end{corollary}
\subsection{Significance of our results in a broader context}
To date,  Vorono\"{\dotlessi} summation formulas provide one of the most useful methods to  obtain bounds for mean values of arithmetic functions. A case in point is the Dirichlet divisor problem,  a fundamental problem in number theory. The error term 
\begin{align} \label{error dir}
\Delta(x) \coloneqq \sideset{}{'}\sum_{n\leq x}d(n) - x\log x-(2\gamma-1))x-\frac{1}{4}
\end{align}
was shown to be $\ll x^{1/2}$ by Dirichlet. It was Vorono\"{\dotlessi}'s ingenious idea of writing the  general summation formula \eqref{kvsf} for the summatory function of $d(n)$ times a test function $f(n)$  that allowed him to obtain $\Delta(x) \ll x^{1/3}$.  The overarching principle behind how a summation formula like \eqref{kvsf}  proves useful is that it enables one to express the error $\Delta(x)$ as an infinite series involving certain kernel functions (in this case, the Bessel functions): see for instance \eqref{error as inf sum}. One then needs to invoke known asymptotic estimates for the kernel functions in order to obtain $O$-bounds or omega bounds for $\Delta(x)$. 

To understand how our results fit into this broader fabric, let us keep aside convergence issues for the moment and informally consider $\a \to 0^+, \b \to x$ and $f$ as identically one in our main result, Theorem \ref{Voronoi_Sigma_z}. The integral in \eqref{Voronoi_Sigma_z_eqnn} then yields the main term $$\zeta(k-z) x + \frac{1}{z+1} \zeta \left(\frac{1+z}{k} \right) x^{ \frac{z+1}{k} }, $$
for the sum $\sum_{n \le x} \sigma_{z}^{(k)}(n)$. 
Note that this recovers the main term $x (\log x  + (2\gamma-1))$ of \eqref{vsf} when we let $k=1$, write out the Laurent series expansions and let $z\to 0$. Our result \eqref{Voronoi_Sigma_z_eqnn} also 
allows us to write the  error term in the expression for $\sum_{n \le x} \sigma_{z}^{(k)}(n)$ as an infinite series involving the  generalized Hardy-Koshliakov integrals $H_z^{(k)}(x)$ defined in \eqref{defn_H_z^k}, that is, 
\begin{equation*}
	\Delta_{z, k} (x) =   2 (2\pi)^{\frac{1+z}{k}-1} \sum_{n=1}^{\infty}  S_{z}^{(k)}(n)  \int_{0}^{x}   t^{\frac{1+z}{k}-1} H_z^{(k)}\left( ( 2\pi)^{\frac{1}{k}+1} (n t)^{\frac{1}{k}} \right)\, dt .
\end{equation*}
It is $H_z^{(k)}(x)$  that now plays the role of the kernel function, acting as a vast generalisation of the Bessel kernel $K_{0}(2 \sqrt{x}) -\frac{\pi}{2} Y_{0}(2 \sqrt{x})$. In order to obtain upper bounds and omega results for 	$\Delta_{z, k} (x)$, it is essential to develop the theory of the function $H_z^{(k)}(x)$. 

This paper represents a fundamental contribution towards building this theory. Indeed, in Theorems \ref{asym zero} and \ref{asym of H_z^k}, we determine the asymptotic behavior of $H_z^{(k)}(x)$ as $x \to \infty$. Moreover, in the special case $k=1$,  we obtain an explicit expression for $H_z^{(k)}(x)$ in terms of Bessel functions. With these ingredients in place, we expect to be able to prove, for instance, that $\Delta_{z, k}(x) \ll x^{1/3} (\log x)^2$ for $k>1$ and $-1<\Re(z)<k, z\neq k-1$. 
This would not only yield the asymptotic formulae given by Theorem $1.4$ of Robles and Roy in \cite{roblesroy}, but also extend them to larger ranges of $z$ as well as to $k\geq 4$. It may also be possible to determine better upper bounds for $\Delta_{z, k}(x)$ by obtaining non-trivial estimates for the kernel  $H_z^{(k)} (x)$. 

Since these non-trivial estimates call for extensive additional analysis, including them here would render the manuscript unwieldy. 
Moreover, the results obtained here are of independent interest and mark the natural conclusion of our initial line of inquiry, thereby meriting their documentation in this paper. 

\section{Preliminaries}\label{prelim}
Here we state some known results which will be useful in the sequel. For $0< \Re(s) <1$, the Mellin transform of $\cos(x)$ is given by \cite[p.~1101, Formula (3)]{grn}
\begin{equation}\label{Mellin_cos}
\int_{0}^{\infty} \cos(x) x^{s-1} dx = \Gamma(s) \cos\left(\frac{\pi s}{2}\right).
\end{equation}
%
Stirling's formula for $\Gamma(s)$, $s=\sigma+it$, in a vertical strip $C\leq\sigma\leq D$ is given by \cite[p.~224]{cop}
\begin{equation}\label{strivert}
	|\Gamma(s)|=(2\pi)^{\tfrac{1}{2}}|t|^{\sigma-\tfrac{1}{2}}e^{-\tfrac{1}{2}\pi |t|}\left(1+O\left(\frac{1}{|t|}\right)\right),
\end{equation}
as $|t|\to\infty$. Throughout the paper,  $\int_{(c)}$ will always denote the line integral $\int_{c-i\infty}^{c+i\infty}$.

\begin{theorem}[Parseval's formula] \cite[p.~83, Equation (3.1.13)]{kp}\label{Parseval}
Let $F(s)$ and $G(s)$ be the Mellin transforms of $f(x)$ and $g(x)$ respectively. If $F(1-s)$ and $G(s)$ have a common strip of analyticity, then for any vertical line $Re(s)=c$ in the common strip, we have
\begin{align}\label{Parseval1}
\frac{1}{2 \pi i} \int_{(c)} G(s) F(1-s) ds = \int_{0}^\infty f(t) g(t) dt,
\end{align}
under the assumption that the integral on the right-hand side exists and the conditions 
\begin{equation}\label{conditions}
t^{c-1} g(t) \in L[0, \infty)\hspace{5mm}\text{and}\hspace{5mm}F(1-c-it) \in L(-\infty,  \infty)
\end{equation}
hold. 
\end{theorem}
An extension of Parseval's formula due to Vu Kim Tuan \cite{vu} is given in the next theorem. This result allows application of Parseval's formula in situations where the first condition in \eqref{conditions} does not hold, albeit with an additional restriction. We will require this in the proof of Theorem \ref{kkzxreal}. Before we state this extension though, we define the concepts needed to do so, namely, a new function space and a certain class of functions. 

Let $\mathfrak{M}^{-1}(L)$ denote the space of functions $f(x)$ which are inverse Mellin transforms of functions $F(s)\in L\left(\frac{1}{2}-i\infty, \frac{1}{2}+i\infty\right)$ over the contour Re$(s)=1/2$ with norm $||f||_{\mathfrak{M}^{-1}(L)}$ equal to $\int_{0}^{\infty}\left|F\left(\frac{1}{2}+it\right)\right|\, dt$.

We  let $\mathcal{K}$ be the set of functions $g(x)$ integrable on any segment $[\epsilon, E], 0<\epsilon<E<\infty$, and such that the improper integral
\begin{equation*}
	\mathfrak{M}\{g(x); s\}=G(s)=\int_{0}^{\infty}x^{s-1}g(x)\, dx,\hspace{8mm}\textup{Re}(s)=\frac{1}{2},
\end{equation*}
converges boundedly, that is, there exists a constant $C>0$ such that for almost all $\epsilon, E>0$ and $t\in\mathbb{R}$, we have $\left|\int_\epsilon^{E}x^{it-1/2}g(x)\, dx\right|<C$.

Then the extension of Parseval's theorem \cite[Lemma 1]{vu} is as follows.
 \begin{theorem}\label{parseval_extension_tuan}
Let $f(x)\in\mathfrak{M}^{-1}(L)$ and $g(x)\in\mathcal{K}$. Then the following convolution formula holds:
\begin{equation}\label{parseval_extension_tuan_eqn}
	\int_{0}^{\infty}g(xt)f(t)\, dt=\frac{1}{2\pi i}\int_{\left(\frac{1}{2}\right)}G(s)F(1-s)x^{-s}\, ds.
\end{equation}
\end{theorem}
\begin{remark}\label{extension_s}
Using Cauchy's residue theorem, we note that \eqref{parseval_extension_tuan_eqn} can be extended to any vertical strip containing the line $[1/2-i\infty, 1/2+i\infty]$ as long as it does not contain any poles of the integrand and the integrals along the horizontal segments of the rectangular contour tend to zero as the height of the contour tends to $\infty$.
\end{remark}
\begin{remark}\label{parsevalcos}
As mentioned in \cite[Corollary 1]{vu}, the cosine function belongs to the class $\mathcal{K}$ and hence the extension of Parseval's formula, that is, \eqref{parseval_extension_tuan_eqn} holds with $g(x)=\cos(x)$ and $f\in\mathfrak{M}^{-1}(L)$. It is this fact that will be employed in the proof of Theorem \ref{kkzxreal}.
\end{remark} 
These results are also given in \cite[p.~15-17]{yakubovich}.

The next result, which gives the evaluation of a Mellin transform of a certain rational function, will be used in the sequel.
\begin{lemma}\label{Mellin}
For $-k-\Re(z)< c=\Re(s) < k-\Re(z)$,
\begin{align*}
\frac{1}{2\pi i}\int_{(c)}\frac{\pi t^{-s}\, ds}{\cos\left(\frac{\pi}{2k}\left(z+s\right) \right)}=\frac{2kt^{k+z}}{t^{2k}+ 1}.
\end{align*}
\end{lemma}
\begin{proof}
	Employ the change of variable $t=v^{1/(2k)}$ in the integral below so that
\begin{align*}
	\int_{0}^\infty t^{s-1}\frac{2kt^{k+z}}{t^{2k}+ 1}\, dt =\int_0^\infty \frac{v^{\frac{(s+k+z)}{2k}-1}}{v+1} dv=\frac{\pi }{\sin\left(\pi\left(\frac{1}{2}+\frac{z+s}{2k} \right)\right)}=\frac{\pi }{\cos\left(\frac{\pi}{2k}\left(z+s\right) \right)}, 
\end{align*}
since $-k<\textup{Re}(s+z)<k$. The result now follows from the Mellin inversion theorem \cite[p.~341]{Mcla}.
\end{proof}

Next, we define an important special function called the  Meijer $G$-function \cite[p.~415, Definition 16.17]{NIST}.
Let $m,n,p,q$ be integers such that $0\leq m \leq q$, $0\leq n \leq p$. Let $a_1, \cdots, a_p$ and $b_1, \cdots, b_q$ be complex numbers such that $a_i - b_j \not\in \mathbb{N}$ for $1 \leq i \leq n$ and $1 \leq j \leq m$.   The Meijer $G$-function is defined by 
\begin{align}\label{MeijerG}
G_{p,q}^{\,m,n} \!\left(  \,\begin{matrix} a_1,\cdots , a_p \\ b_1, \cdots , b_q \end{matrix} \; \Big| X   \right) := \frac{1}{2 \pi i} \int_L \frac{\prod_{j=1}^m \Gamma(b_j - w) \prod_{j=1}^n \Gamma(1 - a_j +w) X^w  } {\prod_{j=m+1}^q \Gamma(1 - b_j + w) \prod_{j=n+1}^p \Gamma(a_j - w)}\, dw.
\end{align}
Here $L$  goes from $-i \infty$ to $+i \infty$ separating the poles of $\Gamma(1-a_j+w)$  from the poles of $\Gamma(b_j-w)$.  Note that the integral converges  absolutely if $p+q  < 2(m+n)$ and $|\arg(X)| < (m+n - \frac{p+q}{2}) \pi$.  In the case $p+q  = 2(m+n)$ and $\arg(X)=0$,  the integral converges absolutely if $ \left(  \Re(w) + \frac{1}{2} \right) (q-p) > \Re(\psi ) +1$,  where $\psi = \sum_{j=1}^q b_j - \sum_{j=1}^p a_j$.  

The following result elucidates the asymptotic behaviour of the Meijer $G$-function $G_{p,q}^{\,m,0}(X)$ when the argument is large.
\begin{proposition}{\cite[Theorem 2,  p.~190]{Luke}}\label{asymptotic_Meijer G}
If $ 1  \leq m \leq q$,  then for $|X| \rightarrow \infty$,  we have
\begin{align*}
& G_{p,q}^{\,m,0} \!\left( X   \right) \sim A^{m,0}_{\hspace{0.3cm} q}  \, H_{p, q}\left(X e^{i\pi (q-m)} \right),  \textrm{ if} \,\, m\leq q-1,  \delta \leq \arg(X) \leq (m-p+1)\pi - \delta,  \delta >0; \\
& G_{p,q}^{\,m,0} \!\left( X   \right) \sim \overline{A^{m,0}_{\hspace{0.3cm} q}}  \, H_{p, q}\left(X e^{-i\pi (q-m)} \right),  \textrm{if}\,\, m\leq q-1,  \delta-(m-p+1)\pi \leq \arg(X)\leq-\delta,  \delta >0; 
\end{align*}
where \cite[p.~183, Equation (2)]{Luke}
\begin{align}\label{A}
A^{m,0}_{\hspace{0.3cm}  q}= \left(-\frac{1}{2\pi i}\right)^{\nu} \exp\bigg(\!\!\!-i\pi\!\!\! \sum_{j=m+1}^q b_j  \bigg), \quad \nu= q-m,
\end{align}
 $ \overline{A^{m,0}_{\hspace{0.3cm} q}} $ is obtained  by replacing $i$ by $-i$ in \eqref{A},  and the function $H_{p, q}(X) $ is defined by \cite[p.~180]{Luke}
\begin{align}\label{H_{p,q}}
H_{p, q}(X) = \frac{(2\pi)^{(\sigma-1)/2}}{\sqrt{\sigma}} \exp\left(-\sigma\, X^{1/\sigma}  \right) X^{\theta} \sum_{k=0}^\infty M_k X^{-k/\sigma},
\end{align}
with $\sigma= q-p$,  $ \theta =\frac{1}{\sigma}\left( \frac{1-\sigma}{2}+ \Xi_1 - \Lambda_1 \right)$, where $\Xi_1 = \sum_{j=1}^q b_j$ and $\Lambda_1=\sum_{j=1}^p a_j$.  Here  $M_0=1$ and $M_k'$s are independent of $X$. 

\end{proposition}

 We now state Slater's theorem \cite[p.~145, Equation (7)]{Luke}, which allows us to represent Meijer $G$-function in terms of generalized hypergeometric functions. 
If $p \leq q$ and $ b_j - b_h \not\in \mathbb{Z}$ for $j\neq h$, $1 \leq j, h \leq m$, then 
\begin{align}\label{Slater}
	 G_{p,q}^{\,m,n} \!\left(   \,\begin{matrix} a_1, \cdots , a_p \\ b_1, \cdots , b_q \end{matrix} \; \Big| z   \right)  
	& = \sum_{h=1}^{m} \frac{ z^{b_h}  \prod_{ j=1,  j\neq h}^{m} \Gamma(b_j - b_h ) \prod_{j=1}^n   \Gamma( 1 + b_h -a_j ) }{ \prod_{j=m+1}^{q} \Gamma(1 + b_h - b_{j}) \prod_{j=n+1}^{p} \Gamma(a_{j} - b_h)  }\\
	&\qquad\quad\times {}_p F_{q-1} \left(  \begin{matrix}
		1+b_h - a_1,\cdots, 1+ b_h - a_p \\
		1+ b_h - b_1, \cdots, *, \cdots, 1 + b_h - b_q 
	\end{matrix} \Big| (-1)^{p-m-n} z  \right), \nonumber 
\end{align}
where $*$ indicates that the entry $1 + b_h - b_h$ is omitted.

Lastly, we define certain mathematical objects which will play an important role in the proof of Theorem \ref{H_z_sum of two_K_z}. Consider the monic polynomial $(w-x_1)(w-x_2)\cdots(w-x_n)$ and let $X_n=\{x_1, x_2, \cdots, x_n\}$. For all $n, k\in\mathbb{N}$, the elementary 
symmetric polynomial $e_\ell(X_n)$ is given by \cite[p.~24]{egge}
\begin{equation}\label{esp}
	e_\ell(X_n):=\sum_{1\leq j_1<\cdots<j_\ell\leq n}\prod_{m=1}^{\ell}x_{j_m}.
\end{equation}
It is well-known that
\begin{equation*}
\sum_{j=0}^{n}e_j(X_{n})t^j=\prod_{j=1}^{n}(1+x_jt).
\end{equation*}
For all $n, k\in\mathbb{N}$, the Stirling number of the second kind $S(n, k)$ is the number of set partitions of $\{1, 2, \cdots, n\}$ with exactly $k$ non-empty parts. Clearly, $S(n, k)=0$ for $n<k$. By convention, $S(0, 0)=1$. See \cite[p.~204, Chapter V]{comtet} for more details.
 
\section{A generalization of the Hardy-Koshliakov integral $H_z^{(k)}(x)$}\label{hk}

\subsection{Convergence of $H_z^{(k)}(x)$} We begin with determining the values of $z$ for which the integral $H_z^{(k)}(x)$ in \eqref{defn_H_z^k} converges.
\begin{theorem}\label{gen_Hardy_conv}
	Let $x>0$ and $k\in\mathbb{N}$. Then $H_z^{(k)}(x)$ converges in $ -1 < \Re(z) < k$. Moreover, when $x=0$, it converges in $ -1 < \Re(z) < k-1$.
\end{theorem}
\begin{proof}
	Let $\epsilon>0$ be  small and $M$ be a large positive real number. For simplicity let $\nu=k-z$. We split the integral into three parts, namely, 
	\begin{align*}
		H_z^{(k)}(x) & = \int_{0}^\epsilon + \int_{\epsilon}^M + \int_{M}^\infty \cos\left(\frac{1}{t^k}\right) \cos(x t)  \frac{dt}{t^\nu}, =: I_1 + I_2 + I_3 \text{ (say)}. 
	\end{align*}
	It is easy to observe that $I_2$ is finite since the integrand is a continuous function on the closed and bounded interval $[\epsilon, M]$.
In the first integral $I_1$, replacing  $1/t^k$ by $T$ gives 	
	\begin{align*}
		I_1=  \frac{1}{k} \int_{\epsilon^{-k}}^\infty \cos(T)\cos\left(\frac{x}{T^{1/k}}\right) T^{\frac{\nu-1}{k} -1} dT.
	\end{align*}
	Using the series expansion of cosine, we have
	\begin{align*}
		I_1 &= \frac{1}{k} \int_{\epsilon^{-k}}^\infty \cos(T) \left[1 - \frac{x^2}{2! T^{2/k}} +  \frac{x^4}{4! T^{4/k}} - \cdots + O\left(  \frac{x^{2m}}{(2m)! T^{2m/k}} \right) \right] T^{\frac{\nu-1}{k} -1} dT
\\
		& = \frac{1}{k} \int_{\epsilon^{-k}}^\infty \cos(T)T^{\frac{\nu-1}{k} -1} dT - \frac{x^2}{ 2! k} \int_{\epsilon^{-k}}^\infty \cos(T)T^{\frac{\nu-3}{k} -1} dT+  \frac{x^4}{ 4! k} \int_{\epsilon^{-k}}^\infty \cos(T)T^{\frac{\nu-5}{k} -1} dT \\
		& \quad + \cdots+ O\left( \frac{x^{2m}}{(2m)!} \int_{\epsilon^{-k}}^\infty T^{\frac{\nu-(2m+1)}{k} -1} dT \right).
	\end{align*}
	The first term above is convergent for $\Re(\frac{\nu-1}{k})<1$, the second   for $\Re(\frac{\nu-3}{k})<1$ and so on.  The final  term is convergent for $\Re(\frac{\nu-(2m+1)}{k})<0$. These conditions hold simultaneously  if $\Re(\nu)<  \min\{ k+1,  k+3, \hdots,  {2m+1} \}$. Choosing $m$ large enough so that $2m+1 >k+1$, we see that   $I_1$ converges for $\Re(\nu) < k+1$. 
	Turning to $I_3$, we similarly have 
	\begin{align*}
		I_3& =  \int_{M}^\infty \cos\left(\frac{1}{t^k}\right)  \cos(x t)\frac{dt}{t^\nu} \\
		&= \int_{M}^\infty \cos(x t) t^{(-\nu+1)-1} dt - \frac{1}{2!}\int_{M}^\infty \cos(x t) t^{(-\nu-2k+1)-1} dt \\
		& + \frac{1}{4!}\int_{M}^\infty \cos(x t) t^{(-\nu-4k+1)-1} dt+ \cdots + O\left( \frac{1}{(2m)!} \int_{M}^\infty t^{(-\nu-2mk+1)-1} dt\right). 
	\end{align*}
	Similar to the discussion for $I_1$, the conditions for convergence are $\Re(\nu) >0, \Re(\nu)>-2k, \hdots, \Re(\nu)>-2mk+1$. As $k>0$, choosing $m$ sufficiently large yields that $I_3$ is convergent for $\Re(\nu) >0$. 
Combining the conditions for  convergence of $I_1$ and $I_3$, we have that   $H_z^{(k)}(x)$ converges for 
	$0< \Re(\nu) <k+1$, that is, $-1 < \Re(z) < k$, as needed. 
	
	Now let $x=0$. Employing the change of variable $t=u^{-1/k}$ in \eqref{defn_H_z^k}, we see that for $-1<\textup{Re}(z)<k-1$,
	\begin{align}
		H_z^{(k)}(0)=\frac{1}{k}\int_{0}^{\infty}u^{\frac{k-1-z}{k}-1}\cos(u)\, du
		=\frac{1}{k}\Gamma\left(\frac{k-1-z}{k}\right)\cos\left(\frac{\pi}{2}\left(\frac{k-1-z}{k}\right)\right),
	\end{align}
	where in the last step we used \eqref{Mellin_cos}. 
\end{proof}

\subsection{Differential equation satisfied by  $H_{z}^{(k)}(x)$
}\label{de}

Hardy \cite{hardyde} proved \eqref{Hardy} by finding a fourth order differential equation for $H_{0}^{(1)}(x)$. In what follows here, and in the Appendix, we adapt Hardy's method to derive the differential equation of order $2k+2$ for $H_{z}^{(k)}(x)$. This will play a crucial role in the second proof of Theorem \ref{H_z_sum of two_K_z}.
\begin{theorem}\label{hade}
Let $x\geq0$, $k\in\mathbb{N}$ and $k-2<\textup{Re}(z)<k-1$. The function $H_{z}^{(k)}(x)$ defined in  \eqref{defn_H_z^k} satisfies the homogeneous linear differential equation of order $2k+2$ given by
\begin{equation}\label{hade_eqn}
	x^{2}\frac{d^{2k+2}w}{dx^{2k+2}}+(2z+k+3)x\frac{d^{2k+1}w}{dx^{2k+1}}+(z+1)(z+k+1)\frac{d^{2k}w}{dx^{2k}}+(-1)^kk^2w=0.
\end{equation}
\end{theorem}
\begin{proof}
Replacing $t$ by $1/t$ in \eqref{defn_H_z^k},  one can see that $H_z^{(k)}(x)$ can be equivalently written in the form 
\begin{align*}
	H_z^{(k)}(x) = \int_{0}^\infty \cos\left( \frac{x}{t} \right) \cos\left(t^k \right) \frac{dt}{t^{z-k+2}}. 
\end{align*}
Let us define the following two functions:
\begin{align}
	J_k:=J_{k}(x,  s) & :=  \int_{0}^\infty \cos\left( \frac{x}{t} \right) \cos\left(t^k \right) \frac{dt}{t^{s}}, \label{integral jk} \\ 
	I_{k}:= I_{k}(x,  s) &:=  \int_{0}^\infty \sin\left( \frac{x}{t} \right) \sin\left(t^k \right) \frac{dt}{t^{s}}.  \label{integral ik}
\end{align}
Since $J_{k}(x,  z-k+2) = H_{z}^{(k)}(x)$, Theorem \ref{gen_Hardy_conv} implies that $J_k$ converges in $1-k<\textup{Re}(s)<2$. In a similar vein, one can show that $I_k$ converges in  $-k<\textup{Re}(s)<k+2$. 
First suppose that $0<\textup{Re}(s)<1$. Observe that for any $k\in\mathbb{N}$, both $J_k$ and $I_k$ converge in this strip.

From \cite[p.~433]{bromwich}, one can see that $J_k$ (and also $I_k$) are uniformly convergent with respect to $x$ in any interval $0<x_0\leq x\leq x_1$. Hence differentiation under the integral sign with respect to $x$ yields
\begin{align*}
	\frac{dJ_{k}}{dx} & = - \int_{0}^\infty \sin\left( \frac{x}{t} \right)   \cos\left(t^k \right) \frac{dt}{t^{s+1}}.  
\end{align*}
We cannot directly differentiate the above integral with respect to $x$ under the integral sign, for, the resulting integral becomes divergent. However, 
\begin{align}\label{can differentiate}
	\int_{0}^{\infty}\left(1-\cos(t^k)\right)\sin\left(\frac{x}{t}\right)\frac{dt}{t^{s+1}}&=	\frac{dJ_{k}}{dx}+\int_{0}^{\infty}\sin\left(\frac{x}{t}\right)\frac{dt}{t^{s+1}}\nonumber\\
	&=\frac{dJ_{k}}{dx}+x^{-s}\Gamma(s)\sin\left(\frac{\pi s}{2}\right),
\end{align}
where the last step resulted from the well-known identity
\begin{align}\label{wki}
\int_{0}^{\infty}u^{\xi-1}\sin(x u)\, du=x^{-\xi}\Gamma(\xi)\sin\left(\frac{\pi\xi}{2}\right)\hspace{8mm}(-1<\textup{Re}(\xi)<1).
\end{align}
(Note that in our case, we have assumed $0<\textup{Re}(s)<1$, thus permitting us to use the above evaluation.) Now we can differentiate \eqref{can differentiate} under the integral sign thereby obtaining
\begin{align*}
&\frac{d^{2}J_{k}}{dx^2}-x^{-s-1}\Gamma(s+1)\sin\left(\frac{\pi s}{2}\right)\nonumber\\
&=	\int_{0}^{\infty}\left(1-\cos(t^k)\right)\cos\left(\frac{x}{t}\right)\frac{dt}{t^{s+2}}\nonumber\\
&=-\frac{1}{x}\int_{0}^{\infty}\frac{\left(1-\cos(t^k)\right)}{t^s}\frac{d}{dt}\sin\left(\frac{x}{t}\right)\, dt\nonumber\\
&=-\frac{1}{x}\left[\frac{\left(1-\cos(t^k)\right)}{t^s}\sin\left(\frac{x}{t}\right)\right]_{0}^{\infty}+\frac{1}{x}\int_{0}^{\infty}\sin\left(\frac{x}{t}\right)\frac{d}{dt}\left(\frac{\left(1-\cos(t^k)\right)}{t^s}\right)\ dt\nonumber\\
&=\frac{k}{x}I_k(x, s+1-k)-\frac{s}{x}\frac{dJ_k}{dx}-x^{-s-1}\Gamma(s+1)\sin\left(\frac{\pi s}{2}\right),
\end{align*}
where, the last step follows from \eqref{wki} and the fact that the condition $0<\textup{Re}(s)<1$ renders the boundary terms zero. Hence
\begin{align}\label{2nd_order_J}
	\frac{d^{2}J_{k}}{dx^2} + \frac{s}{x} \frac{dJ_{k}}{dx} =  \frac{k}{x} I_{k}(x,  s+1-k). 
\end{align}
Similarly,  one can derive that 
\begin{align}\label{2nd_order_I}
	\frac{d^{2}I_{k}}{dx^2} + \frac{s}{x} \frac{dI_{k}}{dx} =  \frac{k}{x} J_{k}(x,  s+1-k). 
\end{align}
Multiply both sides of \eqref{2nd_order_J} by $x$ and then  differentiate the resulting equation with respect to $x$ to get
\begin{align}
	x\frac{d^{3}J_{k}}{dx^3} + (s+1)\frac{d^2J_{k}}{dx^2} = k\frac{d}{dx} I_{k}(x,  s+1-k). 
\end{align}
Now multiply both sides by $x$ and differentiate once again to see that 
\begin{align}\label{4th order_J}
	x \frac{d^{4}J_{k}}{dx^4} + (s+2)  \frac{d^{3}J_{k}}{dx^3} = -k I_{k}(x,  s+3-k),
\end{align}
where we used the fact $\frac{d^2I_{k}(x,s)}{dx^2} = -I_{k}(x, s+2)$. 
Differentiating  \eqref{4th order_J}  $k-3$ times with respect to $x$, we arrive at the equation  
\begin{align}\label{k+1_even_J}
	x \frac{d^{k+1}J_{k}}{dx^{k+1}} + (s+k-1)  \frac{d^{k}J_{k}}{dx^k} = (-1)^{\frac{k+3}{2}} k I_{k}(x,  s)
\end{align}
for $k$ odd, and at
\begin{align}\label{k+1_odd_J}
	x \frac{d^{k+1}J_{k}}{dx^{k+1}} + (s+k-1)  \frac{d^{k}J_{k}}{dx^k} = (-1)^{\frac{k+2}{2}} k \frac{d}{dx} I_{k}(x,  s-1)
\end{align}
for $k$ even. Similarly,

\begin{align}\label{k+1_even_I}
	x \frac{d^{k+1}I_{k}}{dx^{k+1}} + (s+k-1)  \frac{d^{k}I_{k}}{dx^k} &= (-1)^{\frac{k+3}{2}} k J_{k}(x,  s), \quad \text{for} ~~k~~\text{odd},\\
	x \frac{d^{k+1}I_{k}}{dx^{k+1}} + (s+k-1)  \frac{d^{k}I_{k}}{dx^k} &= (-1)^{\frac{k+2}{2}} k \frac{d}{dx} J_{k}(x,  s-1),  \quad \text{for} ~~k~~\text{even}. \label{k+1_odd_I}
\end{align}
Now applying the differential operator $
D:= x \frac{d^{k+1}}{dx^{k+1}} + (s+k-1)  \frac{d^{k}}{dx^k}
$ on both sides of \eqref{k+1_even_J}, utilizing \eqref{k+1_even_I} and substituting $s=z-k+2$,  we derive the differential equation in \eqref{hade_eqn} for $k$ odd. 
However,  when $k+1$ is odd,  replace $s$ by $s+1$ on both sides of \eqref{k+1_odd_J}, then apply the differential operator $  x \frac{d^{k}}{dx^{k}} + (s+k-1)  \frac{d^{k-1}}{dx^{k-1}}$ on both sides of the resulting equation, and then employ \eqref{k+1_odd_I}, to obtain
\begin{align*}
	&x^2  \frac{d^{2k+1}}{dx^{2k+1}}J_{k}(x,s+1) + x(2s+3k-1)  \frac{d^{2k}}{dx^{2k}}J_{k}(x,s+1)+ ( s+k-1)(s+2k-1) \frac{d^{2k-1}}{dx^{2k-1}} J_{k}(x,s+1)\\
	&=  k^2 \frac{d}{dx} J_{k}(x,  s-1). 
\end{align*}
Finally, differentiating the above equation, again, with respect to $x$ and observing that \newline $\frac{d^2 J_{k}(x,s-1)}{dx^2} = -J_{k}(x, s+1)$, we arrive at \eqref{hade_eqn} again upon replacing $s$ by $z-k+1$. 

Our assumption $0<\textup{Re}(s)<1$ implies $k-2<\textup{Re}(z)<k-1$. But from Theorem \ref{gen_Hardy_conv}, $H_z^{(k)}(x)$ itself converges for $-1<\textup{Re}(z)<k$. Hence \eqref{hade_eqn} holds for $k\in\mathbb{N}$ and all $z$ with  $k-2<\textup{Re}(z)<k-1$.

\end{proof}
\begin{remark}\label{other solution}
	We note in passing that \eqref{2nd_order_J}, \eqref{2nd_order_I} and the above analysis show that 
	\begin{equation}\label{hkzx-analogue}
		I_k(x, z-k+2)=\int_{0}^\infty \sin\left( \frac{x}{t} \right) \sin\left(t^k \right) \frac{dt}{t^{z-k+2}}=\int_{0}^\infty  t^{z-k}\sin(x t)\sin\left(\frac{1}{t^k}\right)\, dt
	\end{equation}
also satisfies the same differential equation given by \eqref{hade_eqn}.
\end{remark}
\subsection{The auxiliary integral $K_{z}^{(k)}(x)$}\label{kkzxaux}

We derive some properties of the integral $K_z^{(k)}(x)$ defined in \eqref{line integral_K_z}. The first one expresses it in terms of a Meijer-G function.
\begin{theorem}\label{meijergk}
Let $k\in\mathbb{N}$. Let $X=\frac{1}{4} \left( \frac{x}{2k} \right)^{2k} $, where $|\arg(x)|\leq\pi/(2k)$. Then 
\begin{align}\label{meijergkeqn}
	K_z^{(k)}(x)= \frac{1}{ \sqrt{k}\,\, 2^{\frac{1+z}{k}}}  G_{0,2k+2}^{\,k+2,0} \!\left(  \,\begin{matrix} \{  \} \\ b_1^{'}, \cdots , b_{2k+2}^{'} \end{matrix} \; \Big| X  \right),
\end{align}	
where
\begin{align}\label{b_j^{'}}
	b_j^{'} = \begin{cases}
		\frac{j-1}{k},  & {\rm if} \,\,    1 \leq j \leq k,   \\
		\frac{k-1-z}{2 k},  & {\rm if} \,\,  j=k+1,  \\
		\frac{2k-1-z}{2 k},  & {\rm if} \,\,  j=k+2,  \\
		\frac{4k-2j+5}{2k}  & {\rm if} \,\,  k+3 \leq j \leq 2k+2, \\
	\end{cases}
\end{align}
and where we additionally assume $-1<\textup{Re}(z)<k$ if $|\arg(x)|=\pi/(2k)$.
\end{theorem}
\begin{proof}
Invoke the variant of Euler's reflection  formula, namely
	$\cos\left( \frac{\pi w}{2} \right) = \frac{\pi}{\Gamma\left( \frac{1-w}{2} \right) \Gamma\left( \frac{1+w}{2} \right)  }$, the duplication formula $\frac{\Gamma(s)}{\Gamma\left( \frac{1+s}{2}  \right)} = \frac{  \Gamma\left( \frac{s}{2} \right) 2^{s-1} }{\sqrt{\pi}} $ in \eqref{line integral_K_z} and then replace $s$ by $2ks$ so as to get upon simplification
	\begin{align}
	K_z^{(k)}(x) = \sqrt{\pi}  \frac{1}{2 \pi i } \int_{\left(\frac{c}{2k}\right)} \frac{ \Gamma\left( ks \right) \Gamma\left( 2s + \frac{k-1-z}{k} \right) }{\Gamma\left( \frac{1}{2}-ks \right)} \left( \frac{2}{x} \right)^{2ks} \, ds. 
\end{align}
Again use the duplication formula for $\Gamma\left( 2s + \frac{k-1-z}{k} \right)$ followed by the Gauss multiplication formula  \cite[p.~52]{temme}
\begin{equation}\label{gmf}
	\prod_{j=1}^{m}\Gamma\left(w+\frac{j-1}{m}\right)=(2\pi)^{\frac{1}{2}(m-1)}m^{\frac{1}{2}-mw}\Gamma(mw)\hspace{8mm}(m\in\mathbb{N})
\end{equation}
for $\Gamma(ks)$ and $\Gamma\left( \frac{1}{2}-ks \right)$ to arrive at
\begin{align*}
	K_z^{(k)}(x) & =  \frac{1}{ \sqrt{k}\,\, 2^{\frac{1+z}{k}}}   \frac{1}{2 \pi i } \int_{\left(\frac{c}{2k}\right)}  \frac{ \prod_{j=1}^{k} \Gamma\left( s + \frac{j-1}{k} \right)  \Gamma\left( s + \frac{k-1-z}{2 k} \right)  \Gamma\left( s + \frac{2k-1-z}{2 k} \right)}{ \prod_{j=1}^{k}\Gamma\left(  \frac{2j-1}{2k} -s \right) }   \left(\frac{2k}{x}\right)^{2ks} 4^s \, ds \nonumber \\
	& =  \frac{1}{ \sqrt{k}\,\, 2^{\frac{1+z}{k}}}   \frac{1}{2 \pi i } \int_{\left(\frac{c}{2k}\right)} \frac{ \prod_{j=1}^{k} \Gamma\left( s + \frac{j-1}{k} \right)  \Gamma\left( s + \frac{k-1-z}{2 k} \right)  \Gamma\left( s + \frac{2k-1-z}{2 k} \right)  }{\prod_{j=k+3}^{2k+2} \Gamma\left( \frac{2j-2k-5}{2k}-s  \right)} X^{-s}  \, ds,
\end{align*}
where $X=\frac{1}{4} \left( \frac{x}{2k} \right)^{2k} $.   Replace $s$ by $-s$ to obtain
\begin{align}\label{before_applying Meijer G}
	K_z^{(k)}(x)=  \frac{1}{ \sqrt{k}\,\, 2^{\frac{1+z}{k}}}   \frac{1}{2 \pi i } \int_{\left(-\frac{c}{2k}\right)} \frac{ \prod_{j=1}^{k} \Gamma\left(  \frac{j-1}{k}-s \right)  \Gamma\left(  \frac{k-1-z}{2 k} -s \right)  \Gamma\left(  \frac{2k-1-z}{2 k} -s \right)  }{\prod_{j=k+3}^{2k+2} \Gamma\left( \frac{2j-2k-5}{2k}+s  \right)} X^{s}  \, ds.
\end{align}
Comparing this with  the definition \eqref{MeijerG}  of the Meijer $G$-function, we see that $m =k+2$,  $n=p=0$,  $q=2k+2$ 
and the $b_j^{'}$ are as defined in \eqref{b_j^{'}}.

One can check that $\sum_{j=1}^{2k+2} b_j^{'} = 1+k- \frac{1+z}{k}$.  Since $p+q < 2(m+n)$, the integral  representation  \eqref{before_applying Meijer G} of $K_z^{(k)}(x)$ converges absolutely for $|\arg(X)|<\pi$, that is, for $|\arg(x)|<\pi/(2k)$, and can be expressed as the Meijer $G$-function given on the right-hand side of \eqref{meijergkeqn}.

Now let $|\arg(x)|=\pi/(2k)$. Then we show that $K_{z}^{(k)}(x)$ is conditionally convergent, provided $\max\{ 0,  1 -k +  \Re(z) \} < \Re(s)=c\leq\frac{1+\textup{Re}(z)}{k+1}$. We will prove this in the case $\arg(x)=\pi/(2k)$. The result can be similarly obtained in the other case, that is, when $\arg(x)=-\pi/(2k)$.
Let $x=re^{i\pi/(2k)}, r>0$. Then
\begin{align*}
	K_{z}^{(k)}(x)&=\frac{1}{2 \pi i k} \int_{(c)} \Gamma(s) \cos\left( \frac{\pi s}{2}  \right)  \Gamma\left( \frac{s-1-z}{k} +1 \right)e^{-\frac{i\pi s}{2k}}r^{-s} \, ds\nonumber\\
	&=\frac{1}{2 \pi k} \left[\int_{-\infty}^{-T}+\int_{-T}^{T}+\int_{T}^{\infty}\right] \Gamma(c+it) \cos\left( \frac{\pi}{2}(c+it)  \right)  \Gamma\left( \frac{c+it-1-z}{k} +1 \right)e^{-\frac{i\pi(c+it)}{2k}}\frac{dt}{r^{c+it}},
\end{align*}
where $T>\textup{Im}(z)$ is large enough. The integral from $-T$ to $T$ is clearly finite. We now consider the integral from $T$ to $\infty$. 
From \cite[p.~73]{dav}, as $|s|\to\infty$ in the angle $-\pi+\delta<\arg(s)<\pi-\delta$, for any fixed $\delta>0$,
\begin{equation*}
	\Gamma(s)=\sqrt{2\pi}\exp{\left(\left(s-\frac{1}{2}\right)\log(s)-s\right)}\left(1+O\left(\frac{1}{|s|}\right)\right).
\end{equation*}
With $\theta=\tan^{-1}(t/c)$, this implies that as $t\to\infty$,
\begin{equation}\label{cfirst}
	\Gamma(c+it)=\exp{\left(\left(c-\frac{1}{2}\right)\log(t)-t\theta-c\right)}\cdot\exp{(if_1(t))}\left(1+O\left(\frac{1}{t}\right)\right),
\end{equation}
where $f_1(t)=t\log(t)+\theta\left(c-\frac{1}{2}\right)-t$. Next,
\begin{equation}\label{csecond}
	\cos\left(\frac{\pi (c+it)}{2}\right)=\frac{1}{2}\left(\exp{\left(\frac{i\pi}{2}(c+it)\right)}+\exp{\left(-\frac{i\pi}{2}(c+it)\right)}\right).
\end{equation}
Let $c_1=1+\frac{1}{k}(c-1-\textup{Re}(z))$ and $t_1=\frac{1}{k}(t-\textup{Im}(z))$. Clearly, $c_1>0$ and $t_1>0$. With $\theta_1=\tan^{-1}(t_1/c_1)$, this implies that as $t\to\infty$,
\begin{equation}\label{cthird}
\Gamma(c_1+it_1)=\exp{\left(\left(c_1-\frac{1}{2}\right)\log(t_1)-t_1\theta_1-c_1\right)}\cdot\exp{(if_2(t))}\left(1+O\left(\frac{1}{t}\right)\right),
\end{equation}
where $f_2(t)=t_1\log(t_1)+\theta_1\left(c_1-\frac{1}{2}\right)-t_1$. Also,
\begin{equation}\label{cfourth}
\exp{\left(-\frac{i\pi s}{2k}\right)}=	\exp{\left(\frac{\pi}{2k}(\textup{Im}(z)-ic)\right)}\exp{\left(\frac{\pi t_1}{2}\right)}.
\end{equation} 
Moreover, since $\tan^{-1}(x)+\tan^{-1}(1/x)=\pi/2$ for $x>0$, we find that as $t\to\infty$,
\begin{align}\label{cfifth}
	\theta&=\frac{\pi}{2}-\frac{c}{t}+O\left(\frac{1}{t^2}\right),\nonumber\\
	\theta_1&=\frac{\pi}{2}-\frac{c_1}{t_1}+O\left(\frac{1}{t_1^2}\right).
\end{align}
Now let $f(t)=f_1(t)+f_2(t)$. Hence using from \eqref{cfirst}-\eqref{cfifth} and observing that  $f'(t)\sim\log(t)$ and $f''(t)\ll1/t$ as $t\to\infty$, we see that
\begin{align*}
&\int_{T}^{\infty}\Gamma(c+it) \cos\left( \frac{\pi}{2}(c+it)  \right)  \Gamma\left( \frac{c+it-1-z}{k} +1 \right)e^{-\frac{i\pi(c+it)}{2k}}\frac{dt}{r^{c+it}}\nonumber\\
&=a_{z, k, r, c}\int_{T}^{\infty}t^{(c-\frac{1}{2})+(c_1-\frac{1}{2})}\exp{(i(f(t)-t\log(r)))}\left(1+O\left(\frac{1}{t}\right)\right)\, dt\nonumber\\
&=a_{z, k, r, c}\left\{\left[\frac{t^{c+c_1-1}}{i(f'(t)-\log(r))}\exp{\left(i(f(t)-t\log(r))\right)}\right]_{T}^{\infty}+O\left(\int_{T}^{\infty}t^{c+c_1-2}\, dt\right)\right\},
\end{align*}
where $a_{z, k, r, c}$ is a constant. Here, in the last step, we performed integration by parts on the first integral. Thus the integral will be finite only if $c+c_1\leq1$, which implies that $c\leq(1+\textup{Re}(z))/(k+1)$. 
One can similarly show the existence of the integral from $-\infty$ to $-T$. This shows that $K_{z}^{(k)}(x)$ is conditionally convergent on the ray $\arg(x)=\pi/(2k)$, provided $\max\{ 0,  1 -k +  \Re(z) \} < \Re(s)=c\leq\frac{1+\textup{Re}(z)}{k+1}$. 
Proceeding as in the first part of the proof, we see that \eqref{meijergkeqn} holds for $|\arg(x)|=\pi/(2k)$ as well.

\end{proof}
\begin{remark}\label{kkzx-analytic in x}
 In $|\arg(x)|<\pi/(2k)$,  $K_z^{(k)}(x)$ is an analytic function of $x$ as can be seen from \cite[p.~30, Theorem 2.3]{temme} or, with the help of Theorem \ref{meijergk}, from \cite[p.~618]{prud3}.
\end{remark}
For $x>0$, $K_z^{(k)}(x)$ has a representation as an integral of a real variable given in the following theorem. This will be instrumental in proving Theorem \ref{H_z_sum of two_K_z}.
	\begin{theorem}\label{kkzxreal}
		Let $K_z^{(k)}(z)$ be defined in \eqref{line integral_K_z}. For $x\geq0$ and \textup{Re}$(z)<k$, we have
		\begin{align}\label{defn_K_z}
			K_z^{(k)}(x) = \int_{0}^\infty \exp\left( -\frac{1}{t^k} \right) \cos( xt) \frac{\, dt}{t^{k-z}}.
		\end{align}
	\end{theorem}
\begin{proof}
We first prove the above result for Re$(z)<k-1/2$ and then extend it to Re$(z)<k$ by analytic continuation. As mentioned in Remark \ref{parsevalcos}, we can take $g(t)=\cos(xt)$ in Theorem \ref{parseval_extension_tuan}. Then, from \eqref{Mellin_cos}, $G(s)=\Gamma(s)\cos\left(\frac{\pi s}{2}\right)x^{-s}$. Moreover, if $f(t)=t^{z-k}\exp{\left(-t^{-k}\right)}$, then, with the change of variable $t=u^{-1/k}$, it is easy to see that
	\begin{equation*}
	F(s)=\frac{1}{k}\Gamma\left(\frac{k-s-z}{k}\right)\hspace{7mm}(\textup{Re}(s)<k-\textup{Re}(z)).
	\end{equation*}
Now for $f$ to be in $\mathfrak{M}^{-1}(L)$, we must have $1/2<k-\textup{Re}(z)$. This explains the condition 
$\textup{Re}(z)<k-1/2$ that we initially need to impose.

Invoking Theorem \ref{parseval_extension_tuan} with the above choices of $k$ and $f$, we see that 
\begin{align*}
\int_{0}^\infty \exp\left( -\frac{1}{t^k} \right) \cos( xt) \frac{\, dt}{t^{k-z}}&=\frac{1}{2 \pi i } \int_{\left(\frac{1}{2}\right)} \Gamma(s) \cos\left( \frac{\pi s}{2}  \right)  \Gamma\left( \frac{s-1-z}{k} +1 \right) \frac{\, ds}{k x^s}\nonumber\\
&=K_z^{(k)}(x),
\end{align*}
where in the last step we used Remark \ref{extension_s} and the fact that the line $[1/2-i\infty, 1/2+i\infty]$ lies in the half-plane $\max\{ 0,  1 -k +  \Re(z) \} < \Re(s)$.  This proves \eqref{defn_K_z} for Re$(z)<k-1/2$.

Next, using the techniques of Theorem \ref{gen_Hardy_conv}, it is easy to see that the right-hand side of \eqref{defn_K_z} is convergent in Re$(z)<k$ and is analytic in this region. Along with the discussion following \eqref{line integral_K_z}, we see that by analytic continuation, \eqref{defn_K_z} holds for Re$(z)<k$.
\end{proof}	

 Koshliakov \cite[Equation (9)]{koshlyakov-1936a} has shown that for $n\in\mathbb{N}$ and $x>0$,
		\begin{equation*}
			K_0\left(4\pi e^{\frac{i\pi}{4}}\sqrt{nx}\right)+K_0\left(4\pi e^{-\frac{i\pi}{4}}\sqrt{nx}\right)=\frac{1}{2\pi i}\int_{(\frac{3}{2})}\frac{\Gamma^2(s)\cos\left(\frac{\pi s}{2}\right)}{(2\pi)^{2s}(nx)^s}\, ds.
		\end{equation*}
		An easy application of the residue theorem after shifting the line of integration from Re$(s)=3/2$ to Re$(s)=c, 0<c<1$, and then comparing with \eqref{line integral_K_z} yields
		\begin{equation*}
			K_0^{(1)}(4\pi^2nx)=K_0\left(4\pi e^{\frac{i\pi}{4}}\sqrt{nx}\right)+K_0\left(4\pi e^{-\frac{i\pi}{4}}\sqrt{nx}\right).
		\end{equation*}	
	More generally, we have
	\begin{theorem}
 Let $K_{\nu}(\xi)$ be the modified Bessel function of the second kind defined in \eqref{kbes}. For $x>0$, $n\in\mathbb{N}$ and \textup{Re}$(z)<1$,
		\begin{equation}\label{K1}
		K_z^{(1)}(4\pi^2nx)=(2\pi\sqrt{nx})^{-z}\left\{e^{-\frac{i\pi z}{4}}K_z\left(4\pi e^{\frac{i\pi}{4}}\sqrt{nx}\right)+e^{\frac{i\pi z}{4}}K_z\left(4\pi e^{-\frac{i\pi}{4}}\sqrt{nx}\right)\right\}.
	\end{equation}
	\end{theorem}
\begin{proof}	
	For Re$(s)>\pm$ Re$(z)$ and $\textup{Re}(a)>0$, we have
	\begin{equation*}
		\int_{0}^{\infty}x^{s-1}K_{z}(ax)\, dx=2^{s-2}a^{-s}\Gamma\left(\frac{s-z}{2}\right)\Gamma\left(\frac{s+z}{2}\right).
	\end{equation*} 
Replace $s$ by $2s$, $x$ by $\sqrt{x}$ and then let $a=4\pi e^{\pm i\pi/4}\sqrt{n}$ so that for Re$(s)>\pm$ Re $\left(\frac{z}{2}\right)$,
\begin{equation*}
	\int_{0}^{\infty}x^{s-1}K_z\left(4\pi e^{\pm i\pi/4}\sqrt{nx}\right)\, dx=2^{2s-1}\left(4\pi e^{\pm i\pi/4}\sqrt{n}\right)^{-2s}\Gamma\left(s-\frac{z}{2}\right)\Gamma\left(s+\frac{z}{2}\right).
\end{equation*}
Therefore,
\begin{align*}
	&\int_{0}^{\infty}x^{s-1}\left\{e^{\frac{i\pi z}{4}}K_z\left(4\pi e^{i\pi/4}\sqrt{nx}\right)+e^{-\frac{i\pi z}{4}}K_z\left(4\pi e^{-i\pi/4}\sqrt{nx}\right)\right\}\, dx\nonumber\\
	&=(4\pi^2n)^{-s}\Gamma\left(s-\frac{z}{2}\right)\Gamma\left(s+\frac{z}{2}\right)\cos\left(\frac{\pi}{2}\left(\frac{z}{2}-s\right)\right).
\end{align*}
Hence by the Mellin inversion theorem \cite[p.~341]{Mcla}, we have with $c=$Re$(s)>\pm$ Re $\left(\frac{z}{2}\right)$,
\begin{align*}
&\frac{1}{2\pi i}\int_{(c)}\Gamma\left(s-\frac{z}{2}\right)\Gamma\left(s+\frac{z}{2}\right)\cos\left(\frac{\pi}{2}\left(\frac{z}{2}-s\right)\right)(4\pi^2nx)^{-s}\, ds\nonumber\\
&=e^{\frac{i\pi z}{4}}K_z\left(4\pi e^{i\pi/4}\sqrt{nx}\right)+e^{-\frac{i\pi z}{4}}K_z\left(4\pi e^{-i\pi/4}\sqrt{nx}\right).
\end{align*}
Now replace $z$ by $-z$ in the above equation and use the well-known fact $K_{-z}(y)=K_z(y)$ so as to get for Re$(s)>\pm$ Re $\left(\frac{z}{2}\right)$, 
\begin{align}\label{almost_there}
	&\frac{1}{2\pi i}\int_{(c)}\Gamma\left(s-\frac{z}{2}\right)\Gamma\left(s+\frac{z}{2}\right)\cos\left(\frac{\pi}{2}\left(\frac{z}{2}+s\right)\right)(4\pi^2nx)^{-s}\, ds\nonumber\\
	&=e^{-\frac{i\pi z}{4}}K_z\left(4\pi e^{i\pi/4}\sqrt{nx}\right)+e^{\frac{i\pi z}{4}}K_z\left(4\pi e^{-i\pi/4}\sqrt{nx}\right).
\end{align}
Finally, replace $s$ by $s-z/2$ in \eqref{almost_there} so that for $\max\left\{0, \textup{Re}(z)\right\}<c'=\textup{Re}(s)$,
\begin{align*}
	&\frac{1}{2\pi i}(4\pi^2nx)^{z/2}\int_{(c')}\Gamma(s-z)\Gamma(s)\cos\left(\frac{\pi s}{2}\right)(4\pi^2nx)^{-s}\, ds\nonumber\\
	&=e^{-\frac{i\pi z}{4}}K_z\left(4\pi e^{i\pi/4}\sqrt{nx}\right)+e^{\frac{i\pi z}{4}}K_z\left(4\pi e^{-i\pi/4}\sqrt{nx}\right).
\end{align*}
Upon adding the restriction Re$(s)<1$ and observing \eqref{line integral_K_z}, this leads us to \eqref{K1}.
\end{proof}

\subsection{Relation between $H_{z}^{(k)}(x)$ and $K_{z}^{(k)}(x)$: Proof of Theorem \ref{H_z_sum of two_K_z}}
The integrals $H_{z}^{(k)}(x)$ and $K_{z}^{(k)}(x)$  defined in \eqref{defn_H_z^k} and \eqref{line integral_K_z} respectively are related by means of the  identity in Theorem \ref{H_z_sum of two_K_z} of which we now give a proof. Note also that neither Parseval's formula \eqref{Parseval1} nor its extension in \eqref{parseval_extension_tuan_eqn} is capable of handling the integral $H_z^{(k)}(x)$ since the integrand of $H_z^{(k)}(x)$ cannot be decomposed in any way into functions $f$ and $g$ so that they satisfy the conditions of Theorems \ref{Parseval} or  \ref{parseval_extension_tuan}. 


We prove Theorem \ref{H_z_sum of two_K_z} using the theory of linear differential equations \cite[p.~21, Section 6]{coddington-levinson}. But before that, we need the following lemma.
\begin{lemma}\label{coeff_de}
Let $X_{2k+2}=\{b_1, b_2, \cdots, b_{2k+2}\}$ where $b_j, 1\leq j\leq 2k+2,$ are defined in \eqref{coefficients_bj}. Let $e_{\ell}(X_{2k+2})$ denote the elementary symmetric polynomial defined in \eqref{esp} and let $S(n, k)$ denote the Stirling numbers of the second kind defined in Section \ref{prelim}. Then
\begin{align}\label{coeff_de_eqn}
	\sum_{j=0}^{2k+2-m}(-2k)^je_j(X_{2k+2})S(2k+2-j, m)=\begin{cases}
		1,\hspace{1mm}\textup{if}\hspace{1mm}m=2k+2,\\
		2z+k+3,\hspace{1mm}\textup{if}\hspace{1mm}m=2k+1,\\
		(z+1)(z+k+1),\hspace{1mm}\textup{if}\hspace{1mm}m=2k,\\
		0,\hspace{1mm}\textup{if}\hspace{1mm}1\leq m\leq 2k-1.
	\end{cases}
\end{align}
\end{lemma}
\begin{proof}
	The result is trivial for $m=2k+2$ as $S(2k+2, 2k+2)=1$. 
	Let $m=2k+1$. Using 
	\begin{equation}\label{e1}
	e_1(X_{2k+2})= \sum_{j=1}^{2k+2} b_j= 1+k - \frac{1+z}{k},
	\end{equation}
	 as well as the result \cite[p.~227]{comtet} 
	 \begin{align}\label{snn-1}
	 S(n, n-1)=n(n-1)/2,
	 \end{align}
	  in the second step, we see that
\begin{align*}
\sum_{j=0}^{1}(-2k)^je_j(X_{2k+2})S(2k+2-j, 2k+1)&=S(2k+2, 2k+1)-2ke_1(X_{2k+2})S(2k+1, 2k+1)\nonumber\\
&=(k+1)(2k+1)\!-2k\Big(\!1+k-\frac{1+z}{k}\Big) =2z+k+3.
\end{align*}
Now let $m=2k$. Then substituting \eqref{e1}, \eqref{snn-1} and the identity \cite[p.~227]{comtet},
$
	S(n, n-2)=\frac{1}{24}n(n-1)(n-2)(3n-5), 
	$
in the second step below, we have
\begin{align}\label{m=2k}
	&\sum_{j=0}^{2}(-2k)^je_j(X_{2k+2})S(2k+2-j, 2k)\nonumber\\
	&=S(2k+2, 2k)-2ke_1(X_{2k+2})S(2k+1, 2k)+4k^2e_2(X_{2k+2})S(2k, 2k)\nonumber\\
	&=	\frac{1}{24}(2k+2)(2k+1)(2k)(6k+1)-2k\left(1+k - \frac{1+z}{k}\right)(k(2k+1))+4k^2e_2(X_{2k+2}).
\end{align}
Using \eqref{coefficients_bj}, we now show
\begin{equation}\label{e2}
	e_2(X_{2k+2})= \!\!\!\!\!\! \sum_{1\leq i<j\leq 2k+2}\!\!\!\!\!\!
	b_ib_j=\frac{1}{24k^2}\left(12k^4+16k^3-3k^2(7+8z)-k(7+6z)+6(1+z)^2\right).
\end{equation}
To that end, observe that
\begin{align}\label{e2-calc}
	\sum_{1\leq i<j\leq 2k+2}b_ib_j&=\sum_{2\leq i<j\leq k}\frac{(i-1)(j-1)}{k^2}+b_{k+1}\sum_{i=2}^{2k+2}b_i+\left(\sum_{i=2}^{k}\frac{i-1}{k}\right)\left(\sum_{j=k+2}^{2k+1}\left(2+\frac{3-2j}{k}\right)\right)\nonumber\\
	&\quad+\sum_{k+2\leq i<j\leq 2k+1}\left(2+\frac{3-2i}{k}\right)\left(2+\frac{3-2j}{k}\right)+b_{2k+2}\sum_{i=2}^{2k+1}b_i-b_{k+1}b_{2k+2},
	\end{align}
where the last expression on the right was subtracted since it was considered twice, once in $b_{k+1}\sum_{i=2}^{2k+2}b_i$, and again in $b_{2k+2}\sum_{i=2}^{2k+1}b_i$. Now it can be seen that
\begin{align}\label{evaluations}
\sum_{2\leq i<j\leq k}\frac{(i-1)(j-1)}{k^2}&=\frac{1}{24k}(3k^3-10k^2+9k-2),\nonumber\\
b_{k+1}\sum_{i=2}^{2k+2}b_i&=\left(\frac{1}{2}-\frac{1+z}{2k}\right)\left(\frac{1}{2}+k-\frac{1+z}{2k}\right),\nonumber\\
\sum_{i=2}^{k}\frac{i-1}{k}\sum_{j=k+2}^{2k+1}\left(2+\frac{3-2j}{k}\right)&=\frac{k(k-1)}{4},\nonumber\\
b_{2k+2}\sum_{i=2}^{2k+1}b_i&=\left(1-\frac{1+z}{2k}\right)\left(k-\frac{1+z}{2k}\right),\nonumber\\
b_{k+1}b_{2k+2}&=\left(\frac{1}{2}-\frac{1+z}{2k}\right)\left(1-\frac{1+z}{2k}\right).
\end{align}
Substituting \eqref{evaluations} in \eqref{e2-calc} and simplifying, we arrive at \eqref{e2}, and substituting \eqref{e2}, in turn, in \eqref{m=2k} gives \eqref{coeff_de_eqn} for $m=2k$ upon simplification.

It remains to show the validity of \eqref{coeff_de_eqn} in the case $1\leq m\leq 2k-1$. To that end, we represent $S(2k+2-j,m)$ using the identity \cite[p.~204, Theorem A]{comtet}
\begin{align*}
S(\ell, m)=\frac{1}{m!}\sum_{n=1}^{m}(-1)^{m-n}\binom{m}{n}n^{\ell},
\end{align*}
and then interchange the order of summation consequently obtaining
\begin{align*}
	\!\sum_{j=0}^{2k+2-m} \!\!\!\! (-2k)^je_j(X_{2k+2})S(2k+2-j, m)&\!=\!\frac{(-1)^m}{m!}\sum_{n=1}^m(-1)^n\binom{m}{n}n^{2k+2}\sum_{j=0}^{2k+2}\left(\frac{-2k}{n}\right)^je_j(X_{2k+2})\nonumber\\
	&=\frac{(-1)^m}{m!}\sum_{n=1}^m(-1)^n\binom{m}{n}n^{2k+2}\prod_{j=1}^{2k+2}\left(1-\frac{2k}{n}b_j\right).
\end{align*}
Now it is important to observe that for any $1\leq n\leq m$ and any $1\leq m\leq 2k-1$, the product $\prod_{j=1}^{2k+2}\left(1-\frac{2k}{n}b_j\right)$ equals zero since there is precisely one factor in the product which vanishes. 

Indeed, for any odd $n$ of the form $n=2\ell-1$, where $1\leq\ell\leq k$, we have  $1-\frac{2k}{n}b_{2k-\ell+2}=0$ as can be seen from \eqref{coefficients_bj}. Similarly, for any even $n$ of the form $n=2\ell$, where $1\leq\ell\leq k-1$, the expression $1-\frac{2k}{n}b_{\ell+1}=0$. This proves \eqref{coeff_de_eqn} in the remaining case $1\leq m\leq 2k-1$ and completes the proof.
\end{proof}
\subsubsection{Proof of Theorem \ref{H_z_sum of two_K_z}:}
We initially prove the result for $k-2<\textup{Re}(z)<k-1$ and later extend it by analytic continuation to $-1<\textup{Re}(z)<k$. Define
	\begin{align}\label{inverse Mellin_H_z}
	A_z^{(k)}(x):= \frac{1}{2} \left\{ \exp\left(\frac{i \pi(k-1-z)}{2 k}  \right) K_z^{(k)}\left( e^{-\frac{i \pi}{2k}} x  \right) + \exp\left(\frac{- i \pi(k-1-z)}{2 k}  \right) K_z^{(k)}\left( e^{\frac{i \pi}{2k}} x  \right)     \right\}.
	\end{align}
We first show that the identity in \eqref{H_z in terms K_z} holds for $x=0$, that is, 
\begin{equation}\label{H=Azero}
H_z^{(k)}(0)=A_z^{(k)}(0)=\frac{\pi}{\sqrt{k}2^{\frac{1+z}{k}}}G_{0,2k+2}^{\,k+1,0} \!\left(  \,\begin{matrix} \{  \} \\ b_1, \cdots , b_q \end{matrix} \; \Big| 0   \right)=\frac{1}{k}\Gamma\left(\frac{k-1-z}{k}\right)\cos\left(\frac{\pi}{2}\left(\frac{k-1-z}{k}\right)\right).
\end{equation}
From \eqref{Mellin_cos}, for $-1<\textup{Re}(z)<k-1$,
\begin{align}\label{Hzero}
H_z^{(k)}(0)
=\frac{1}{k}\Gamma\left(\frac{k-1-z}{k}\right)\cos\left(\frac{\pi}{2}\left(\frac{k-1-z}{k}\right)\right).
\end{align}
Invoking Theorem \ref{kkzxreal} and employing again the change of variable $t=u^{-1/k}$, we see that for Re$(z)<k-1$,
\begin{align*}
K_z^{(k)}(0)=\frac{1}{k}\int_{0}^{\infty}u^{\frac{k-1-z}{k}-1}e^{-u}\, du=\frac{1}{k}\Gamma\left(\frac{k-1-z}{k}\right),
\end{align*}
which when substituted in \eqref{inverse Mellin_H_z} yields
\begin{equation}\label{Azero}
A_z^{(k)}(0)=\frac{1}{k}\Gamma\left(\frac{k-1-z}{k}\right)\cos\left(\frac{\pi}{2}\left(\frac{k-1-z}{k}\right)\right).
\end{equation}
Next, we prove the last equality of \eqref{H_z in terms K_z}. To that end, note that by Slater's theorem\footnote{In general, Meijer $G$-function has a complicated branch point at $x=0$. However, since $b_1=0$, $b_j>0$ for $2\leq j\leq k$, and $\textup{Re}(b_{k+1})>0$ because of the condition $k-2<\textup{Re}(z)<k-1$, Slater's theorem is applicable, thereby giving the non-trivial value of the Meijer $G$-function.} \eqref{Slater},
\begin{align*}
G_{0,2k+2}^{\,k+1,0} \!\left(  \,\begin{matrix} \{  \} \\ b_1, \cdots , b_{2k+2} \end{matrix} \; \Big| 0   \right)=\frac{\prod_{j=2}^{k+1}\Gamma(b_j-b_1)}{\prod_{j=k+2}^{2k+2}\Gamma(1+b_1-b_j)}
&=\frac{\prod_{j=1}^{k}\Gamma\left(\frac{j}{k}\right)\Gamma\left(\frac{1}{2}-\frac{1+z}{2k}\right)}{\prod_{j=1}^{k}\Gamma\left(\frac{1}{2k}+\frac{j-1}{k}\right)\Gamma\left(\frac{1+z}{2k}\right)}\nonumber\\
&=\frac{1}{\sqrt{k\pi}}\frac{\Gamma\left(\frac{1}{2}-\frac{1+z}{2k}\right)}{\Gamma\left(\frac{1+z}{2k}\right)},
\end{align*}
where the last step results from applying \eqref{gmf} twice, once with $w=1/k$ and $m=k$, and the second time with $w=1/(2k)$ and $m=k$.
This implies that 
\begin{align}\label{Gzero}
	\frac{\pi}{\sqrt{k}2^{\frac{1+z}{k}}}G_{0,2k+2}^{\,k+1,0} \!\left(  \,\begin{matrix} \{  \} \\ b_1, \cdots , b_{2k+2} \end{matrix} \; \Big| 0   \right)&=\frac{\sqrt{\pi}}{k 2^{\frac{1+z}{k}}}\frac{\Gamma\left(\frac{1}{2}-\frac{1+z}{2k}\right)}{\Gamma\left(\frac{1+z}{2k}\right)}\nonumber\\
	&=\frac{1}{k}\Gamma\left(\frac{k-1-z}{k}\right)\cos\left(\frac{\pi}{2}\left(\frac{k-1-z}{k}\right)\right),
\end{align}
as can be seen by specializing the identity \cite[p.~73]{dav},
${\Gamma\left(\frac{s}{2}\right)} {\Gamma\left(\frac{1}{2}-\frac{s}{2}\right)}^{-1}=2^{1-s}\pi^{-1/2}\Gamma(s)\cos\left(\frac{\pi s}{2}\right),
$
with $s=1-(1+z)/k$. From \eqref{Hzero}, \eqref{Azero} and \eqref{Gzero}, we have proven \eqref{H=Azero} in totality. 

We next show that for $x>0$,
\begin{equation}\label{A=G}
	A_z^{(k)}(x)=\frac{\pi}{\sqrt{k}2^{\frac{1+z}{k}}}G_{0,2k+2}^{\,k+1,0} \!\left(  \,\begin{matrix} \{  \} \\ b_1, \cdots , b_q \end{matrix} \; \Big| \frac{1}{4} \left(\frac{x}{2k} \right)^{2k}   \right).
\end{equation}
To that end, using \eqref{line integral_K_z} and making a note of the discussion following it, it is easy to see that for $\max\{ 0,  1 -k +  \Re(z) \} < \Re(s)=c\leq\frac{1+\textup{Re}(z)}{k+1}$ ,
	\begin{align}\label{inverse Mellin_H_z1}
A_z^{(k)}(x)= \frac{1}{ 2\pi i} \int_{(c)} \Gamma(s) \cos\left( \frac{\pi s}{2}  \right)  \Gamma\left( \frac{s-1-z}{k} +1 \right)  \cos\left( \frac{\pi}{2}  \left(  \frac{s-1-z}{k} +1  \right)  \right) \frac{\, ds}{k x^s}.
\end{align}
Proceeding along the similar lines as in the proof of Theorem \ref{meijergk}, we see that for \newline $\max\{ 0,  \frac{1 -k +  \Re(z)}{2k} \} < \Re(s)=c'\leq\frac{1+\textup{Re}(z)}{2k (k+1)}$,
\begin{align*}
	A_z^{(k)}(x)=\frac{1}{\sqrt{k}}2^{\frac{k}{2}-\frac{1+z}{k}}\pi^{\frac{k+2}{2}}\frac{1}{2\pi i}\int_{(c')}\frac{\prod_{j=1}^{k}\Gamma\left(s+\frac{j-1}{k}\right)\Gamma\left(\frac{1}{2}+s-\frac{1+z}{2k}\right)\left(4\left(\frac{2k}{x}\right)^{2k}\right)^{s}\, ds}{\prod_{j=1}^{k}\Gamma\left(-s+\frac{2j-1}{2k}\right)\Gamma\left(-s+\frac{1+z}{2k}\right)}.
\end{align*}
Now replace $s$ by $-s$ and use \eqref{MeijerG} and \eqref{coefficients_bj} to arrive at \eqref{A=G}. 

Our next task is to show that for $x>0$,
\begin{equation}\label{H=G}
	H_z^{(k)}(x)=\frac{\pi}{\sqrt{k}2^{\frac{1+z}{k}}}G_{0,2k+2}^{\,k+1,0} \!\left(  \,\begin{matrix} \{  \} \\ b_1, \cdots , b_{2k+2} \end{matrix} \; \Big| \frac{1}{4} \left(\frac{x}{2k} \right)^{2k}   \right).
\end{equation} 
This is achieved in two parts. In the first, we show that the right-hand side of \eqref{H=G} satisfies the same differential equation as that of $H_z^{(k)}(x)$, namely, the one in \eqref{hade_eqn}. This is what we do next. The second part of the proof, which is concerned with actually showing that the two sides of \eqref{H=G} are equal, is given in the Appendix. 

It is well-known \cite[p.~417]{NIST} that $w\!=\!G_{p,q}^{\,m,n} \!\left( \begin{matrix} a_1,\cdots , a_p \\ b_1, \cdots , b_q \end{matrix}  \Big| \xi  \! \right)$ satisfies the differential equation
\begin{equation*}
	\left((-1)^{p-m-n}\xi (\theta-a_1+1)\cdots(\theta-a_p+1)-(\theta-b_1)\cdots(\theta-b_q)\right)w=0,
\end{equation*}
where $\theta=\xi\frac{d}{d\xi}$. With $\xi=\frac{1}{4} \left(\frac{x}{2k} \right)^{2k} $, this implies that $G_{0,2k+2}^{\,k+1,0} \!\left(  \,\begin{matrix} \{  \} \\ b_1, \cdots , b_{2k+2} \end{matrix} \; \Big| \frac{1}{4} \left(\frac{x}{2k} \right)^{2k}   \right)$ satisfies the differential equation
\begin{align*}
	\bigg[&\left(\xi\frac{d}{d\xi}\right)^{2k+2}-e_1(X_{2k+2})\left(\xi\frac{d}{d\xi}\right)^{2k+1}+e_2(X_{2k+2})\left(\xi\frac{d}{d\xi}\right)^{2k}-\cdots\nonumber\\
	&+(-1)^{2k+2}e_{2k+2}(X_{2k+2})+(-1)^k\xi\bigg]w=0,
\end{align*}
where $X_{2k+2}=\{b_1, b_2, \cdots, b_{2k+2}\}$ with $b_j$ defined in \eqref{coefficients_bj} and $e_\ell(X_{2k+2})$ is the elementary symmetric polynomial defined in \eqref{esp}, or, written more compactly, the differential equation
\begin{align*}
\left(\sum_{j=0}^{2k+2}(-1)^je_j(X_{2k+2})\left(\xi\frac{d}{d\xi}\right)^{2k+2-j}+(-1)^k\xi\right)w=0.
\end{align*}
Since $\xi=\frac{1}{4} \left(\frac{x}{2k} \right)^{2k} $ implies $\left(\xi\frac{d}{d\xi}\right)^{\ell}(w)=\left(\frac{x}{2k}\frac{d}{dx}\right)^{\ell}(w)$, the above differential equation, upon simplification, takes the form
\begin{align*}
	\left(x^{-2k}\sum_{j=0}^{2k+2}(-2k)^je_j(X_{2k+2})\left(x\frac{d}{dx}\right)^{2k+2-j}+(-1)^kk^2\right)w=0.
\end{align*}
Now employ the well-known identity \cite[p.~157]{comtet}
\begin{align}\label{s2}
	\left(x\frac{d}{dx}\right)^\ell(w)=\sum_{m=1}^{\ell}S(\ell, m)x^m\frac{d^mw}{dx^m},
\end{align}
where $S(\ell, m)$ denote the Stirling numbers of the second kind defined in Section \ref{prelim}, to write $\left(x\frac{d}{dx}\right)^{2k+2-j}$ as a sum and then interchange the order of summation while noting $S(0,m)=0$ to derive
\begin{align}\label{geqn}
	\left(\sum_{m=1}^{2k+2}x^{m-2k}\frac{d^m}{dx^m}	\sum_{j=0}^{2k+2-m}(-2k)^je_j(X_{2k+2})S(2k+2-j,m)+(-1)^kk^2\right)w=0.
\end{align}
Invoking Lemma \ref{coeff_de}, we are led to \eqref{hade_eqn}. This proves that both sides of \eqref{H=G} satisfy the same differential equation in \eqref{hade_eqn}. The remaining part of the proof of \eqref{H=G} now follows from the Appendix.

This completes the proof of Theorem \ref{H_z_sum of two_K_z} for $k-2<\textup{Re}(z)<k-1$. Since $K_z^{(k)}(x)$ is analytic in Re$(z)<k$ as can be seen from the discussion following \eqref{line integral_K_z} and $H_z^{(k)}(x)$ is analytic in $-1<\textup{Re}(z)<k$, by analytic continuation, the identity holds for $-1<\textup{Re}(z)<k$.
\qed
\subsection{Asymptotics of $H_{z}^{(k)}(x)$}
We first obtain the asymptotic behavior of $H_{z}^{(k)}(x)$ as $x\to0^{+}$.
\begin{theorem}\label{asym zero}
Let $H_{z}^{(k)}(x)$ be defined in \eqref{defn_H_z^k}. Then $H_{z}^{(k)}(x)=O(1)$ as $x\to0^{+}$.	
\end{theorem}
\begin{proof}
	Note that in \eqref{Hzero}, it was shown that $H_{z}^{(k)}(0)$ is a finite quantity. Using Abel's and Dirichlet's tests for uniform convergence of infinite integrals \cite{bromwich}, it can be seen that $H_{z}^{(k)}(x)$ is continuous at $x=0$ whence $H_{z}^{(k)}(x)=O(1)$ as $x\to0^{+}$. 
\end{proof}
The behavior of $H_{z}^{(k)}(x)$ as $x\to\infty$ is derived next.
\begin{theorem}\label{asym of H_z^k}
Let $y>0$, $k\in\mathbb{N}$ and $-1<\textup{Re}(z)<k$. As $y \rightarrow \infty$,
	\begin{align}\label{Hinfty}
		H_{z}^{(k)}(y) \sim  \frac{\sqrt{\pi} \left(\frac{1}{4}\big(\frac{y}{2k} \big)^{2k}  \right)^{\frac{1}{4(k+1)} - \frac{1+z}{2k(k+1)}} }{ \sqrt{k(k+1)}\,\, 2^{\frac{1+z}{k}}}  \cos\left( \frac{\pi}{4} + (2k+2) \left(\frac{1}{4}\left(\frac{y}{2k} \right)^{2k}  \right)^{\frac{1}{2k+2}} \right). 
	\end{align}
\end{theorem}

\begin{proof}
Letting $x=ye^{\pm\frac{i \pi}{2k}}, y>0,$ in Theorem \ref{meijergk},  we get, with $Y=\frac{1}{4}\left(\frac{y}{2k}\right)^{2k}$,
	\begin{align*}
		K_z^{(k)} \left(e^{\pm\frac{i \pi}{2k}} y \right) & = \frac{1}{ \sqrt{k}\,\, 2^{\frac{1+z}{k}}}  G_{0,2k+2}^{\,k+2,0} \!\left(  \,\begin{matrix} \{  \} \\ b_1^{'}, \cdots , b_q^{'} \end{matrix} \; \Big| e^{\pm i \pi} Y  \right).
	\end{align*}
Employ the asymptotic expansion of the Meijer $G$-function from Proposition \ref{asymptotic_Meijer G} to get
	\begin{align}\label{asymp_K_z}
		K_z^{(k)} \left(e^{\frac{i \pi}{2k}} y \right) \sim \frac{1}{ \sqrt{k}\,\, 2^{\frac{1+z}{k}}} \left( - \frac{1}{2\pi i}\right)^{k} \exp\left( -i \pi \sum_{j=k+3}^{2k+2} b_j^{'} \right) H_{0,2k+2}\left( Y e^{ i \pi (k+1)}  \right).
	\end{align}
From \eqref{b_j^{'}} one can easily check that $\sum_{j=k+3}^{2k+2} b_j^{'} = k/2$.  Taking the leading term,  the definition of $H_{p, q}$ in \eqref{H_{p,q}} gives
	\begin{align}\label{asymptotic_H}
		H_{0,2k+2}\left( Y e^{ i \pi (k+1)}  \right) \sim \frac{(2\pi)^{(2k+1)/2} Y^\theta }{\sqrt{2k+2}} \exp\left( -i (2k+2) Y^{1/(2k+2)}  + i \pi (k+1) \theta \right),
	\end{align}
	where $\theta = \frac{1}{4(k+1)} - \frac{1+z}{2k(k+1)}$. 
	Substituting \eqref{asymptotic_H} in \eqref{asymp_K_z}, as $y\rightarrow \infty$,
	\begin{align}
		K_z^{(k)} \left(e^{\frac{i \pi}{2k}} y \right)&  \sim \frac{(2\pi)^{(2k+1)/2} Y^\theta }{ \sqrt{2k(k+1)}\,\, 2^{\frac{1+z}{k}}} \left( - \frac{1}{2\pi i}\right)^{k} \exp\left(-\frac{i \pi k}{2} -i (2k+2) Y^{1/(2k+2)}  + i \pi (k+1) \theta \right)\nonumber \\
		& \sim \frac{\sqrt{\pi} Y^\theta }{ \sqrt{k(k+1)}\,\, 2^{\frac{1+z}{k}}}  \exp\left(-i (2k+2) Y^{1/(2k+2)}  + i \pi (k+1) \theta \right). \label{asymp_K_z +ve}
	\end{align}
	Similarly, it can be seen that
	\begin{align}\label{asymp_K_z -ve}
		K_z^{(k)} \left(e^{\frac{-i \pi}{2k}} y \right) & \sim \frac{\sqrt{\pi} Y^\theta }{ \sqrt{k(k+1)}\,\, 2^{\frac{1+z}{k}}}  \exp\left(i (2k+2) Y^{1/(2k+2)}  - i \pi (k+1) \theta \right).
	\end{align}
	Finally, making use of \eqref{asymp_K_z +ve} and \eqref{asymp_K_z -ve} in Theorem \ref{H_z_sum of two_K_z}, 
	we obtain after some simplification,  
	\begin{align*}
		H_{z}^{(k)}(y) \sim \frac{\sqrt{\pi} Y^\theta }{ \sqrt{k(k+1)}\,\, 2^{\frac{1+z}{k}}} \cos\left( \frac{\pi}{4} + (2k+2) Y^{1/{2k+2}} \right). 
	\end{align*}
Substituting the values of $Y$ and $\theta$ in the above formula, we arrive at \eqref{Hinfty}. 
\end{proof}

\begin{remark}
	When $k=1$ and $z=0$,  Theorem \ref{asym of H_z^k} implies that as $x \rightarrow \infty$,  
	\begin{align*}
		H_{0}^{(1)}(x) \sim \frac{\sqrt{\pi}}{2 \sqrt[4]{x}} \cos\left(\frac{\pi}{4} + 2 \sqrt{x} \right). 
	\end{align*}	This  can also be verified from the asymptotic formulas of $Y_0(x)$ and $K_0(x)$ upon using \eqref{Hardy}.
\end{remark}

\subsection{The special case $H_{z}^{(1)}(x)$}
In this section, we explicitly evaluate $H_z^{(1)}(x)$ in terms of Bessel functions. This will be required while proving Corollary \ref{varlauform}. 

\begin{theorem}\label{k=1}
	Let $H_z^{(k)}(x)$ be as defined in \eqref{defn_H_z^k} and let $M_{\nu}(x) = \frac{2}{\pi} K_{\nu}(x) - Y_{\nu}(x)$. For $-1<\textup{Re}(z)<1$ and $x>0$, we have
	\begin{equation}\label{k=1 eqn}
		H_z^{(1)}(x)=\frac{\pi}{2}x^{-\frac{z}{2}}\left( \cos \left( \tfrac{1}{2}\pi z \right) 
		M_{z}( 2 \sqrt{x}) - \sin \left( \tfrac{1}{2} \pi z \right) 
		J_{z}(2\sqrt{x})\right).
	\end{equation}
\end{theorem}
\begin{proof}
From the first equality in Theorem \ref{H_z_sum of two_K_z}, \eqref{inverse Mellin_H_z} and \eqref{inverse Mellin_H_z1}, we have, for $ \max\{ 0,  1 -k +  \Re(z) \} < c\leq\frac{1+\textup{Re}(z)}{k+1} $, 
	\begin{align}\label{integral_representation_H_z^(k)}
		H_z^{(k)}(x)= \frac{1}{ 2\pi i} \int_{(c)} \Gamma(s) \cos\left( \frac{\pi s}{2}  \right)  \Gamma\left( \frac{s-1-z}{k} +1 \right)  \cos\left( \frac{\pi}{2}  \left(  \frac{s-1-z}{k} +1  \right)  \right) \frac{\, ds}{k x^s}.
	\end{align}
	Letting $k=1$, replacing $s$ by $w+z/2$, we then use the  formula $2\cos(A)\cos(B)=\cos(A+B)+\cos(A-B)$ in order to get for $|\textup{Re}\left(\frac{z}{2}\right)|<c'=\textup{Re}(w)\leq\frac{1}{2}$,
	\begin{align}\label{integral_representation_H_z^(1)}
		H_z^{(1)}(x)= \frac{x^{-z/2}}{ 4\pi i} \int_{(c')} \Gamma\left(w-\frac{z}{2}\right) \Gamma\left(w+\frac{z}{2}\right)\left(\cos(\pi w)+\cos\left( \frac{\pi z}{2}  \right)\right) x^{-w} \, dw.
	\end{align}
	From \cite[Lemma 5.1]{dixitmoll}, for $|\textup{Re}\left(\frac{z}{2}\right)|<c'=\textup{Re}(w)<3/4$, we have
	\begin{align}\label{above}
		&\frac{1}{ 2\pi^2 i} \int_{(c')} \Gamma\left(w-\frac{z}{2}\right) \Gamma\left(w+\frac{z}{2}\right)\left(\cos(\pi w)+\cos\left( \frac{\pi z}{2}  \right)\right) (4\pi^2ty)^{-w} \, dw\nonumber\\
		&=\cos \left( \tfrac{1}{2}\pi z \right) 
		M_{z}( 4 \pi \sqrt{ty}) - \sin \left( \tfrac{1}{2} \pi z \right) 
		J_{z}(4 \pi \sqrt{ty}).
	\end{align}
	Employing \eqref{above} with $ty=x/(4\pi^2)$ on the right-hand side of \eqref{integral_representation_H_z^(1)} and noting that $H_z^{(1)}(x)$ converges for $-1<\textup{Re}(z)<1$ as proved in Theorem \ref{gen_Hardy_conv}, we arrive at \eqref{k=1 eqn}.
\end{proof}

\section{Voronoi summation formula for $\sigma_{z}^{(k)}(n)$:  Theorems \ref{Voronoi_Sigma_z} and \ref{Gen_Voronoi_Sigmakz} }

Armed with results in the previous section, we are now all set to prove the Vorono\"{\dotlessi} summation formula for $\sigma_{z}^{(k)}(n)$ for an analytic  function $f$ in a closed contour containing the segment $[\alpha, \beta]$, where $0<\alpha<\beta$ and $\alpha, \beta\notin\mathbb{Z}$.
\subsection{Proof of Theorem \ref{Voronoi_Sigma_z} } 
	First,  let us define 
\begin{align}\label{Phi_z}
\Phi_k(x;z) := C\sum_{n=1}^{\infty}  S_{z}^{(k)}(n)  K_z^{(k)}\left( (2\pi)^{\frac{1}{k}+1} (n x)^{\frac{1}{k}}   \right),
\end{align}
where $C:= C_k(x;z)=2 (2\pi x)^{\frac{1+z}{k}-1}$. 
Using \cite[p.~30, Theorem 2.3]{temme} and the discussion following \eqref{line integral_K_z}, it is clear that $K_z^{(k)}\left( (2\pi)^{\frac{1}{k}+1} (n x)^{\frac{1}{k}}   \right)$ is analytic in Re$(x)>0$.
One can establish the uniform convergence of the above series in \eqref{Phi_z} for Re$(x)>0$ in a manner similar to that proved by Koshliakov in \cite[p.~125-126]{koshwigleningrad}. Hence by Weierstrass' theorem on analytic functions, we see that $\Phi_k(x;z)$ is analytic in Re$(x)>0$.

Employing \eqref{line integral_K_z},  one can write the above series representation of $\Phi_k(x;z)$ as 
\begin{align*}
\Phi_k(x;z) = \frac{C}{ 2\pi i k} \sum_{n=1}^{\infty}  S_{z}^{(k)}(n) \int_{(c)} \Gamma(s) \cos\left( \frac{\pi s}{2}  \right)  \Gamma\left( \frac{s-1-z}{k} +1 \right) { \left( (2\pi)^{\frac{1}{k}+1} (n x)^{\frac{1}{k}}   \right)^{-s}}{\, ds},
\end{align*}
where $ c > \max\{ 0,  1 -k +  \Re(z) \}$.  Moreover,  we need $\Re(s)=c >  \max\{ 1,  1  +  \Re(z) \}$ to take the summation inside the integration as the Dirichlet series associated to $S_z^{(k)}(n)$  will be absolutely and uniformly convergent in this region.  Therefore,  using \eqref{divisor_function_S_z},  we get
\begin{align}\label{line_integral_Phi_z}
\Phi_k(x;z)= \frac{C}{ 2\pi i k}  \int_{(c)} \Gamma(s) \zeta(s) \cos\left( \frac{\pi s}{2}  \right)  \Gamma\left( \frac{s-1-z}{k} +1 \right) \zeta\left(   \frac{s-1-z}{k} +1  \right)  \frac{\, ds}{ \left( (2\pi)^{\frac{1}{k}+1} x^{\frac{1}{k}}   \right)^s}.
\end{align}
The asymmetric form of \eqref{zetafe} is given by
\begin{equation}\label{asymmetric_zeta(s)}
	\zeta(s) = 2^s \pi^{s-1} \Gamma(1-s) \zeta(1-s) \sin\left(  \frac{\pi s}{2} \right).
\end{equation}
To simplify further, we shall use \eqref{asymmetric_zeta(s)} in the form
\begin{align}\label{asymmetric_zeta(s)1}
\Gamma(s) \zeta(s) = \frac{\zeta(1-s) 2^{s-1} \pi^s }{\cos\left(\frac{\pi s}{2}  \right)}
\end{align}
Use \eqref{asymmetric_zeta(s)1} twice in  \eqref{line_integral_Phi_z} to obtain
\begin{align}\label{2nd form_Phi-z}
\Phi_k(x;z) = -  \frac{C}{ 2\pi i k}  \int_{(c)}  \frac{\zeta(1-s)\zeta\left(   \frac{z+1-s}{k}   \right) }{ \sin\left( \frac{\pi}{2} \left(  \frac{s-1-z}{k}  \right)   \right)} \frac{ \, ds}{ 4 (2 \pi)^{\frac{1+z}{k}-1} x^{ \frac{s}{k}}}.
\end{align}
Now we substitute $z+1-s= k s'$. This implies that the new line of integration  $\Re(s')=c' < \min\left\{ 0,  \frac{\Re(z)}{k} \right\}$ since $\Re(s)=c > \max\{ 0,  1 -k +  \Re(z) \}$.  Upon simplification, \eqref{2nd form_Phi-z} becomes 
\begin{align}\label{integral form_Phi_z}
\Phi_k(x;z) =   \frac{C}{ 4 x (2\pi x)^{\frac{1+z}{k}-1} } \frac{1}{2 \pi i}  \int_{(c')}  \frac{ \zeta(s') \zeta( k s' -z) }{\sin\left( \frac{\pi s'}{2}  \right) } x^{s'}  \, ds'.
\end{align}
Here we note that in the current line of integration $\Re(s')=c'$,  one can not use the series definition of $ \zeta(s') \zeta( k s' -z) $,  so we would like to shift the line of integration to a new line $\Re(s')=c''$ with $2>c''> \max\left\{ 1, \frac{1+\Re(z)}{k}\right\}$ as $  \zeta(s') \zeta( k s' -z)$ is absolutely and uniformly convergent in this new region.  To do that we consider the following rectangular contour $\mathcal{C}$ defined by $[c''-i T,  c''+iT,  c'+i T,  c'-iT]$.  In the process, we encounter simple poles at $s=0, 1,  \frac{1+z}{k}$ inside this contour.  Applying Cauchy's residue theorem,  we get
\begin{align}\label{using_CRT}
\frac{1}{ 2 \pi i} \int_{\mathcal{C}}  \frac{ \zeta(s') \zeta( k s' -z) }{\sin\left( \frac{\pi s'}{2}  \right) } x^{s'}  \, ds'= R_{0}+R_{1}+R_{\frac{1+z}{k}},
\end{align}
where $R_{\rho}$ denotes the residual the term corresponding to the pole at $s'=\rho$.  Letting $T \rightarrow \infty$,  one can easily show that the horizontal integrals go to zero.  Therefore,  \eqref{using_CRT} reduces to
\begin{align}\label{with residual terms}
\frac{1}{2 \pi i}  \int_{(c')}  \frac{ \zeta(s') \zeta( k s' -z) }{\sin\left( \frac{\pi s'}{2}  \right) } x^{s'}  \, ds' = \frac{1}{2 \pi i}  \int_{(c'')}  \frac{ \zeta(s') \zeta( k s' -z) }{\sin\left( \frac{\pi s'}{2}  \right) } x^{s'}  \, ds' - \left( R_{0}+R_{1}+R_{\frac{1+z}{k}}  \right),
\end{align}
where the residual terms are the following expressions
\begin{align}\label{residual terms}
R_0= -\frac{\zeta(-z)}{\pi},  \quad R_1= \zeta(k-z) x,  \quad R_{\frac{1+z}{k}} = \frac{x^{\frac{1+z}{k}} \zeta\left( \frac{1+z}{k} \right) }{k \sin \left( \frac{\pi}{2}\left( \frac{1+z}{k} \right) \right)}.
\end{align}
Now using the Dirichlet series representation \eqref{gen_sigma_z} of $\zeta(s')\zeta(ks'-z)$,  one can write 
\begin{align}\label{right_vertical}
 \frac{1}{2 \pi i}  \int_{(c'')}  \frac{ \zeta(s') \zeta( k s' -z) }{\sin\left( \frac{\pi s'}{2}  \right) } x^{s'}  \, ds' = \sum_{n=1}^\infty \sigma_{z}^{(k)}(n)   \frac{1}{2 \pi i}  \int_{(c'')} \frac{ \left(\frac{x}{n}  \right)^{s'} }{\sin\left( \frac{\pi s'}{2} \right)} \, ds'.
\end{align}
It is well-known \cite[p.~91, Equation (3.3.10)]{kp} that for any $0<d<2$, 
\begin{align}\label{inverse_Mellin_Cosec}
\frac{1}{2 \pi i}  \int_{(d)} \frac{x^{-s}}{ \sin\left( \frac{\pi s}{2} \right)} \, ds = \frac{2}{\pi} \frac{1}{1+x^2}. 
\end{align}
In view of \eqref{right_vertical} and \eqref{inverse_Mellin_Cosec},  \eqref{with residual terms} reduces to
\begin{align}\label{left vertical}
\frac{1}{2 \pi i}  \int_{(c')}  \frac{ \zeta(s') \zeta( k s' -z) }{\sin\left( \frac{\pi s'}{2}  \right) } x^{s'} \, ds' = \frac{2 x^2}{ \pi } \sum_{n=1}^\infty  \frac{\sigma_{z}^{(k)}(n)}{n^2 + x^2} -\left( R_{0}+R_{1}+R_{\frac{1+z}{k}}  \right).
\end{align}
Now substituting \eqref{left vertical} in \eqref{integral form_Phi_z},  we obtain
\begin{align}\label{2nd last form_Phi_z}
\Phi_k(x;z) =   \frac{C}{ 4 x (2\pi x)^{\frac{1+z}{k}-1} } \left\{ \frac{2 x^2}{ \pi } \sum_{n=1}^\infty  \frac{\sigma_{z}^{(k)}(n)}{n^2 + x^2} -\left( R_{0}+R_{1}+R_{\frac{1+z}{k}}  \right) \right\}.
\end{align}
Substituting the residual terms from \eqref{residual terms} in \eqref{2nd last form_Phi_z},  we see that
\begin{align}\label{Phi_z_final}
\Phi_k(x;z) = \Psi_k(x;z),
\end{align}
where
\begin{align}\label{Psi_z}
\Psi_k(x;z):= \frac{\zeta(-z)}{2 \pi  x} - \frac{  \zeta(k-z)}{2}  -  \frac{x^{\frac{1+z}{k}-1} \zeta\left( \frac{1+z}{k} \right) }{2 k \sin \left( \frac{\pi}{2}\left( \frac{1+z}{k} \right) \right)} + \frac{x}{\pi} \sum_{n=1}^\infty  \frac{\sigma_{z}^{(k)}(n)}{n^2 + x^2}. 
\end{align}
One can easily check that $ \Psi_k(x;z)$ is analytic,  as a function of $x$,  in the entire complex plane except  on the negative real axis and at $x= \pm i n,  n \in \mathbb{N}\cup \{ 0 \} $. Thus, $ \Psi_k(i x;z)$ is analytic on $\mathbb{C}$ except on the positive imaginary axis and at integers.  Similarly, $ \Psi_k(-i x;z)$ is analytic on $\mathbb{C}$ except on the negative imaginary axis and at integers.  The combination of these two facts implies that $\Psi_k(i x;z) +  \Psi_k(-i x;z)$ is analytic on $\mathbb{C}$ except on the imaginary axis and possibly at integers.  But
\begin{align*}
\lim_{x \rightarrow \pm n} ( x \mp n) \Psi_k(i x;z) = \frac{1}{2 \pi i}\sigma_z^{(k)}(n),  \quad \mathrm{and} \quad
\lim_{x \rightarrow \pm n} ( x \mp n) \Psi_k(-i x;z) = -\frac{1}{2 \pi i}\sigma_z^{(k)}(n),
\end{align*}
which implies that the function $\Psi_k(i x;z) +  \Psi_k(-i x;z)$ is analytic in $\Re(x) >0$. Observe that for $x$ lying inside an interval $(u, v)$ on the positive real line not containing
any integer in its interior, we have, using \eqref{Psi_z},
\begin{align}\label{removable_at n}
	\Psi_k(i x;z) +  \Psi_k(-i x;z) = - \zeta(k-z) - \frac{1}{k} x^{\frac{1+z}{k}-1} \zeta\left( \frac{1+z}{k} \right).
\end{align}
However, both sides of \eqref{removable_at n} are analytic in Re$(x)>0$, and hence, by analytic continuation we see that \eqref{removable_at n} holds in Re$(x)>0$.

Now let us consider $f(x)$ to be analytic function of $x$ inside a closed contour $\gamma$ that intersects the real line at $\alpha$ and $\beta$,  where $0< m-1 < \alpha < m \leq n-1 < \beta < n$ and $m, n \in \mathbb{N}$. Let $\gamma'$ and $\gamma''$ denote the upper and lower portion of the contour,  respectively.   This means that $\alpha \gamma' \beta$ and $\alpha \gamma'' \beta$ denote the paths from $\alpha$ to $\beta$ in the upper- and lower- half planes  respectively.  By Cauchy's residue theorem,
\begin{align}\label{Using_CRT}
\frac{1}{2\pi i} \int_{\alpha \gamma'' \beta \gamma' \alpha}  f(x) \Psi_k(i x;z) \, dx = \sum_{ \alpha < n < \beta } \lim_{x \rightarrow n} (x-n) f(x) \Psi_k(i x;z)  = \frac{1}{2 \pi i}\sum_{ \alpha < n < \beta } \sigma_z^{(k)}(n)f(n).
\end{align}
Now 
\begin{align}\label{divide_two parts}
\int_{\alpha \gamma'' \beta \gamma' \alpha}  f(x) \Psi_k(i x;z) \, dx= \int_{\alpha \gamma'' \beta }  f(x) \Psi_k(i x;z) \, dx -\int_{\alpha \gamma' \beta}  f(x) \Psi_k(i x;z) \, dx. 
\end{align}
Thus from \eqref{removable_at n},  \eqref{Using_CRT} and \eqref{divide_two parts}, 
\begin{align}\label{f(n)_Sigma_z}
\sum_{ \alpha < n < \beta } \sigma_z^{(k)}(n) f(n) & = \int_{\alpha \gamma'' \beta }  f(x) \Psi_k(i x;z) \, dx + \int_{\alpha \gamma' \beta}  f(x) \Psi_k(-i x;z) \, dx  \nonumber \\
& +  \int_{\alpha \gamma' \beta}  f(x) \left( \zeta(k-z) + \frac{1}{k} x^{\frac{1+z}{k}-1} \zeta\left( \frac{1+z}{k} \right)\right) \, dx.
\end{align}
Again,  utilizing Cauchy's residue theorem,  one can instantly observe that
\begin{align}\label{CRT_upper half plane}
\int_{\alpha \gamma' \beta}  f(x) \left( \zeta(k-z) + \frac{x^{\frac{1+z}{k}-1}}{k} \zeta\left( \frac{1+z}{k} \right)\right) \, dx = \int_{\alpha}^{\beta}  f(t) \left( \zeta(k-z) + \frac{1}{k} t^{\frac{1+z}{k}-1} \zeta\left( \frac{1+z}{k} \right)\right) \, dt. 
\end{align}
From the discussion following \eqref{Phi_z} and and \eqref{Psi_z},  it is clear that \eqref{Phi_z_final} holds for $ -\pi/2 < \arg(x) < \pi/2$.  Thus, $ \Psi_k(i x; z)= \Phi_k(i x; z)$ holds for $-\pi < \arg(x) <0$,   and $ \Psi_k(-i x; z)= \Phi_k(-i x; z) $ holds for $0 < \arg(x) <\pi$.   Employing these two facts and together with \eqref{CRT_upper half plane},   \eqref{f(n)_Sigma_z} becomes
 \begin{align*}
 \sum_{ \alpha < n < \beta } \sigma_z^{(k)}(n) f(n) & = \int_{\alpha \gamma'' \beta }  f(x) \Phi_k(i x;z) \, dx + \int_{\alpha \gamma' \beta}  f(x) \Phi_k(-i x;z) \, dx \nonumber \\
 & +  \int_{\alpha}^{\beta}  f(t) \left( \zeta(k-z) + \frac{1}{k} t^{\frac{1+z}{k}-1} \zeta\left( \frac{1+z}{k} \right)\right) \, dt. 
 \end{align*}
From the discussion following \eqref{Phi_z}, we know that the series defining $ \Phi_k(i x;z)$ and $ \Phi_k(-i x;z)$ are uniformly convergent in $-\pi < \arg(x) <0$ and $0 < \arg(x) <\pi$ respectively.  Thus using these series representations in the above identity and then interchanging the order of summation and integration, we arrive at 
\begin{align*}
 \sum_{ \alpha < n < \beta }  \sigma_z^{(k)}(n) f(n) & =  \int_{\alpha}^{\beta}  f(t) \left( \zeta(k-z) + \frac{1}{k} t^{\frac{1+z}{k}-1} \zeta\left( \frac{1+z}{k} \right)\right) \, dt \\
& +  2 (2\pi)^{\frac{1+z}{k}-1}  \sum_{n=1}^{\infty}  S_{z}^{(k)}(n) \bigg[  \int_{\alpha \gamma'' \beta }  f(x)  K_z^{(k)}\left( (2\pi)^{\frac{1}{k}+1} ( i n x)^{\frac{1}{k}}   \right) (i x)^{\frac{1+z}{k}-1} \, dx \nonumber \\
& \hspace{3.5cm}+ \int_{\alpha \gamma' \beta }  f(x)  K_z^{(k)}\left( (2\pi)^{\frac{1}{k}+1} (-i n x)^{\frac{1}{k}}   \right) (-i x)^{\frac{1+z}{k}-1} \, dx  \bigg]. \nonumber
\end{align*}
Here we use the residue theorem twice to obtain 
\begin{align*}
 \sum_{ \alpha < n < \beta }  f(n) \sigma_z^{(k)}(n) & =  \int_{\alpha}^{\beta}  f(t) \left( \zeta(k-z) + \frac{1}{k} t^{\frac{1+z}{k}-1} \zeta\left( \frac{1+z}{k} \right)\right) \, dt +  2 (2\pi)^{\frac{1+z}{k}-1}  \sum_{n=1}^{\infty}  S_{z}^{(k)}(n) \nonumber  \\
& \times   \int_{\alpha}^{\beta }  f(t)  t^{\frac{1+z}{k}-1} \bigg[ \exp\left(\frac{-i \pi(k-1-z)}{2 k}  \right) K_z^{(k)}\left( (2\pi)^{\frac{1}{k}+1} (n t)^{\frac{1}{k}} e^{\frac{i \pi}{2k}}  \right)  \, dt \nonumber  \\
& \hspace{3cm} +  \exp\left(\frac{i \pi(k-1-z)}{2 k}  \right) K_z^{(k)}\left( (2\pi)^{\frac{1}{k}+1} (n t)^{\frac{1}{k}} e^{-\frac{i \pi}{2k}}   \right) \, dt \bigg].  
\end{align*}
Finally,  invoking Theorem \ref{H_z_sum of two_K_z},  we arrive at \eqref{Voronoi_Sigma_z_eqnn}. This completes the proof. 
\qed

\begin{proof}[Corollary \ref{varlauform}][]
	
Letting $k=1$ in Theorem \ref{Voronoi_Sigma_z} gives
\begin{align*}
	\sum_{\alpha<j<\beta}\sigma_{z}(j)f(j)&=\int_{\alpha}^{\beta}(\zeta(1-z)+t^{z}\zeta(1+z))f(t)dt+2(2\pi)^{z}\sum_{n=1}^{\infty}\sigma_{z}(n)\int_{\alpha}^{\beta}t^zf(t)H_z^{(1)}(4\pi^2nt)dt.
\end{align*}
Now invoke Lemma \ref{k=1} with $x=4\pi^2nt$, simplify and then replace $z$ by $-z$ to arrive at \eqref{varlauform1}. Here, we have made use of the elementary fact $\sigma_{-z}(n)n^{z/2}=\sigma_{z}(n)n^{-z/2}$ as well as the fact \cite[p.~842]{bdrz1} that $ \cos \left( \tfrac{1}{2}\pi z \right) 
M_{z}( 2 \sqrt{x}) - \sin \left( \tfrac{1}{2} \pi z \right) $ is invariant under the replacement of $z$ by $-z$, which, in fact, is an easy consequence of the definition in \eqref{yj} and the identity $K_{-z}(\xi)=K_{z}(\xi)$.
\end{proof}

\subsection{Proof of Theorem \ref{Gen_Voronoi_Sigmakz}} 
We first prove the result for $z\neq k-1$. Using the inverse Mellin transform of $F(s)$, one can write 
\begin{align}\label{firststep}
\sum_{n=1}^\infty \sigma_{z}^{(k)}(n) f(n) = \sum_{n=1}^\infty \sigma_{z}^{(k)}(n) \frac{1}{2 \pi i} \int_{(c)} F(s) n^{-s} ds = \frac{1}{2 \pi i} \int_{(c)} F(s) \zeta(s) \zeta(k s-z) ds,
\end{align}
where $c > \max \left\{1, \frac{1+\Re(z)}{k}\right\}.$ Since $f\in \mathscr S(\mathbb{R})$, its Mellin transform $F(s)$ is holomorphic on $\textup{Re}(s) >0$. Moreover, integration by parts gives the following identity for $\mathcal{M}(f)(s)$: 
\begin{align*}
	\mathcal{M}(f)(s) &= \left. f(x)\frac{x^s}{s} \right|_0^\infty -  \int_0^\infty f'(x) \frac{x^s}{s} dx 
	= -\frac{1}{s} \mathcal{M}(f')(s+1). 
\end{align*}
Hence
\begin{equation*}
s(s+1) \cdots (s+i) F(s)=(-1)^{i+1} \mathcal{M}(f^{(i+1)})(s+i+1).
\end{equation*}
This proves that $F(s)$ has an analytic continuation to the whole complex plane except for possible simple poles at $s=0, -1,-2, \cdots.$
We know that $\zeta(s) \zeta(k s-z)$ has simple poles at $1$ and $\frac{1+z}{k}$. 
To transform the line integral in \eqref{firststep}, we shall consider  the following contour $\mathcal{C}:= [c-iT, c+iT, \lambda + iT, \lambda-iT]$, where
\begin{equation}\label{conditionlambda}
 \lambda = - \epsilon,\hspace{1mm}\text{with}\hspace{1mm} \max\left\{0,-\frac{\textup{Re}(z)}{k}\right\}< \epsilon <1.
\end{equation}
  Choose $T$ large enough so that $|\Im(z)/k| < T$. 
Now employing the Cauchy residue theorem, we have 
\begin{align}\label{CRT}
\frac{1}{2 \pi i} \int_{\mathcal{C}} F(s) \zeta(s) \zeta(k s-z) ds = R_{0} + R_{1} + R_{\frac{1+z}{k}},
\end{align}
where the residues are given by
\begin{align}\label{residues}
R_{0} & = \lim_{s \rightarrow 0} s F(s) \zeta(s) \zeta(k s-z) = \frac{1}{2}\mathcal{M}(f')(1)\zeta(-z)= \frac{\zeta(-z)}{2} \int_{0}^\infty f'(y) dy =- \frac{\zeta(-z)f(0^{+})}{2}, \nonumber\\
R_{1}& = \lim_{s \rightarrow 1} (s-1) F(s) \zeta(s) \zeta(k s-z)=F(1) \zeta(k-z)=  \zeta(k-z) \int_{0}^\infty f(y) dy,\nonumber\\
R_{\frac{1+z}{k}} & = \lim_{s \rightarrow \frac{1+z}{k}} \left(s - \frac{1+z}{k}\right) F(s) \zeta(s) \zeta(ks-z) = \frac{1}{k} F\left( \frac{1+z}{k} \right)\zeta\left( \frac{1+z}{k} \right).
\end{align}
Now let $T \rightarrow \infty$ in \eqref{CRT}. It can be checked that the integrals along the horizontal segments go to zero whence
\begin{align}\label{eq:CRT simp}
\frac{1}{2 \pi i} \int_{(c)} F(s) \zeta(s) \zeta(k s-z) ds = R_{0} + R_{1} + R_{\frac{1+z}{k}} + I,
\end{align}
where
\begin{equation}\label{III}
I:=\frac{1}{2 \pi i} \int_{(\lambda)} F(s) \zeta(s) \zeta(k s-z) ds.
\end{equation}
We would like to write $I$ in terms of an infinite series involving the function $S_z^{(k)}(n)$. To that end, using \eqref{asymmetric_zeta(s)} twice, we have
\begin{align*}
\zeta(s) \zeta(ks-z) & = \frac{(2 \pi)^{(k+1)s-z}}{\pi^2} \Gamma(1-s) \zeta(1-s) \Gamma(1-ks+z) \zeta(1-ks+z)\\
 & \times \sin\left(\frac{\pi s}{2}\right) 
 \sin\left(\frac{\pi}{2}(ks-z)\right).
\end{align*}
Substituting this expression into the right-hand side of \eqref{eq:CRT simp} and employing the change of variable $s=\frac{1+z}{k}-w$, we have
\begin{align*}
I
&=\frac{1}{2 \pi i} \int_{\left(-\lambda + \frac{1+\Re(z)}{k}\right)}  \frac{(2 \pi)^{(k+1)\left(\frac{1+z}{k}-w\right)-z}}{\pi^2} F\left(\frac{1+z}{k}-w\right) \Gamma\left(w+1-\frac{1+z}{k}  \right)\zeta\left(w+1-\frac{1+z}{k}  \right) \\
&\qquad\qquad \times \Gamma(kw) \zeta(kw) \cos\left(\frac{\pi}{2}\left( w+1-\frac{1+z}{k} \right) \right) \cos\left(\frac{\pi kw}{2}\right) dw.
\end{align*}
From \eqref{conditionlambda}, we have  $-1<\lambda<\min\left\{0,\frac{\textup{Re}(z)}{k}\right\}$ which implies Re$(kw)>1$ as well as Re$\left(w+1-\frac{1+z}{k}  \right)>1$. Hence, invoking \eqref{divisor_function_S_z} and  interchanging of the order of summation and integration (which is justified by the absolute and uniform convergence), we have
\begin{align*}
I&=\frac{(2 \pi)^{(k+1)\left(\frac{1+z}{k}\right)-z}}{\pi^2} \sum_{n=1}^\infty S_{z}^{(k)}(n) \frac{1}{2 \pi i} \int_{\left(-\lambda + \frac{1+\Re(z)}{k}\right)} F\left(\frac{1+z}{k}-w\right)N_z^{(k)}(w)\left( (2\pi)^{k+1}n  \right)^{-w}\, dw\nonumber\\
&=\frac{(2 \pi)^{(k+1)\left(\frac{1+z}{k}\right)-z}}{\pi^2} \sum_{n=1}^\infty S_{z}^{(k)}(n) \frac{1}{2 \pi i} \int_{\left(-k\lambda + 1+\Re(z)\right)} F\left(\frac{1+z-\xi}{k}\right)N_z^{(k)}\left(\frac{\xi}{k}\right)\left( (2\pi)^{1+\frac{1}{k}}n^{\frac{1}{k}}  \right)^{-\xi}\, \frac{d\xi}{k},
\end{align*}
where
\begin{align*}
 N_{z}^{(k)}(w):= \Gamma\left(w+1-\frac{1+z}{k}  \right) \cos\left(\frac{\pi}{2}\left( w+1-\frac{1+z}{k} \right) \right) \Gamma(kw) \cos\left(\frac{\pi}{2}kw\right),
\end{align*}
and where in the last step we employed the change of variable $w=\xi/k$ so that $\textup{Re}(\xi)=-k\lambda + 1+\Re(z)$. 

Observe that $\max\{ 0,  1 -k +  \Re(z)\}<\Re(\xi)$. We also need Re$(\xi)\leq\frac{1+\textup{Re}(z)}{k+1}$ for reasons to be clear soon, however, we unfortunately have $\Re(\xi)>1>\frac{1+\textup{Re}(z)}{k+1}$ at this stage (since $\lambda<\textup{Re}(z)/k$).  To circumvent this problem, we shift the line of integration to $\max\{ 0,  1 -k +  \Re(z)\}<c''=\Re(\xi)\leq\frac{1+\textup{Re}(z)}{k+1}$ and apply Cauchy's residue theorem. Since $-1<\textup{Re}(z)<k$ and Re$(\xi)>0$, we do not encounter any poles of the integrand in this process. (There is no pole at $\xi=1$ as well because the possibility of $F\left(\frac{1+z-\xi}{k}\right)$ giving rise to it arises \emph{only} when $z=0$, since $-1<\textup{Re}(z)<k$, and even if that is the case, $\cos\left(\frac{\pi}{2} \left(1-\frac{z}{k}\right)\right)=0$ there.) Also, the integrals along the horizontal segments tend to zero as the height of the contour tends to $\infty$. Hence
\begin{align}\label{aftershift}
	I=\frac{(2 \pi)^{(k+1)\left(\frac{1+z}{k}\right)-z}}{\pi^2} \sum_{n=1}^\infty S_{z}^{(k)}(n) \frac{1}{2 \pi i} \int_{(c'')} F\left(\frac{1+z-\xi}{k}\right)N_z^{(k)}\left(\frac{\xi}{k}\right)\left( (2\pi)^{1+\frac{1}{k}}n^{\frac{1}{k}}  \right)^{-\xi}\, \frac{d\xi}{k},
\end{align}
with  $\max\{ 0,  1 -k +  \Re(z)\}<c''=\Re(\xi)\leq\frac{1+\textup{Re}(z)}{k+1}$.
 
Now insert the integral representation of $F$, namely,
\begin{equation*}
	F\left(\frac{1+z-\xi}{k}\right)=\int_{0}^{\infty}y^{\frac{1+z-\xi}{k}-1}f(y)\, dy
\end{equation*} 
in \eqref{aftershift}, then interchange the order of integration which is permissible due to  the decay of $F$ (since $f\in\mathscr{S}(\mathbb R))$ so that
\begin{align}\label{Iend}
	I&=\frac{(2 \pi)^{(k+1)\left(\frac{1+z}{k}\right)-z}}{\pi^2} \sum_{n=1}^\infty S_{z}^{(k)}(n) \int_{0}^{\infty}y^{\frac{1+z}{k}-1}f(y)\frac{1}{2 \pi i} \int_{\left(c''\right)}N_z^{(k)}\left(\frac{\xi}{k}\right)\left( (2\pi)^{1+\frac{1}{k}}n^{\frac{1}{k}}  \right)^{-\xi}\, \frac{d\xi}{k}\, dy\nonumber\\
& = \frac{(2 \pi)^{(k+1)\left(\frac{1+z}{k}\right)-z}}{\pi^2} \sum_{n=1}^\infty S_{z}^{(k)}(n) \int_{0}^\infty H_z^{(k)}\left( (2\pi)^{1+ \frac{1}{k}} (ny)^{\frac{1}{k}} \right)y^{\frac{1+z}{k}-1} f(y)\, dy
\end{align}
where, in the last step, we invoked \eqref{integral_representation_H_z^(k)}. Note that we had to shift the line of integration to Re$(\xi)=c''$ to be able to use \eqref{integral_representation_H_z^(k)}.

From \eqref{firststep}, \eqref{residues}, \eqref{eq:CRT simp}, \eqref{III} and \eqref{Iend}, we are led to \eqref{extra}.

Now when $z=k-1$, the only difference is that the poles of the integrand of \eqref{CRT} at  $1$ and $(1+z)/k$ coalesce giving a double pole because of which $R_1=	\int_{0}^{\infty}  f(t) \left(\frac{(k+1)\gamma+\log(t)}{k}\right) dt$.
\qed

\section{A generalization of Theorem \ref{Wigert_Lkw} of Wigert: Towards  Theorem \ref{thm: gen of wigert}}  \label{wigertsection}

This section begins with certain lemmas which will play a crucial role in proving Theorem \ref{thm: gen of wigert}.  We first evaluate special values of the function $B(z, b)$ defined in \eqref{integral defn B(z,b)} and  \eqref{bzbanother}.
\subsection{Special values of $B(z,a)$} 
The next result evaluates $B(z,a)$ at non-negative even integers.
\begin{lemma}\label{B(z,b) for z even}
For $m \in \mathbb{N}\cup \{ 0 \}$, 
\begin{align*}
B(2m, b) = \frac{\pi}{2b} b^{2m} e^{-b} (-1)^m. 
\end{align*}
\end{lemma} 
\begin{proof}
To prove this lemma, we use the identity \eqref{simplified form B(z, b)} so that 
\begin{align*}
\lim_{z \rightarrow 2m} B(z,b) & = \frac{\pi}{2} (-1)^m b^{2m-1} \cosh(b) - \frac{\pi}{2} (-1)^m \sum_{n=0}^{\infty} \frac{b^{2n}}{\Gamma(2n-2m+2)} \\
& =  \frac{\pi}{2} (-1)^m b^{2m-1} \cosh(b) - \frac{\pi}{2} (-1)^m \sum_{n=m}^{\infty} \frac{b^{2n}}{\Gamma(2n-2m+2)} \\
&= \frac{\pi}{2} (-1)^m b^{2m-1} \cosh(b) - \frac{\pi}{2} (-1)^m \sum_{i=0}^{\infty} \frac{b^{2m+2i}}{\Gamma(2i+2)} \\
&= \frac{\pi}{2} (-1)^m b^{2m-1} \cosh(b) - \frac{\pi}{2} (-1)^m b^{2m-1} \sinh(b) 
 = \frac{\pi}{2} (-1)^m b^{2m-1}e^{-b}.
\end{align*}
\end{proof}
Similarly, one can prove that 
\begin{lemma}
For $m \in \mathbb{N}$, 
\begin{align*}
B(-2m, b) = \frac{\pi}{2b} b^{-2m} (-1)^m \Bigg[ e^{-b} - \sum_{j=0}^{m-1} \frac{b^{2j+1}}{\Gamma(2j+2)} \Bigg]. 
\end{align*}
\end{lemma}
The proof of this lemma is similar that of Lemma \ref{B(z,b) for z even} and is hence omitted.
The next result evaluates $B(z,b)$ at positive odd integers. 
\begin{lemma}\label{B(z,b) for z odd} 
For $m \in \mathbb{N}\cup \{ 0 \}$, 
\begin{align*}
B(2m+1, b) = b^{2m} (-1)^m \sum_{n=0}^\infty \frac{b^{2n}}{\Gamma(2n+1)} \left( \psi(2n+1)-\log b \right). 
\end{align*}
\end{lemma}
\begin{proof}
At first glance, it seems from \eqref{simplified form B(z, b)} that $B(z,b)$ has singularities at odd integers.  However, we show that at positive odd integers, they are removable.
We write  the Laurent series expansions of the terms in \eqref{simplified form B(z, b)}. A bit of calculation implies that for $m \geq 0$,  
\begin{align}\label{Laurent_1st term}
\frac{\pi b^{z-1}}{2}   \frac{\cosh b}{\cos (\pi z/2)} = \frac{a_{-1}}{z-(2m+1)} + a_0 + O \left(|z-(2m+1)|\right),
\end{align}
where $a_{-1}= (-1)^{m+1} b^{2m} \cosh(b)$, and $a_0= (-1)^{m+1} b^{2m} \log(b) \cosh(b)$. 
Here we have used the fact that
\begin{align}\label{Laurent_Sec}
\sec(\pi z/2)=  \frac{c_{-1}}{z-(2m+1)}  + O\left(|z-(2m+1)|\right),
\end{align}
where $c_{-1}= \frac{2}{\pi} (-1)^{m+1}$.
Now we shall try to find the Laurent series expansion of the second term in \eqref{simplified form B(z, b)},  i.e.,  
\begin{align}\label{second term}
\frac{\pi}{2} \sec(\pi z/2) \sum_{n=0}^\infty \frac{b^{2n}}{\Gamma(2n-z+2)}
\end{align} at $z=2m+1$. 
We need to find the Laurent series expansion for the entire function $\frac{1}{\Gamma(2n-z+2)}$ at $z=2m+1$. Note that the sum over $n$ in \eqref{second term} will run from $n=m$ to infinity since the first $m$ terms are zero.  
One can check that
\begin{align}\label{Laurent_1_by_gamma}
\frac{1}{\Gamma(2n-z+2)} = d_0 + d_1 (z- (2m+1)) + O\left(|z -(2m+1)|^2\right), 
\end{align}
with $d_0= \frac{1}{\Gamma(2n-2m+1)}$ and $d_1=  \lim_{z \rightarrow (2m+1)} \frac{d}{dz}\frac{1}{\Gamma(2n-2m+1)} = \frac{\psi(2n-2m+1)}{\Gamma(2n-2m+1)}$ for $n \geq m$, where $\psi(z)$ denotes the logarithmic derivative of the gamma function. Thus,  combining \eqref{Laurent_Sec} and \eqref{Laurent_1_by_gamma},  the Laurent series expansion of \eqref{second term} becomes
\begin{align*}
 \frac{\pi}{2} \left( \frac{c_{-1}}{z-(2m+1)}  + O(z-(2m+1)) \right) \sum_{n=m}^\infty b^{2n} \left(  d_0 + d_1 (z- (2m+1)\right) + O\left(|z -(2m+1)|^2\right). 
\end{align*}
Substituting $c_{-1}, d_0, d_1$ and simplifying, one can find that the coefficient of $\frac{1}{z-(2m+1)}$ is 
\begin{align}\label{residue}
(-1)^{m+1} \sum_{n=m}^\infty \frac{b^{2n}}{\Gamma(2n-2m+1)} = (-1)^{m+1} b^{2m}\cosh(b), 
\end{align}
 and the constant term is 
\begin{align}\label{constant term_2nd}
(-1)^{m+1} \sum_{n=m}^\infty \frac{b^{2n} \psi(2n-2m+1)}{\Gamma(2n-2m+1)} = (-1)^{m+1} b^{2m} \sum_{n=0}^\infty \frac{b^{2n} \psi(2n+1)}{\Gamma(2n+1)} .
\end{align}
Finally, combining \eqref{Laurent_1st term}, \eqref{residue} and \eqref{constant term_2nd}, we can easily see that $B(z,b)$ has a removable singularity at $z=2m+1$ and adding constant terms we complete the proof. 
\end{proof}
\begin{remark}
In particular,  letting $m=0$ in Lemma \ref{B(z,b) for z odd},  we see that
\begin{align*}
B(1,b)= \int_0^\infty \frac{t \cos t}{t^2+b^2} ~dt = \sum_{n=0}^\infty \frac{b^{2n}}{\Gamma(2n+1)} \left( \psi(2n+1)-\log b \right). 
\end{align*}
This result was recently established in \cite[Lemma 3.2]{DGKM}. The integral in the above identity is known as Raabe's consine transform. The reader is encouraged to see \cite[Section 3]{DGKM} for more details on this integral.   
\end{remark}
The next three lemmas offer useful partial fraction decompositions of some algebraic functions. 
\begin{lemma}\label{Partial Fraction}
For $k \geq 1$ odd, 
\begin{align}\label{k odd}
\frac{t^k}{t^{2k} + a^{2k}} = 2t \sum_{j=1}^k \frac{C_{2j-1} }{t^2- \left(a\zeta_{4k}^{2j-1}\right)^2},
\end{align}
and for $k \geq 2$ even,
\begin{align}\label{k even}
\frac{t^k}{t^{2k} + a^{2k}} = 2 a \sum_{j=1}^{k} \frac{C_{2j-1} \zeta_{4k}^{2j-1} }{t^2- \left(a\zeta_{4k}^{2j-1}\right)^2},
\end{align}
where $C_{2j-1}= \frac{1}{2k} a^{1-k}\zeta_{4k}^{(1-k)(2j-1)}$ and $\zeta_{4k}:=e^{\frac{i \pi}{2k}}$. 

\end{lemma}

\begin{proof}
For any $k \geq 1$,  one can easily check that roots of $t^{2k} + a^{2k} =0$ are $t_j= a \zeta_{4k}^{2j+1}$ for $0 \leq j \leq 2k-1 $.  Note that $t_k = - t_0, t_{k+1 }= - t_1, \cdots, t_{2k-1}= - t_{k-1}$. Thus,
 \begin{align*}
t^{2k}+ a^{2k} &  = (t- t_0)(t+ t_0) (t-t_1) (t+ t_1) \cdots (t-t_{k-1}) (t+ t_{k-1})\\
& = ( t^2 - a^2 \zeta_{4k}^2) (t^2 - a^2 \zeta_{4k}^6) \cdots (t^2 - a^2 \zeta_{4k}^{2(2k-1)}).
\end{align*} 
Utilizing the method of partial fraction decomposition,  one can write
\begin{align}\label{partial_frac}
\frac{t^k}{t^{2k} + a^{2k}} = \sum_{j=1}^{k} \left( \frac{C_{2j-1}}{t- a \zeta_{4k}^{2j-1}} + \frac{C_{2j}}{t+ a \zeta_{4k}^{2j-1}} \right),
\end{align}
with
\begin{align}\label{cj}
C_{2j-1}= \frac{a^{1-k}}{2k  \zeta_{4k}^{(k-1)(2j-1)}},  \,  \text{and} \,\, C_{2j} = \frac{a^{1-k} (-1)^{k-1}}{2k  \zeta_{4k}^{(k-1)(2j-1)}}.
\end{align}
When $k \geq 1$ odd, $C_{2j-1}= C_{2j}$ and when $k \geq 2$ even, $C_{2j-1}=-C_{2j}$. Substituting these values of $C_{2j-1}$ and $C_{2j}$ in \eqref{partial_frac}, we obtain \eqref{k odd} and \eqref{k even}. 
\end{proof}

Lemma \ref{Partial Fraction}, in turn, leads to the following partial fraction decompositions, the second of which was obtained by Koshliakov  \cite[pp. 124-125]{koshwigleningrad}. 
\begin{lemma}\label{partial fraction_simplified}
For $k \geq 1$ odd,
\begin{align}\label{k odd simplified}
\frac{t^k}{t^{2k} + a^{2k}} = \frac{(-1)^{\frac{k-1}{2}} a^{1-k} t}{k} \Bigg[ \frac{1}{t^2+ a^2} + \sum_{j=1}^{\frac{k-1}{2}} \frac{B_j}{t^2 + \left(a \zeta_{4k}^{2j} \right)^2} + \frac{\overline{B}_j}{t^2 + \left( a \zeta_{4k}^{-2j} \right)^2} \Bigg],
\end{align}
where $B_j=\zeta_{4k}^{(1-k)(2j)}$ and $\overline{B}_j$ is the conjugate of $B_j$,
and for $k \geq 2$ even,
\begin{align}\label{k even simplified}
\frac{t^k}{t^{2k} + a^{2k}} =  \frac{(-1)^{\frac{k}{2}-1} a^{2-k} }{k}  \sum_{j=1}^{\frac{k}{2}} \Bigg[ \frac{A_j}{t^2 + \left(a \zeta_{4k}^{2j-1} \right)^2} + \frac{\overline{A}_j}{t^2 + \left( a \zeta_{4k}^{-(2j-1)} \right)^2} \Bigg],
\end{align}
where $A_j = \zeta_{4k}^{(2-k)(2j-1)}$ and $\overline{A}_j$ is the conjugate of $A_j$.
\end{lemma}
\begin{proof}
We prove the result only for $k$ odd. The proof for even $k$ is similar.
 From \eqref{k odd}, 
\begin{align*}
\frac{t^k}{t^{2k} + a^{2k}} = 2t \sum_{j=1}^{\frac{k-1}{2}} \frac{C_{2j-1} }{t^2- \left(a\zeta_{4k}^{2j-1}\right)^2}+\frac{2tC_k}{t^2-a^2\zeta_{4k}^{2k}}+2t \sum_{j=\frac{k+3}{2}}^k \frac{C_{2j-1} }{t^2- \left(a\zeta_{4k}^{2j-1}\right)^2},
\end{align*}
where $C_{2j-1}$ are defined in \eqref{cj}.
The term corresponding to $j=(k+1)/2$ is
\begin{align}\label{middle term}
\frac{2tC_k}{t^2-a^2\zeta_{4k}^{2k}} =   \frac{(-1)^{\frac{k-1}{2}} a^{1-k}}{k}    \frac{t}{t^2+ a^2}. 
\end{align}
Also,
\begin{align*}
2t \sum_{j=1}^{\frac{k-1}{2}} \frac{C_{2j-1} }{t^2- \left(a\zeta_{4k}^{2j-1}\right)^2}= \frac{a^{1-k} t}{k} \sum_{j=1}^{(k-1)/2} \frac{\zeta_{4k}^{(1-k)(2j-1)} }{t^2 +\left(a\zeta_{4k}^{2j-1-k}\right)^2}, 
\end{align*}
where we used $\zeta_{4k}^{2k}=-1$.  
Changing the variable $ 2j-1-k$ by $-2J$ yields
\begin{align}\label{first sum}
 \frac{(-1)^{\frac{k-1}{2}} a^{1-k} t}{k} \sum_{J=1}^{(k-1)/2} \frac{ \zeta_{4k}^{-2J(1-k)}}{{t^2 + \left( a \zeta_{4k}^{-2J} \right)^2}}.
\end{align}
The sum from $j=(k+3)/2$ to $k$ is treated in the same way.  Mainly, we replace $ 2j-1-k$ by $2J$ to have
\begin{align}\label{last sum}
2t \sum_{j=\frac{k+3}{2}}^k \frac{C_{2j-1} }{t^2- \left(a\zeta_{4k}^{2j-1}\right)^2}= \frac{(-1)^{\frac{k-1}{2}} a^{1-k} t}{k} \sum_{J=1}^{(k-1)/2} \frac{ \zeta_{4k}^{2J(1-k)}}{{t^2 + \left( a \zeta_{4k}^{2J} \right)^2}}. 
\end{align}
Finally, combining \eqref{middle term}, \eqref{first sum}, and \eqref{last sum}, we derive \eqref{k odd simplified}.
To prove \eqref{k even simplified}, we use \eqref{k even} and separate the sum in two parts, the first from $j=1$ to $k/2$, and then the second from $j=k/2 + 1$ to $k$.  The details are similar.
\end{proof}
The above partial fraction decompositions permit us to obtain an elegant explicit evaluation of an integral:  
\begin{lemma}\label{exact_evaluation}
Let $k \geq 2$ be an even integer and $m$ be an integer such that $0 \leq 2m <k$. Then the following identity holds:
\begin{align*}
\int_{0}^{\infty} \frac{t^{k+2m} \cos(t)}{t^{2k} + a^{2k}} ~ dt = \frac{\pi(-1)^{\frac{k}{2}+m-1}}{2k}  a^{2m-k+1} & \sum_{j=1}^{k/2}  \Bigg[\exp\left( \frac{i\pi}{2k}(1-k+2m)(2j-1)-a e^{ \frac{i\pi}{2k}(2j-1)} \right)   \\
& + \exp\left(- \frac{i\pi}{2k}(1-k+2m)(2j-1)-a e^{ -\frac{i\pi}{2k}(2j-1)} \right) \Bigg].
\end{align*}
\end{lemma}
\begin{proof}
Let  $\mathcal{I}_{k, a}(z):=  \displaystyle\int_{0}^{\infty}\frac{t^{k+z} \cos(t)}{t^{2k} + a^{2k}} \, dt $. It converges in $-1-k <\Re(z) <k$.  Employ \eqref{k even simplified} to see that
\begin{align*}
\mathcal{I}_{k, a}(z)=   \frac{(-1)^{\frac{k}{2}-1} a^{2-k} }{k} \sum_{j=1}^{\frac{k}{2}}A_j \int_{0}^{\infty} \frac{ t^z \cos(t)}{t^2 + \left(a \zeta_{4k}^{(2j-1)}\right)^2}~ dt
 + \bar{A_j}  \int_{0}^{\infty} \frac{  t^z \cos(t) }{t^2 + \left( a\zeta_{4k}^{-(2j-1)} \right)^2} ~ dt,
\end{align*}
where $A_j$ and $ \bar{A_j}$  are defined as in  Lemma \ref{partial fraction_simplified}.  Note that  the above identity holds for $-1 < \Re(z) <2$ since the integrals on the right side are convergent in this region only. Invoking Lemma \ref{Ober}, we obtain
\begin{align*}
\mathcal{I}_{k, a}(z)=  \frac{(-1)^{\frac{k}{2}-1} a^{2-k} }{k} \sum_{j=1}^{\frac{k}{2}}A_j B\left(z, a \zeta_{4k}^{(2j-1)}\right) +  \bar{A_j} B\left(z,a \zeta_{4k}^{-(2j-1)}\right).
\end{align*}
At this juncture,  we use Remark \ref{analytic continuation of B(z,b)} to analytically continue the above identity in the region $- 1< \Re(z) <k$. Now let $z=2m$ with $0 \leq m < k/2$, in the above identity and then utilize Lemma \ref{B(z,b) for z even} to derive
\begin{align*}
\mathcal{I}_{k, a}(2m) = \frac{\pi}{2k}  (-1)^{\frac{k}{2}+m-1} a^{2m-k+1} \sum_{j=1}^{\frac{k}{2}} & \Bigg[ A_j  \zeta_{4k}^{(2m-1)(2j-1)}  \exp\left( -  a \zeta_{4k}^{(2j-1)}\right) \\
& +   \bar{A_j}  \zeta_{4k}^{-(2m-1)(2j-1)}  \exp\left( -  a \zeta_{4k}^{-(2j-1)}\right)\Bigg] .
\end{align*}
Finally, substituting values of $A_j, \bar{A_j}$ and simplifying, one sees that the proof is complete.
\end{proof}
\subsection{Proof of Theorem \ref{thm: gen of wigert}}

%

We first prove the result for $w>0$ and then extend it by analytic continuation to $\Re(w) >0$.  
Substituting $f(x)= \exp(-xw)$ and $F(s):=M(f,s) = \Gamma(s)/w^s$ in Theorem \ref{Gen_Voronoi_Sigmakz} and simplifying,  we get
\begin{align}\label{exact expr wig}
\nonumber
&\sum_{n=1}^\infty \sigma_{z}^{(k)}(n) \exp(-n w)  = - \frac{\zeta(-z)}{2} + \frac{\zeta(k-z)}{w}  + \frac{1}{k}\frac{\Gamma\left( \frac{1+z}{k} \right)}{w^{(1+z)/k}}\zeta\left( \frac{1+z}{k} \right) \\ 
& + \frac{(2 \pi)^{(k+1)\left(\frac{1+z}{k}\right)-z}}{\pi^2} \sum_{n=1}^\infty S_{z}^{(k)}(n) \int_{0}^\infty H_z^{(k)}\left( (2\pi)^{1+ 1/k} (ny)^{1/k} \right)y^{\frac{1+z}{k}-1} \exp(-y w) dy,
\end{align}
where $S_{z}^{(k)}(n)$ and $H_z^{(k)}(x)$ are defined in \eqref{skzn} and \eqref{defn_H_z^k} respectively. Our main aim is to simplify the integral 
\begin{align} \label{def I}
I_{z,w}^{(k)}(n):= \int_{0}^\infty H_z^{(k)}\left( \a y^{1/k} \right)y^{\frac{1+z}{k}-1} \exp(-y w) dy,
\end{align}
where $\a =(2\pi)^{1+ 1/k} n^{1/k}$. Assume first $\frac{k-1}{2}\leq\textup{Re}(z)<k-\frac{1}{2}$. Now write $H_z^{(k)}\left(\a y^{1/k}\right)$ as an integral using \eqref{integral_representation_H_z^(k)} and interchange the order of integration using Fubini's theorem (justified because of the presence of $e^{-y w}$) to get, for $ \max\{ 0,  1 -k +  \Re(z) \} < c=\textup{Re}(s)\leq\frac{1+\textup{Re}(z)}{k+1} $, 
\begin{align*}
	I_{z,w}^{(k)}(n)&=\frac{1}{2\pi ik}\int_{(c)} \Gamma(s) \cos\left( \frac{\pi s}{2}  \right)  \Gamma\left( \frac{s-1-z}{k} +1 \right)  \cos\left( \frac{\pi}{2}  \left(  \frac{s-1-z}{k} +1  \right)  \right) \a^{-s}\nonumber\\
	&\qquad\qquad\quad\times\int_{0}^{\infty}y^{\frac{1+z-s}{k}-1}e^{-yw}\, dy\, ds\nonumber\\
	&=\frac{1}{2\pi ik}w^{-\frac{(1+z)}{k}}\int_{(c)} \Gamma(s) \cos\left( \frac{\pi s}{2}  \right)  \Gamma\left( \frac{s-1-z}{k} +1 \right)  \cos\left( \frac{\pi}{2}  \left(  \frac{s-1-z}{k} +1  \right)  \right)\nonumber\\
	&\qquad\qquad\qquad\qquad\times\Gamma\left(\frac{1+z-s}{k}\right)\left(\a w^{-1/k}\right)^{-s}\, ds.
\end{align*}
Using the reflection formula 
\begin{align*}
	\Gamma\left(\frac{1+z-s}{k}\right)\Gamma\left( \frac{s-1-z}{k} +1 \right) =\frac{\pi}{\sin\left(\pi\left(\frac{1+z-s}{k}\right)\right)},
\end{align*}
and then simplifying, we arrive at
\begin{align*}
	I_{z,w}^{(k)}(n)=\frac{1}{4 ik}w^{-\frac{(1+z)}{k}}\int_{(c)}\frac{\Gamma(s) \cos\left( \frac{\pi s}{2}  \right)}{\cos\left(\pi\left(\frac{1+z-s}{2k}\right)\right)}\left(\a w^{-1/k}\right)^{-s}\, ds.
\end{align*}
Note that our assumption $\frac{k-1}{2}\leq\textup{Re}(z)<k-1/2$ ensures that the line Re$(s)=1/2$ lies inside the strip $ \max\{ 0,  1 -k +  \Re(z) \} < \textup{Re}(s)\leq\frac{1+\textup{Re}(z)}{k+1} $. Hence we now employ Theorem \ref{parseval_extension_tuan} with $G(s)=\Gamma(s) \cos\left( \frac{\pi s}{2}  \right)$ (while keeping in mind Remark \ref{parsevalcos} and \eqref{Mellin_cos}) and $F(s)=\frac{\pi}{\cos\left(\pi\left(\frac{1+z-s}{2k}\right)\right)}$, and use Lemma \ref{Mellin} so as to get, for $\frac{k-1}{2}\leq\textup{Re}(z)<k-1/2$,
\begin{align}\label{lzw}
	I_{z,w}^{(k)}(n)&=w^{-\frac{(1+z)}{k}}\int_{0}^{\infty}\frac{t^{k+z}}{t^{2k}+1}\cos\left(\a w^{-1/k}t\right)\, dt=\frac{\a^{k-z-1}}{w} \int_{0}^\infty \frac{x^{k+z} \cos(x) }{x^{2k}+ \frac{\a^{2k}}{w^2}}\, dx,
\end{align}
where, in the last step, we employed the change of variable $x=\a w^{-1/k}t$.
Now observe using Theorems \ref{asym zero}  and \eqref{asym of H_z^k} that the extreme left-hand side of \eqref{lzw} is analytic in Re$(z)>-1$ whereas the extreme right-hand side is analytic in Re$(z)<k$. Hence by analytic continuation, \eqref{lzw} holds for $-1<\textup{Re}(z)<k$.

We now find an explicit evaluation of the integral on the extreme-right hand side of \eqref{lzw} for any complex $z$ with $-1< \Re(z) <k$.  For simplicity, let $a=\a/ w^{1/k}$.

First, consider the case  $k \geq 2$ even and let $-1<\textup{Re}(z)<2$. Employing the partial fraction decomposition \eqref{k even simplified} and using the definition of $A_j$ from Lemma \ref{partial fraction_simplified}, we have
\begin{align}\label{valid in larger region}
\int_{0}^\infty \frac{x^{k+z} \cos(x)}{x^{2k}+ a^{2k}}\, dx&=  \frac{(-1)^{\frac{k}{2}-1} a^{2-k} }{k}  \sum_{j=1}^{\frac{k}{2}} \Bigg[ A_j \int_{0}^\infty \frac{x^z \cos(x)\, dx}{x^2+ \left(a \zeta_{4k}^{2j-1}\right)^2} + \overline{A}_j \int_{0}^\infty \frac{x^z \cos(x)\, dx}{x^2+ \left(a \zeta_{4k}^{-(2j-1)}\right)^2} \Bigg] \nonumber\\
&= \frac{(-1)^{\frac{k}{2}-1} a^{2-k} }{k}  \sum_{j=1}^{\frac{k}{2}} \Bigg[ A_j  B(z,a \zeta_{4k}^{2j-1}) + \overline{A}_j B(z,a \zeta_{4k}^{-(2j-1)}) \Bigg].
\end{align}
By analytic continuation (see Remark \ref{analytic continuation of B(z,b)}), \eqref{valid in larger region} holds in $-1<\textup{Re}(z)<k$, where $B(z,b)$ is given in \eqref{simplified form B(z, b)}.
From \eqref{exact expr wig}, \eqref{def I}, \eqref{lzw} and \eqref{valid in larger region}, we arrive at \eqref{eeven} for $w>0$. Since both sides of \eqref{eeven} are analytic in Re$(w)>0$, 
the result holds for Re$(w)>0$.

We now turn to the case $k \geq 1$ odd. Assume initially $-1<\textup{Re}(z)<1$. We use the partial fraction decomposition \eqref{k odd simplified} thereby obtaining for $-1<\textup{Re}(z)<1$,
\begin{align*}
\int_{0}^\infty \frac{t^{k+z} \cos t }{t^{2k}+ a^{2k}} ~dt  &=   \frac{(-1)^{\frac{k-1}{2}} a^{1-k} }{k}  \Bigg[ \int_{0}^\infty \frac{t^{z+1} \cos t ~dt}{t^2+ a^2 } 
 + \sum_{j=1}^{\frac{k-1}{2}} \Bigg\{ B_j \int_{0}^\infty \frac{t^{z+1} \cos t}{t^2+ \left(a \zeta_{4k}^{2j}\right)^2} ~ dt \\ 
 &\qquad\qquad\qquad\quad+ \overline{B}_j \int_{0}^\infty \frac{t^{z+1} \cos t}{t^2+ \left(a \zeta_{4k}^{-2j}\right)^2} ~ dt \Bigg\} \Bigg].
\end{align*}
Proceeding along  similar lines as in the case when $k$ was even, we can see upon using \eqref{simplified form B(z, b)}, we have, for $ -1 < \Re(z) < k$,
\begin{align}\label{I-k odd}
\int_{0}^\infty \frac{t^{k+z} \cos t }{t^{2k}+ a^{2k}} ~dt 
&=  \frac{(-1)^{\frac{k-1}{2}} a^{1-k} }{k}  \Bigg[ B(z+1, a)
 + \sum_{j=1}^{\frac{k-1}{2}}  \Big[ B_j B\left(z+1,a \zeta_{4k}^{2j}\right)
  + \overline{B}_j B\left(z+1,a \zeta_{4k}^{-2j}\right) \Big] \Bigg].  
\end{align}
From \eqref{exact expr wig}, \eqref{def I}, \eqref{lzw} and \eqref{I-k odd}, we arrive at \eqref{oodd} for $w>0$. By analytic continuation, the identity holds for Re$(w)>0$.
\qed


\subsubsection{Proofs of Corollaries \ref{Wigert_gen_ for z=2m} and \ref{Wigert_gen_ for z=2m-1} }
Let $z=2m$ with $0\leq m < k/2$ in Theorem \ref{thm: gen of wigert} to obtain
\begin{align} \label{main_form}
\sum_{n=1}^\infty \sigma_{2m}^{(k)}(n) e^{-nw} &= - \frac{\zeta(-2m)}{2} + \frac{\zeta(k-2m)}{w}  + \frac{1}{k}\frac{\Gamma\left( \frac{1+2m}{k} \right) \zeta\left( \frac{1+2m}{k} \right)}{w^{(1+2m)/k}} + P_{2m}^{(k)}(w),
\end{align}	
where, with $a=\a/w^{1/k}=2\pi(2\pi n/w)^{1/k}$,
	\begin{align*}
	P_{2m}^{(k)}(w)&:= \frac{ (-1)^{\frac{k}{2}-1} (2 \pi)^{ 2+\frac{2}{k}-2m  }    }{\pi^2 k w^{2/k} }   \sum_{n=1}^\infty S_{2m}^{(k)}(n)  n^{ \frac{1-2m}{k} }  \sum_{j=1}^{\frac{k}{2}} \Bigg[ A_j  B\left(2m,a \zeta_{4k}^{2j-1}\right) + \overline{A}_j B\left(2m,a \zeta_{4k}^{-(2j-1)}\right) \Bigg]\\
	&=\frac{(-1)^{\frac{k}{2}+m-1}}{k} \left(\frac{2\pi}{w} \right)^{\frac{1+2m}{k}} \sum_{n=1}^\infty S_{2m}^{(k)}(n)  \sum_{j=1}^{k/2}   \Bigg[\zeta_{4k}^{(1-k+2m)(2j-1)}  \exp\left( -2\pi\left(\frac{2\pi n}{w}\right)^{\frac{1}{k}} \zeta_{4k}^{(2j-1)}\right)  \nonumber  \\
	&\qquad\qquad\qquad\qquad\qquad\qquad+ \zeta_{4k}^{-(1-k+2m)(2j-1)}  \exp\left( -2\pi\left(\frac{2\pi n}{w}\right)^{\frac{1}{k}} \zeta_{4k}^{-(2j-1)}\right)  \Bigg],
\end{align*}
where we used Lemma \ref{B(z,b) for z even} in the last step. Corollary \ref{Wigert_gen_ for z=2m}  now follows by substituting the above expression of $P_{2m}^{(k)}(w)$ in \eqref{main_form} and defining $\overline{L}_{k,z}(w):= \sum_{n=1}^{\infty} S_{z}^{(k)}(n) \exp(- n^{1/k} w )$.

The proof of 
Corollary \ref{Wigert_gen_ for z=2m-1} is similar to that of Corollary \ref{Wigert_gen_ for z=2m} and is hence omitted.
%


\section{Concluding remarks}\label{cr}
The focus of this paper was on obtaining the Vorono\"{\dotlessi} summation formula associated with the function $\sigma_{z}^{(k)}(n)$. Two versions were achieved. The first one was in Theorem \ref{Voronoi_Sigma_z} for the finite sum $\sum_{\alpha<n<\beta}\sigma_{z}^{(k)}(n)f(n)$, where $f$ is analytic, and another, in Theorem \ref{Gen_Voronoi_Sigmakz}, for the infinite series $\sum_{n=1}^{\infty}\sigma_{z}^{(k)}(n)f(n)$, where $f$ is a function from the Schwartz class. It might be interesting to find appropriate conditions on the non-analytic functions $f$ for which Theorem \ref{Voronoi_Sigma_z} is still valid, and the non-Schwartz functions $f$ for which Theorem \ref{Gen_Voronoi_Sigmakz} still holds. 

A considerable part of the paper was devoted to obtaining properties of the functions $H_{z}^{(k)}(x)$ and $K_{z}^{(k)}(x)$  defined in \eqref{defn_H_z^k} and \eqref{line integral_K_z} respectively. The proof of the crucial relation between them which was established in Theorem \ref{H_z_sum of two_K_z} necessitated the use of the theory of linear differential equations and required properties of elementary symmetric polynomials and the Stirling numbers of the second kind. 

The fact that the differential equation in \eqref{geqn} simplifies considerably to Theorem \ref{hade_eqn} could be observed only because of the need to prove Theorem \ref{H_z_sum of two_K_z}. Otherwise, it is difficult to a priori conceive why the result in Lemma \ref{coeff_de} should exist. The latter was the main reason for the aforementioned simplification.
 
It would also be interesting to see solutions of the differential equation in Theorem \ref{hade} other than $H_z^{(k)}(x)$ and the integral in Remark \ref{other solution}. This is particularly important in light of \eqref{Hardy} or, more generally, \eqref{k=1 eqn}. Differential equations analogous to the one in Theorem \ref{hade} have played an important role in number theory and special functions. We note two studies in this regard. The first one is by Wigert \cite{wigannalen} and is concerned with our work in a forthcoming paper \cite{dmsv}. The other is by Everitt \cite[Equation (2.6)]{everitt} who considered a differential equation having as one of its solutions a generalization of the Bessel function of the first kind denoted by $J_{\nu, k}(x)$.
 
Since \eqref{H_z in terms K_z} holds, a worthwhile thing to do would be to show that for $0\leq j\leq 2k+1$, 
\begin{align}\label{7.1}
	&\left.\frac{\pi}{\sqrt{k}2^{\frac{1+z}{k}}}\frac{d^j}{dx^j}G_{0,2k+2}^{\,k+1,0} \!\left(  \,\begin{matrix} \{  \} \\ b_1, \cdots , b_{2k+2} \end{matrix} \; \Big| \frac{1}{4} \left(\frac{x}{2k} \right)^{2k}   \right)\right|_{x=0}\nonumber\\
	&=  \begin{cases}
		\frac{1}{k}(-1)^{\frac{j}{2}}\Gamma\left(\frac{k-j-1-z}{k}\right)\cos\left(\frac{\pi}{2}\left(\frac{k-j-1-z}{k}\right)\right),\hspace{1mm}\textup{if}\hspace{1mm}j\hspace{1mm}\textup{is even},\\
		0,\hspace{1mm}\textup{if}\hspace{1mm}j\hspace{1mm}\textup{is odd},
	\end{cases}
\end{align}
where $b_i$ are defined in \eqref{coefficients_bj}. We enlist some steps which may aid in proving \eqref{7.1}.

Let  $\xi= \frac{1}{4} \left(\frac{x}{2k} \right)^{2k} $ and let $G(\xi):=G_{0,2k+2}^{\,k+1,0} \!\left(  \,\begin{matrix} \{  \} \\ b_1, \cdots , b_{2k+2} \end{matrix} \; \Big| \xi   \right)$. From \cite[p.~157]{comtet},
\begin{align*}
x^\ell\frac{d^\ell w}{dx^\ell}	=\sum_{m=1}^{\ell}s(\ell, m)\left(x\frac{d}{dx}\right)^\ell(w),
\end{align*}
where $s(\ell, m)$ are the Stirling numbers of the first kind. Hence
\begin{align}\label{7.3}
x^j\frac{d^j}{dx^j}G(\xi)&=\sum_{m=1}^{j}s(j, m)\left(x\frac{d}{dx}\right)^mG(\xi) =\sum_{m=1}^{j}s(j, m)(2k)^m\left(\x\frac{d}{d\x}\right)^mG(\xi)\nonumber\\
&=\sum_{m=1}^{j}s(j, m)(2k)^m\sum_{n=1}^{m}S(m, n)\xi^n\frac{d^n}{d\xi^n}G(\xi),
\end{align}
where, in the last step, we used \eqref{s2}.  Employing the result \cite[p.~621, Formula (38)]{prud3} (note that $b_1=0$)
\begin{align*}
	\xi^n\frac{d^n}{d\xi^n}G(\xi)=(-1)^nG_{0,2k+2}^{\,k+1,0} \!\left(  \,\begin{matrix} \{  \} \\ n, \cdots , b_{2k+2} \end{matrix} \; \Big| \xi   \right)
\end{align*}
in \eqref{7.3}, we get
\begin{align}
	\frac{d^j}{dx^j}G(\xi)=\frac{1}{x^j}\sum_{m=1}^{j}s(j, m)(2k)^m\sum_{n=1}^{m}S(m, n)(-1)^nG_{0,2k+2}^{\,k+1,0} \!\left(  \,\begin{matrix} \{  \} \\ n, \cdots , b_{2k+2} \end{matrix} \; \Big|\frac{1}{4} \left(\frac{x}{2k} \right)^{2k}  \right).
	\end{align}
While this suggests an application of L'Hopital's rule as the next step towards obtaining \eqref{7.1}, we are unable to obtain \eqref{7.1} this way.

An interesting thing we observed now deserves a mention. Note that proving \eqref{H_z in terms K_z} is equivalent to proving \eqref{integral_representation_H_z^(k)} in view of \eqref{line integral_K_z}.  If one formally applies Parseval's formula \eqref{Parseval1} to the right-hand side of \eqref{integral_representation_H_z^(k)}, it \emph{still} evaluates to $H_z^{(k)}(x)$!  This suggests that perhaps there exists a grand generalization of Parseval's formula which encompasses Vu Kim Tuan's extension given in Theorem \ref{parseval_extension_tuan} to accommodate the case where both functions $f$ and $g$ are highly oscillatory and neither one has its Mellin transform absolutely integrable on $[0, \infty)$.

Finally, it may be important to study the integral 
\begin{equation*}
	\int_{0}^{\infty}t^{z+\frac{1-3k}{2}}J_{\mu}(xt)J_{\nu}\left(\frac{1}{t^k}\right)\, dt,
\end{equation*}
where $J_{\nu}(\xi)$ is the Bessel function of the first kind defined in \eqref{sumbesselj}. Indeed, for $\mu=\nu=-1/2$, it reduces (except for a constant in front) to $H_z^{(k)}(x)$ in view of the relation  $J_{-1/2}(x)=\sqrt{\frac{2}{\pi x}}\cos(x)$.  Similarly, for $\mu=\nu=1/2$, it essentially reduces to the integral in \eqref{hkzx-analogue}
since $J_{1/2}(x)=\sqrt{\frac{2}{\pi x}}\sin(x)$. For $k=1$, this integral was studied by Hanumanta Rao \cite{cvhrao} (see also \cite[p.~437]{watson}) and appears in his work on self-reciprocal functions \cite{watsonself} and which has led to a large amount of research; see the survey on p.~5 of \cite{dkk1}.

\section*{Acknowledgements}
The first and foremost thanks go to Professor Semyon Yakubovich for informing us the extension of Parseval's formula in Theorem \ref{parseval_extension_tuan}. The authors  sincerely thank
Professor Dmitrii Karp for helpful discussions on Meijer $G$-function, and Professor Vinay Kumar Gupta for writing the \emph{Mathematica} code for the functions $\sigma_z^{(k)}(n)$ and $S_{z}^{(k)}(n)$ Part of this work was done when the second author was a postdoctoral fellow at IIT Gandhinagar.  He is grateful to the institute for the support.  The first author's research was supported by SERB ECR Grant ECR/2015/000070 and by Swarnajayanti Fellowship grant SB/SJF/2021-22/08 of SERB (Govt. of India) while the second author's research was supported by the MATRICS grant MTR/2022/000545. The third author was partially supported by the SERB-DST grant ECR/2018/001566
and the DST INSPIRE Faculty Award Program  DST/INSPIRE/Faculty/ Batch-13/2018. The first and the third authors were partially supported by the MHRD SPARC project SPARC/2018 -2019/P567/SL. The authors sincerely thank the respective funding agencies for their support. 


\section*{Appendix}
\begin{center}
\vspace{1mm}
\textbf{A representation of a Meijer $G$-function in terms of the Hardy-Koshliakov integral }\\
\vspace{1mm}
By Shashank Chorge, Atul Dixit and Aviral Srivastava
\end{center}
 
 This appendix is devoted to proving the second part of Theorem \ref{H_z_sum of two_K_z}, namely, the equality between the Hardy-Koshliakov integral $H_{z}^{(k)}(x)$ defined in \eqref{defn_H_z^k} and the Meijer $G$-function  $\frac{\pi}{\sqrt{k}2^{\frac{1+z}{k}}}G_{0,2k+2}^{\,k+1,0} \!\left(  \,\begin{matrix} \{  \} \\ b_1, \cdots , b_{2k+2} \end{matrix} \; \Big| \frac{1}{4} \left(\frac{x}{2k} \right)^{2k}   \right)$, where the $b_j$ are defined in \eqref{coefficients_bj}. 
 
 Observe that in Section 4, we have already shown that the differential equations satisfied by these two functions are the same. But we are unable to apply the uniqueness theorem for linear differential equations \cite[p.~21, Section 6]{coddington-levinson} owing to the difficulty in verifying the initial conditions. 
 
We circumvent the issue in this appendix. Our method is a grand generalization of that of Hardy \cite{hardyde}. Indeed, as we shall see below, there are several new hurdles that one has to deal while proceeding with a general $k$. (Hardy's method was for $k=1$.) 

The process involved is explained in a nutshell as follows. The general solution of the differential equation of the Meijer $G$-function $G_{0,2k+2}^{\,k+1,0} \!\left(  \,\begin{matrix} \{  \} \\ b_1, \cdots , b_{2k+2} \end{matrix} \; \Big| \frac{1}{4} \left(\frac{x}{2k} \right)^{2k}   \right)$ (which is also the differential equation of $H_{z}^{(k)}(x)$) is \cite[p.~621]{prud3}
 \begin{equation*}
 	\sum_{j=1}^{2k+2}c_je^{\pi ikb_j}G_{0,2k+2}^{\,1,0} \!\left(  \,\begin{matrix} \{  \} \\ b_j,b_1,\cdots,b_{j-1}, b_{j+1}, \cdots , b_{2k+2} \end{matrix} \; \Big| \frac{(-1)^k}{4} \left(\frac{x}{2k} \right)^{2k}   \right),
 \end{equation*}
 where $c_j$ are the constants to be determined. We use Slater's theorem \eqref{Slater} to express $G_{0,2k+2}^{\,1,0}$ in terms of the hypergeometric function ${}_0F_{2k+1}$, and use the analytical properties of $H_z^{(k)}(x)$ to show that half of the $2k+2$ coefficients $c_j$ are zero. The resulting expression will then turn out to be exactly the right-hand side of another instance of Slater's theorem whose left-hand side is the Meijer $G$-function $G_{0,2k+2}^{\,k+1,0} \!\left(  \,\begin{matrix} \{  \} \\ b_1, \cdots , b_{2k+2} \end{matrix} \; \Big| \frac{1}{4} \left(\frac{x}{2k} \right)^{2k}   \right)$ under consideration. This will complete the proof.
 
We will prove the equality between the extreme sides of Theorem \ref{H_z_sum of two_K_z} for\footnote{The result is already proved for $x=0$ in the proof of Theorem \ref{H_z_sum of two_K_z}.} $x>0$ in the form
\begin{align}\label{form to be proved in s}
	J_k(x,s) = \frac{\pi}{\sqrt{k}2^{\frac{s+k-1}{k}}}G_{0,2k+2}^{\,k+1,0} \!\left(  \,\begin{matrix} \{  \} \\ b_1, \cdots , b_{2k+2} \end{matrix} \; \Big| \frac{1}{4} \left(\frac{x}{2k} \right)^{2k}   \right),
\end{align}
where $s=z-k+2$, the integral $J_k$ is defined in \eqref{integral jk}, $J_{k}(x,  z-k+2) = H_{z}^{(k)}(x)$, and the $b_j$'s defined in \eqref{coefficients_bj} are reformulated as follows:
	\begin{align}\label{coefficients_bj1}
	b_j = \begin{cases}
		\frac{j-1}{k},  & {\rm if} \,\,    1 \leq j \leq k,   \\
		\frac{1-s}{2k},  & {\rm if} \,\,  j=k+1,  \\
		\frac{4k+3-2j}{2k},  & {\rm if} \,\,  k+2 \leq j \leq 2k+1,  \\
		\frac{k+1-s}{2k},  & {\rm if} \,\,  j=2k+2.
	\end{cases}
\end{align} 
Recall from the proof of Theorem \ref{hade} that the integrals $J_k$ and $I_k$ converge in $1-k<\textup{Re}(s)<2$ and $-k<\textup{Re}(s)<k+2$ respectively. Similarly, one can show that the integral
 \begin{equation}\label{integral ok}
 \frac{d}{dx}I_k(x,s-1)= \int_{0}^{\infty} \sin(t^k)\cos\left(\frac{x}{t}\right)\frac{dt}{t^{s}}
 \end{equation}
 converges in  $1-k<\textup{Re}(s)<k+2$.
 
We first represent $I_k(x, s)$ as a Meijer $G$-function and then use it to prove \eqref{form to be proved in s}. 
\begin{theorem}\label{I_meijerG thm}
	For $x\geq0$ and $-k<\textup{Re}(s)<k+2$,
\begin{align}\label{I_meijerG}
	I_{k}(x,s) =  \frac{\pi}{\sqrt{k}2^{\frac{s+k-1}{k}}}G_{0,2k+2}^{\,k+1,0} \!\left(  \,\begin{matrix} \{  \} \\ b_1', \cdots , b_{2k+2}' \end{matrix} \; \Big| \frac{1}{4} \left(\frac{x}{2k} \right)^{2k}   \right),
\end{align}
where \begin{align}
	b_{j}' =
	\begin{cases}
		b_{j+k+1}, & 1 \leq j \leq k+1, \\
		b_{j-k-1}, & k+2 \leq j \leq 2k+2,
	\end{cases} 
\end{align} 
with $b_j$ being defined in \eqref{coefficients_bj1}.
\end{theorem}

Let us  begin with some lemmas. 
 
  \begin{lemma}\label{Uni_conv}
  	Let $0<s<1$. For $1\leq j \leq k$, $\frac{d^j}{dx^{j}}I_k(x, s)$
 	is uniformly convergent in $x$ on any compact subset of $ \mathbb{R}.$ Moreover, $\frac{d^{k+1}}{dx^{k+1}}I_k(x, s)$ is uniformly convergent in $x$ on any compact subset of $ \mathbb{R}^{+}.$  
 \end{lemma}
  \begin{proof}
Depending on the parity of $j$, we have to show that
\begin{align}\label{two integrals}
\frac{d^j}{dx^{j}}I_k(x, s)&=\int_{0}^{\infty}\frac{\partial^j}{\partial x^{j}}\sin(t^k)\sin\left(\frac{x}{t}\right)\frac{dt}{t^{s}}\nonumber\\
&= (-1)^{\frac{j-1}{2}}\int_{0}^{\infty} \sin(t^k)\cos\left(\frac{x}{t}\right)\frac{dt}{t^{s+j}}\hspace{3mm}\text{or}\hspace{3mm}(-1)^{\frac{j}{2}}\int_{0}^{\infty} \sin(t^k)\sin\left(\frac{x}{t}\right)\frac{dt}{t^{s+j}}	.
\end{align}
We now show that the integral $\int_{0}^{\infty} \sin(t^k)\cos\left(\frac{x}{t}\right)\frac{dt}{t^{s+j}}$ is uniformly convergent in the variable $x$ in $(-\infty, \infty)$. The uniform convergence of the other integral can be similarly proved. 

Let $\epsilon>0$ be given. For $j<k+1-s$, choosing $r_0 = \left(\e(k-s-j+1)\right)^{\frac{1}{k-s-j+1}}$, we see that for any $0<r<r_0$,
$$ \bigg| \int_{0}^{r} \sin(t^k)\cos\bigg(\frac{x}{t}\bigg)\frac{dt}{t^{s+j}} \bigg| \leq \bigg|\int_{0}^{r} \cos\bigg(\frac{x}{t}\bigg)\frac{dt}{t^{s+j-k}}\bigg| \leq \int_{0}^{r}\frac{dt}{t^{s+j-k}}< \frac{(r_0)^{k-s-j+1}}{k-s-j+1}= \e, $$
 which guarantees the uniform convergence around $t=0$. Moreover, if $y_0 = \left(\e(s+j-1)\right)^{\frac{1}{1-s-j}}$, then for any $ y_0< y $,
  \begin{equation}\label{t infty}
   \bigg| \int_{y}^{\infty} \sin(t^k)\cos\bigg(\frac{x}{t}\bigg)\frac{dt}{t^{s+j}} \bigg| \leq \int_{y}^{\infty}\frac{dt}{t^{s+j}}< \frac{(y_0)^{1-s-j}}{s+j-1}= \e,
   \end{equation}
  thereby securing uniform convergence around $t=\infty$. Then, using \cite[p.~437]{bromwich}, it is easy to see $ I_{k}^{(j)}(x,s)$ is uniformly convergent on any compact subset of $ \mathbb{R}$ for $1\leq j\leq k$. 
  
  Next, we show that the second integral on the right-hand side of \eqref{two integrals}, namely,\newline  $\int_{0}^{\infty} \sin(t^k)\sin\left(\frac{x}{t}\right)\frac{dt}{t^{s+j}}$, is uniformly convergent on $\mathbb{R}^{+}$ for $j=k+1$. The proof of the uniform convergence of the first integral is analogous and hence it is omitted.  First note that for a fixed $ \epsilon >0$ and $0<s<1$, there exists $ M_0>0$ such that $ \int_{M}^{\infty} \sin\left(T\right)T^{s-1}dT < \epsilon$ for all $ M > M_0 $. Next, substitute $ t= x/T $ to have 
  \begin{align}
  	\int_{0}^{\infty} \sin(t^k)\sin\left(\frac{x}{t}\right)\frac{dt}{t^{s+k+1}}=\int_{0}^{\infty} \sin\left(\frac{x^k}{T^k}\right)\sin\left(T\right)\frac{T^{s+k-1}}{x^{s+k}}dT.
  \end{align}
  Now let $ 0<x_0\leq x\leq x_1 $. Fix $ \epsilon >0 $ and let $ \epsilon' = x_0^s \epsilon $. Then there exists $ M' > 0 $ such that
  $ \int_{M'}^{\infty} \sin\left(T\right)T^{s-1}dT < \epsilon'$ so that
  \begin{align}
  	\Bigg|\int_{M'}^{\infty} \sin\left(\frac{x^k}{T^k}\right)\sin\left(T\right)\frac{T^{s+k-1}}{x^{s+k}}dT\Bigg|< 	\frac{1}{x^{s}}\Bigg|\int_{M'}^{\infty} \sin\left(T\right)T^{s-1}dT\Bigg|< \frac{\epsilon'}{(x_0)^s}=\epsilon.
  \end{align}
  This shows the uniform convergence around $t=0$. One can prove the uniform convergence around $t=\infty$ by giving arguments similar to that used in \eqref{t infty}. Then another application of \cite[p.~437]{bromwich} proves the uniform convergence of $ I_{k}^{(k+1)}(x,s)$ on any compact subset of $ \mathbb{R}^{+}.$ 
 \end{proof}
 
 \begin{lemma}\label{asym_zero}
 	Let $0<s<1$ and let $I_k(x, s)$ be defined in  \eqref{integral ik}. As $ x \rightarrow 0^{+}$,
 	\begin{align}\label{BigO}
 		I_k(x, s) = O(x).
 	\end{align} 
 \end{lemma}
 \begin{proof}
 	Write the integral $I_k$ as sum of two integrals, one from $0$ to $\pi$, and another from $\pi$ to $\infty$. Then applying the second mean value theorem for the second integral, we get 
\begin{align}\label{SMVT}
	\int_{\pi}^{\infty}\!\! \sin(t^k)\sin\bigg(\frac{x}{t}\bigg)\frac{dt}{t^s}	= \sin\bigg(\frac{x}{\pi}\bigg)\int_{\pi}^{\xi}\!\! \sin(t^k)\frac{dt}{t^s} 
\end{align} 
for some $ \pi < \xi < \infty$. This gives
  \begin{align}
 	|I_k(x, s)| < \bigg|\int_{0}^{\pi}\!\! \sin(t^k)\sin\bigg(\frac{x}{t}\bigg)\frac{dt}{t^s}\bigg|+ \bigg|\sin\bigg(\frac{x}{\pi}\bigg)\int_{\pi}^{\xi}\!\! \sin(t^k)\frac{dt}{t^s} \bigg|.
 \end{align}
 As $ x \rightarrow 0^{+}$, we can write
 \begin{align}
 	|I_k(x, s)| < x \bigg|\int_{0}^{\pi}\!\! \sin(t^k)\frac{dt}{t^{s+1}}\bigg|+ \sin\bigg(\frac{x}{\pi}\bigg)\bigg|\int_{\pi}^{\xi}\!\! \sin(t^k)\frac{dt}{t^s} \bigg|< Mx
 \end{align}
 for some $M>0$.
\end{proof}

We are now ready to represent $I_k(x, s)$ in terms of Meijer $G$-function.
\begin{proof}[Theorem  \textup{\ref{I_meijerG thm}}][]
	
	We first prove the result for $0<s<1$. It can then be extended by analytic continuation to $-k<\textup{Re}(s)<k+2$ using \cite[p.~30, Theorem 2.3]{temme}.
	
It is easy to see using Slater's theorem \eqref{Slater} that the identity holds for $x=0$ as both  sides are identically zero. Now assume $x>0$.
	From Remark \ref{other solution}, we know that, with $s=z-k+2$, the differential equation satisfied by $I_k(x, s)$ is the same as that satisfied by $G_{0,2k+2}^{\,k+1,0} \!\left(  \,\begin{matrix} \{  \} \\ b_1, \cdots , b_{2k+2} \end{matrix} \; \Big| \frac{1}{4} \left(\frac{x}{2k} \right)^{2k}   \right)$, namely, \eqref{hade_eqn}. Note that the  $b_j$'s are as defined in \eqref{coefficients_bj}.  That this Meijer $G$-function satisfies \eqref{hade_eqn}  is shown in the proof of Theorem \ref{H_z_sum of two_K_z}.

Using \cite[p.~621]{prud3}, we must have
$$
I_k(x, s) = \sum_{j=1}^{2k+2} c_j e^{\pi ikb_j} G_{0,2k+2}^{1,0} \!\left(  \,\begin{matrix} \{  \} \\ b_j,b_1 \cdots , b_{j-1},b_{j+1}\cdots ,b_{2k+2} \end{matrix} \; \Big| \frac{(-1)^k}{4} \left(\frac{x}{2k} \right)^{2k}   \right)
$$
where $ c_j$'s are constants to be determined. By Slater's theorem \eqref{Slater}, $I_k(x, s)$ equals
\begin{align}\label{iknew}
\sum_{j=1}^{2k+2} \frac{c_j e^{\pi ikb_j}\big(\frac{(-1)^k}{4}(\frac{x}{2k})^{2k}\big)^{b_j}}{\prod_{\ell=1}^{2k+2}\G(1+b_j-b_{\ell})} {}_0F_{2k+1}\bigg(\begin{matrix} - \\ 1+b_j-b_1 \cdots * \cdots 1+b_j-b_{2k+2} \end{matrix} \; \Big| \frac{(-1)^{k-1}}{4}\left(\frac{x}{2k}\right)^{2k} \bigg),
\end{align}
where * denotes that the term $1+b_j-b_j $ is excluded. Now using the series definition of the hypergeometric function and simplifying, we see that
\begin{align}\label{finalI}
	I_k(x, s) =\sum_{j=1}^{2k+2}\sum_{n=0}^{\infty}C_j M(j,n) x^{2k(n+b_j)},
\end{align}
where, with $ (b)_r= b(b+1)(b+2)\cdots(b+r-1)$,
\begin{align}\label{Cj defn}
C_j = \frac{c_j e^{\pi ikb_j}\bigg(\frac{(-1)^k}{4 (2k)^{2k}}\bigg)^{b_j}}{\prod\limits_{\ell=1}^{2k+2}\G(1+b_j-b_{\ell})},\hspace{3mm} M(j,n) = \frac{(-1)^{n(k-1)}}{\prod\limits_{\substack{m=1 \\ m \neq j}}^{2k+2} (1+b_j-b_m)_n 4^{n}  n! (2k)^{2kn}}.
\end{align}
The goal now is to calculate $C_j$'s. Separating the $n=0$ term from each of the $2k+2$ series, noting that $M(j,0)=1, 1\leq j\leq2k+2$, and substituting the expressions of $b_j$ from \eqref{coefficients_bj1}, we see that
\begin{align}\label{expand}
	I_k(x, s) &=C_1+C_2x^2+\cdots+C_{\ell}x^{2(\ell-1)}+\cdots+ C_kx^{2(k-1)}+C_{k+1}x^{1-s}+C_{k+2}x^{2k-1}\nonumber\\
	&\quad+\cdots+C_{k+p}x^{2k-(2p-3)}+\cdots+C_{2k+1}x+C_{2k+2}x^{k+1-s}+\sum_{j=1}^{2k+2}\sum_{n=1}^{\infty}C_j M(j,n) x^{2k(n+b_j)}.
	\end{align}
	We first show that $C_1=C_{k+1}=0$. Using Lemma \ref{asym_zero} and letting $x\to0^{+}$, it is easy to see that
	\begin{align}\label{c1}
		C_1=0. 
	\end{align}
	 Also, dividing both sides of \eqref{expand} by $x$, noting that Lemma \ref{asym_zero} implies $I_k(x, s)/x=O(1)$ as $x\to0^{+}$, it is clear that 	\begin{align}\label{ck+1}
	 	C_{k+1}=0. 
	 \end{align}
	We now turn our attention to finding the coefficients $C_{2},C_{3}, \cdots, C_{\lfloor{\frac{k}{2}+1}\rfloor}$. Note that for any $\ell$ with $ 2 \leq \ell \leq \lfloor{\frac{k}{2}+1}\rfloor $, we have $2(\ell-1)\leq k$, because of which $I_k(x, s)$ can be differentiated under the integral sign $2(\ell-1)$ times with respect to $x$  (by appealing to Lemma \ref{Uni_conv}). Further, taking $\lim_{x\to0^{+}}$  on the resulting two sides, we obtain
	\begin{equation}
	 \lim_{x \rightarrow 0^+}  \int_{0}^{\infty} (-1)^{\ell+1} \sin(t^k)\sin\left(\frac{x}{t}\right)\frac{dt}{t^{s+2(\ell-1)}}=(2(\ell-1))! C_{\ell}.
	\end{equation}
	Again, Lemma \ref{Uni_conv} and \cite[p.~436]{bromwich} allows us to take $\lim_{x\to0^{+}}$  inside the integral, resulting in
		\begin{align}\label{cell}
		C_{\ell}=0\hspace{20mm}(2 \leq \ell \leq \lfloor{\tfrac{k}{2}+1}\rfloor).
	\end{align}

Our next task is to evaluate the coefficients $C_{k+p}$, where $\lceil{\frac{k+3}{2}}\rceil\leq p\leq k+1$. Using arguments similar to those used above, we get 
\begin{align}
 \lim_{x \rightarrow 0^+}  \int_{0}^{\infty} (-1)^{k-p+1} \sin(t^k)\cos\left(\frac{x}{t}\right)\frac{dt}{t^{s+2k-2p+3}}= (2k - (2p-3))! C_{k+p}.
\end{align}
By another application of Lemma \ref{Uni_conv}, for $\lceil{\frac{k+3}{2}}\rceil\leq p\leq k+1$, we get
\begin{align}\label{ck+p}
	C_{k+p} &=\frac{(-1)^{k-p+1}}{(2k-2p+3)!} \int_{0}^{\infty}  \!\! \sin(t^k)\frac{dt}{t^{s+2k-2p+3}} \notag\\
	&=\frac{(-1)^{k-p+1}}{(2k-2p+3)!} \G\bigg(\frac{2p-2k-2-s}{k}\bigg)\sin\bigg(\frac{\pi}{2}\bigg(\frac{2p-2k-2-s}{k}\bigg)\bigg).
\end{align}
It remains to evaluate $C_{\lfloor{\frac{k}{2}+1}\rfloor+1},\cdots, C_{k}, C_{k+2}, \cdots,C_{k+\lceil{\frac{k+3}{2}}\rceil-1}$, and $C_{2k+2}$. We accomplish this when $k$ is odd, and leave the case $k$ even for the reader. The reason why one needs to consider the parity of $k$ at this stage is because the evaluation of $C_{\lfloor{\frac{k}{2}+1}\rfloor+1},\cdots, C_{k}, C_{k+2}, \cdots,\newline C_{k+\lceil{\frac{k+3}{2}}\rceil-1}$ depends on that of $C_{2k+2}$ , and the method for evaluating the latter depends on the parity of $k$.
Note that
\begin{align}\label{limit_C_2k+2}
	C_{2k+2}=\lim_{x \rightarrow 0^+} \frac{x^s I_k^{(k+1)}(x, s)}{(k+1-s)(k-s) \cdots (1-s)} .
\end{align} 
Now we write
\begin{align}\label{breakI1}
	I_k^{(k+1)}(x,s) = (-1)^{\frac{k+1}{2}}	  \int_{0}^{\infty}\!\! \sin\bigg(\frac{1}{u^k}\bigg)\sin(xu)u^{s-1+k}du = (-1)^{\frac{k+1}{2}}	(A(x)+B(x)),
\end{align}
where
\begin{align}\label{Aintegral}
	A(x) = \int_{0}^{M}\!\! \sin\bigg(\frac{1}{u^k}\bigg)\sin(xu)u^{s-1+k}du,\hspace{4mm} 	B(x) = \int_{M}^{\infty}\!\! \sin\bigg(\frac{1}{u^k}\bigg)\sin(xu)u^{s-1+k}du.
\end{align}
We first treat $B(x)$. Expand $ \sin\big(\frac{1}{u^k}\big)$ around $u=\infty$ to obtain
\begin{align}\label{7.25}
B(x)&= \int_{M}^{\infty}\!\! \sin(xu)u^{s-1}du +  \sum_{n=1}^{\infty} \int_{M}^{\infty}\!\! \frac{(-1)^n \sin(xu) u^{s-1-2nk}du}{(2n+1)!}\nonumber\\
&= \frac{1}{x^s} \int_{xM}^{\infty} \sin(T)T^{s-1}dT +  \sum_{n=1}^{\infty} \int_{M}^{\infty}\!\! \frac{(-1)^n \sin(xu) u^{s-1-2nk}du}{(2n+1)!},
\end{align}  
where, in the second step, we made the substitution $T=xu$ in the first expression on the right-hand side. Around $a=0$, we have \cite[p.~188, Formulas \textbf{8.21.8, 8.21.14}]{NIST}
\begin{align}\label{asym_expn_sin}
	\int_{a}^{\infty} \sin(T)T^{s-1}dT = \G(s)\sin\bigg(\frac{\pi s}{2}\bigg)+a^{s}\sum_{n=1}^{\infty} \frac{(-1)^n a^{2n-1}}{(s+2n-1)(2n-1)!}.
\end{align} 
Apply \eqref{asym_expn_sin} with $a=xM$ and substitute the resulting expression in \eqref{7.25} to get
\begin{align}\label{bx expr}
B(x)=x^{-s}\G(s)\sin\left(\frac{\pi s}{2}\right)+M^{s}\sum_{n=1}^{\infty} \frac{(-1)^n (Mx)^{2n-1}}{(s+2n-1)(2n-1)!} + \sum_{n=1}^{\infty} \int_{M}^{\infty}\!\! \frac{(-1)^n \sin(xu) u^{s-1-2nk}du}{(2n+1)!} .
\end{align}
Next, substitute \eqref{bx expr} in \eqref{breakI1}, then multiply both sides by $x^s$, and let $x\to0^{+}$ to obtain
\begin{align}
	\lim_{x\rightarrow0^{+}} x^sI_k^{(k+1)}(x,s)= (-1)^{\frac{k+1}{2}}
	\G(s)\sin\left(\frac{\pi s}{2}\right),
\end{align}	
since the integral $A(x)$ is bounded and the two series involved are uniformly convergent on any compact subset of $(0,1)$. Therefore, from \eqref{limit_C_2k+2}, we have
\begin{align}\label{c 2k+2}
	C_{2k+2} = \frac{(-1)^{\frac{k+1}{2}}
		\G(s)\sin\big(\frac{\pi s}{2}\big)}{(k+1-s)(k-s) \cdots (1-s)} .
\end{align}
The remaining coefficients are now found using \eqref{c 2k+2}. Since $k$ is odd, any coefficient among $C_{\lfloor{\frac{k}{2}+1}\rfloor+1},\cdots, C_{k}$ is of the form $C_{j+\frac{k+1}{2}}$, where $1\leq j\leq\frac{k-1}{2}$, and whose corresponding power of $x$ is of the form $x^{k+2j-1}$. Similarly, any coefficient among $C_{k+2}, \cdots,C_{k+\lceil{\frac{k+3}{2}}\rceil-1}$ is of the form $C_{k+p}$, where $2\leq p\leq \frac{k+1}{2}$, and whose corresponding power of $x$ is of the form $x^{2k-2p+3}$. 

Any coefficient among these two lists has the exponent of its corresponding power of $x$ of the form $k+r+1, r\geq0$. Let us designate such a generic coefficient by $C(r)$. Observe from \eqref{expand} and \eqref{c 2k+2} that
\begin{align}\label{otherC}
	(k+1+r)!C(r) &= \lim_{x \rightarrow 0^+} \big(I_k(x,s)-C_{2k+2}x^{k+1-s}\big)^{(k+1+r)}\nonumber\\
	&=  \lim_{x \rightarrow 0^+}  \big(I_k^{(k+1)}(x,s)-(-1)^{\frac{k+1}{2}}
	\G(s)\sin\left(\frac{\pi s}{2}\big)x^{-s}\right)^{(r)}.
\end{align}
We proceed to evaluate this limit. Recall the definitions of $A(x)$ and $B(x)$ in \eqref{Aintegral}. It is easy to see that
\begin{align}\label{Aparity}
	A^{(r)}(x) = \begin{cases}
		(-1)^{\frac{r}{2}}\int_{0}^{M}\!\! \sin\big(\frac{1}{u^k}\big)\sin(xu)u^{s-1+k+r}du, & \text{$r$ is even,} \\
		(-1)^{\frac{r-1}{2}}\int_{0}^{M}\!\! \sin\big(\frac{1}{u^k}\big)\cos(xu)u^{s-1+k+r}du, & \text{$r$ is odd}.
	\end{cases}	
\end{align}
Now, if $r$ is odd, one can check that
\begin{align}\label{br odd}
B^{(r)}(x)&=	\frac{-1}{x^{s+r}}(s)_r \G(s)\sin\left(\frac{\pi s}{2}\right)+M^{s}\sum_{n=\frac{1+r}{2}}^{\infty} \frac{(-1)^n M^{2n-1}(2n-1)(2n-2)\cdots (2n-r)x^{2n-1-r}}{(s+2n-1)(2n-1)!}\nonumber\\
&\quad+\sum_{n=1}^{\infty} \frac{(-1)^{n}(-1)^{\frac{r-1}{2}}}{(2n+1)!}\int_{M}^{\infty}\!\! \cos(xu)u^{s-2 nk-1+r}du,
\end{align}
whereas, if $r$ is even, then
\begin{align}\label{br even}
	B^{(r)}(x)&=
	\frac{1}{x^{s+r}}(s)_r \G(s)\sin\big(\frac{\pi s}{2}\big)+M^{s}\sum_{n=1+\frac{r}{2}}^{\infty} \frac{(-1)^n M^{2n-1}(2n-1)(2n-2)\cdots (2n-r)x^{2n-1-r}}{(s+2n-1)(2n-1)!} \nonumber\\
	&\quad+
	\sum_{n=1}^{\infty} \frac{(-1)^{n}(-1)^{\frac{r}{2}}}{(2n+1)!}\int_{M}^{\infty}\!\! \sin(xu)u^{s-2 nk-1+r}du.
\end{align}
Therefore, when $r$ is odd, from\eqref{otherC},  \eqref{Aparity} and \eqref{br odd}, we see, upon simplification, that
\begin{align*}
	(k+1+r)!C(r) 
	&=
	(-1)^{\frac{k+1}{2}}
	\Bigg\{ (-1)^{\frac{r-1}{2}}\int_{0}^{M}\!\! \sin\bigg(\frac{1}{u^k}\bigg)u^{s-1+k+r}du \\ 
	&\quad
	+ \sum_{n=0}^{\infty} \frac{(-1)^{n+\frac{r+1}{2}}M^{s-2nk+r}}{(2n+1)!(s-2nk+r)} \Bigg\}.
\end{align*}
Now employ the change of variable $T=u^{-k}$ in the integral on the right-hand side, then use \eqref{asym_expn_sin} with $s$ replaced by $-(s+k+r)/k$, and then observe that the  two infinite series cancel, thereby leading to
\begin{align}\label{C_oddk_oddr}	
	C(r) = \frac{(-1)^{1+\frac{k+r}{2}}}{k(k+r+1)!} \G\bigg(-\frac{(s+k+r)}{k}\bigg)\sin\bigg(\frac{\pi}{2}\bigg(\frac{s+k+r}{k}\bigg)\bigg).	
\end{align}
If $r$ is even, one can similarly proceed as above and conclude that
\begin{align}\label{other c(r)}
C(r)=0.
\end{align}
Therefore, from \eqref{c1}, \eqref{ck+1}, \eqref{cell}, \eqref{ck+p}, \eqref{c 2k+2}, \eqref{C_oddk_oddr}, and \eqref{other c(r)}, we see that when $k$ is odd,
\begin{equation}
	C_{j} = 
	\begin{cases}
		0   , 
		&  1\leq j \leq k+1,  \\
		\frac{(-1)^{j}}{ k(4k-2j+3)!} \G\big(-\frac{2j-4k-2-s}{k}\big)\sin\big(\frac{\pi}{2}\big(\frac{s+2-2j}{k}\big)\big) ,
		& k+2 \leq j \leq 2k+1, \\
		\frac{(-1)^{\frac{k+1}{2}}\G(s)\sin\big(\frac{\pi s}{2}\big)}{(k+1-s)(k-s) \cdots (1-s)}, 
		& j= 2k+2.
	\end{cases}
\end{equation}
Similarly, when $k$ is even, it can be shown that 
\begin{equation}
	C_{j} = 
	\begin{cases}
		0   , 
		&  1\leq j \leq k+1,  \\
		\frac{(-1)^{j}}{ k(4k-2j+3)!} \G\big(-\frac{2j-4k-2-s}{k}\big)\sin\big(\frac{\pi}{2}\big(\frac{s+2-2j}{k}\big)\big) ,
		& k+2 \leq j \leq 2k+1, \\
		\frac{(-1)^{\frac{k}{2}}
			\G(s)\cos\big(\frac{\pi s}{2}\big)}{(k+1-s)(k-s) \cdots (1-s)}, 
		& j= 2k+2.
	\end{cases}
\end{equation}
Therefore, from \eqref{iknew}, \eqref{Cj defn} and the above expressions for $C_j$, we find that
\begin{align}\label{ik almost done}
	I_k &=\sum_{j=k+2}^{2k+2}C_{j}x^{2kb_{j}} {}_0F_{2k+1}\bigg(\begin{matrix} - \\ 1+b_j-b_1 \cdots * \cdots 1+b_j-b_{2k+2} \end{matrix} \; \Bigg| \frac{(-1)^{k-1}}{4}\left(\frac{x}{2k}\right)^{2k} \bigg) \notag 	\\
	&=\sum_{i=1}^{k+1}C_{i+k+1}x^{2kb_{i+k+1}} {}_0F_{2k+1}\bigg(\begin{matrix} - \\ 1+b_{i+k+1}-b_1 \cdots * \cdots 1+b_{i+k+1}-b_{2k+2} \end{matrix} \; \Bigg| \frac{(-1)^{k-1}}{4}\left(\frac{x}{2k}\right)^{2k} \bigg)\notag\\
	&=\sum_{i=1}^{k+1}C_{i}'x^{2kb_{i}'} {}_0F_{2k+1}\bigg(\begin{matrix} - \\ 1+b_{i}'-b_{1}' \cdots * \cdots 1+b_{i}'-b_{2k+2}' \end{matrix} \; \Bigg| \frac{(-1)^{k-1}}{4}\left(\frac{x}{2k}\right)^{2k} \bigg),
\end{align}
where, for $   1 \leq i \leq k+1$, $ C_{i}' = C_{k+i+1}$ and 
\begin{align}
	b_{i}' =
	\begin{cases}
		b_{i+k+1}, & 1 \leq i \leq k+1, \\
		b_{i-k-1}, & k+2 \leq i \leq 2k+2.
	\end{cases} 
\end{align} 
Using \eqref{gmf} and the elementary properties of gamma function, one can check that
\begin{align}\label{ci dash}
	C_{i}' = \frac{\Pi}{\sqrt{k}2^{\frac{s+k-1}{k}}} \frac{\prod_{n=1}^{k+1}\G(b_n'-b_{i}')^{*}}{\prod_{n=k+2}^{2k+2}\G(1+b_{i}'-b_n')} \bigg(\frac{1}{(4(2k)^{2k})}\bigg)^{b_{i}'}.
\end{align}
Finally, from \eqref{ik almost done}, \eqref{ci dash}, and an application of Slater's theorem, we are led to \eqref{I_meijerG}.
	\end{proof}

We are all set to prove the representation of $J_k(x, s)$ as a Meijer $G$-function.

\begin{proof}[\textup{\ref{form to be proved in s}}][]
	We first prove the identity for $0<s<1$ and then extend it by analytic continuation to $1-k<\textup{Re}(s)<2$ by applying \cite[p.~30, Theorem 2.3]{temme}. 
	
	To that end, from \eqref{k+1_even_I} and Theorem \ref{I_meijerG thm}, we have
	\begin{align}
		(-1)^{\frac{k+3}{2}} k J_k(x,s) =  \frac{\pi}{\sqrt{k}2^{\frac{s+k-1}{k}}}\bigg[x \frac{d^{k+1}}{dx^{k+1}}&G_{0,2k+2}^{\,k+1,0} \!\left(  \,\begin{matrix} \{  \} \\ b_1', \cdots , b_{2k+2}' \end{matrix} \; \Big| \frac{1}{4} \left(\frac{x}{2k} \right)^{2k}   \right)
		+\\  &(s+k-1) \frac{d^{k}}{dx^{k}}G_{0,2k+2}^{\,k+1,0} \!\left(  \,\begin{matrix} \{  \} \\ b_1', \cdots , b_{2k+2}' \end{matrix} \; \Big| \frac{1}{4} \left(\frac{x}{2k} \right)^{2k}   \right)\bigg].
	\end{align}
	We now employ \cite[p.~621, Formula (40)]{prud3} and the definition of $\Delta(k, a)$ given in \cite[p.~798]{prud3} to get
	\begin{align*}
		(-1)^{\frac{k+3}{2}} k x^{k} J_k(x,s) &=  \frac{\pi}{\sqrt{k}2^{\frac{s+k-1}{k}}}\Bigg[(2k)^{k+1}G_{2k,4k+2}^{\,k+1,2k} \!\left(  \,\begin{matrix} 0, \frac{1}{2k} , \cdots \frac{2k-1}{2k} \vspace{0.25 cm}\\  b_1', \cdots , b_{2k+2}', \frac{k+1}{2k},\frac{k+2}{2k}  \cdots , \frac{3k}{2k} \end{matrix} \; \Bigg| \frac{1}{4} \left(\frac{x}{2k} \right)^{2k}   \right)\\
		&\quad+(s+k-1)(2k)^{k}G_{2k,4k+2}^{\,k+1,2k} \!\left(  \,\begin{matrix} 0, \frac{1}{2k} , \cdots \frac{2k-1}{2k} \vspace{0.25 cm}\\  b_1', \cdots , b_{2k+2}', \frac{k}{2k},\frac{k+1}{2k}  \cdots , \frac{3k-1}{2k} \end{matrix} \; \Bigg| \frac{1}{4} \left(\frac{x}{2k} \right)^{2k}   \right)\Bigg].
	\end{align*}
	Now use \eqref{MeijerG} to represent both the Meijer $G$-functions occurring in the above equation as line integrals and perform a tedious, though routine, calculation involving Gauss multiplication formula \eqref{gmf} and other basic properties of the Gamma function to arrive at   
	\begin{align}
		\frac{(-1)^{\frac{k+3}{2}}2^{\frac{s+k-1}{k}} k^{3/2} x^{k} J_k(x,s)}{\pi}=\frac{(2k)^{k+1}}{2 \pi i}\hspace{-2mm}\int_{(c)}^{}\hspace{-2mm}
		\frac{\prod_{j=1}^{k}\Gamma(\frac{2j-2+k}{2k}-w)\Gamma(\frac{k+1-s}{2k}-w)}{\Gamma(\frac{s-1}{2k}+w)\prod_{j=1}^{ k}\Gamma(\frac{2j-k-1}{2k}+w)}\left(\tfrac{1}{4}\left(\tfrac{x}{2k}\right)^{2k}\right)^{w}dw,
	\end{align}
	where $-1/(2k)<c=\textup{Re}(w)<1/(2k)$. Now replace $w$ by $w+1/2$ in the integral on the right, and replace $j$ by $j-k-1$ in the product of the denominator of the integrand to obtain
	\begin{align}
		\frac{(-1)^{\frac{k+3}{2}}2^{\frac{s+k-1}{k}} k^{3/2} x^{k} J_k(x,s)}{\pi}=\frac{kx^k}{2 \pi i}\hspace{-1.5mm}\int_{(c')}^{}\hspace{-1mm}
		\frac{\prod_{j=1}^{k}\Gamma(\frac{j-1}{k}-w)\Gamma(\frac{1-s}{2k}-w)}{\Gamma(\frac{k+s-1}{2k}+w)\prod_{j=k+2}^{ 2k+1}\Gamma(\frac{2j-2k-3}{2k}+w)}\left(\tfrac{1}{4}\left(\tfrac{x}{2k}\right)^{2k}\right)^{w}dw,
	\end{align}
	where $-(k+1)/(2k)<c'=\textup{Re}(w)<-(k-1)/(2k)$. 
	The result now follows upon noting that $k$ is odd and using the definition in \eqref{MeijerG}. The case when $k$ is even is analogous and can be verified by the reader.
	\end{proof}

%
%
%
%
%
%
%
%
%
%
%
%
%
%
%
%
%
%
%
%

%
%
%
%

%
%
%
%
%
%
\end{document}